\documentclass[pdflatex,sn-mathphys-num]{sn-jnl}

\usepackage{graphicx}%
\usepackage{multirow}%
\usepackage{amsmath,amssymb,amsfonts}%
\usepackage{amsthm}%
\usepackage{mathrsfs}%
\usepackage[title]{appendix}%
\usepackage{xcolor}%
\usepackage{textcomp}%
\usepackage{manyfoot}%
\usepackage{booktabs}%
\usepackage{algorithm}%
\usepackage{algorithmicx}%
\usepackage{algpseudocode}%
\usepackage{listings}%
\usepackage{tikz-cd}%

\theoremstyle{thmstyleone}%
\newtheorem{theorem}{Theorem}[section]
\newtheorem{proposition}[theorem]{Proposition}
\newtheorem{lemma}[theorem]{Lemma}
\newtheorem{corollary}[theorem]{Corollary}

\theoremstyle{thmstyletwo}%
\newtheorem{example}[theorem]{Example}
\newtheorem{remark}[theorem]{Remark}

\theoremstyle{thmstylethree}%
\newtheorem{definition}[theorem]{Definition}

\begin{document}

\title[Reproducing Kernel Hilbert Spaces for Virtual Persistence Diagrams]{Reproducing Kernel Hilbert Spaces for Virtual Persistence Diagrams}


\author{\fnm{Charles} \sur{Fanning}} \email{cfannin8@students.kennesaw.edu}

\author{\fnm{Mehmet Emin} \sur{Aktas}} \email{maktas1@kennesaw.edu}

\affil{\orgdiv{School of Data Science and Analytics}, \orgname{Kennesaw State University}, \orgaddress{\street{1000 Chastain Road}, \city{Kennesaw, \postcode{30144}, \state{Georgia}, \country{United States of America}}}}


\abstract{A persistence diagram is a finite multiset of birth-death pairs representing the lifetimes of topological features across a filtration. Existing functional and kernel representations of persistence diagrams are typically constructed extrinsically through embeddings into auxiliary spaces. For filtrations with finite indexing sets, the associated virtual persistence diagram group obtained by Grothendieck completion of the persistence diagram monoid is a finitely generated lattice. We define a phase map sending each persistence interval to a circular coordinate and a character map aggregating the phases of intervals in a virtual persistence diagram. We introduce heat damping on characters of virtual persistence diagram groups to suppress the unstable frequencies. We derive Lipschitz bounds for the resulting kernels and apply them in a synthetic segmentation experiment.}






\maketitle

\section{Introduction}\label{sec:intro}

Persistence diagrams give stable summaries of filtered topological data and are widely used in topological data analysis (TDA) \cite{892133,Zomorodian2005ComputingPH,Oudot2015PersistenceT,CohenSteiner2007StabilityPD}. Many applications require kernels, feature maps, and other analytic representations of persistence diagrams, and existing constructions typically obtain these representations through embeddings into linear or function spaces via (positive definite) kernels. Algebraic approaches to persistence show that persistence diagrams also arise from M\"obius inversion of rank-type data, where signed multiplicities occur naturally and the resulting persistence objects extend beyond interval decompositions \cite{Patel_2018,Kim2021GeneralizedPD}. Virtual persistence diagrams encode this algebraic structure by passing from the commutative monoid of persistence diagrams to an abelian group equipped with the translation-invariant metric construction (Equation~\eqref{eq:groth-metric}) on the Grothendieck completion \cite{Bubenik2022VirtualPD}. This raises the natural question of whether harmonic analysis on the resulting abelian group, particularly through the Fourier--Stieltjes representation (Equation~\eqref{eq:fourier-stieltjes}) and Bochner representation (Equation~\eqref{eq:bochner-representation}), can give intrinsic analytic representations of persistence diagrams and their functions.

This paper develops those harmonic-analytic tools in the setting of filtrations with finite indexing sets. Such filtrations arise naturally in many data settings: digital images (see Section~\ref{subsec:image-segmentation}), voxel data, graph filtrations with finitely many edge labels, categorical data, and bounded integer-valued filtrations. For these filtrations, the associated virtual persistence diagram group obtained by the Grothendieck completion from Definition~\ref{def:groth} is a finitely generated lattice identified explicitly with \(\mathbb Z^{X\setminus A}\) in Equation~\eqref{eq:KXA-lattice}. We study the Fourier analysis, heat kernels (as in the paragraph after Definition~\ref{def:fourier-multiplier}), and Lipschitz bounds~(Definition~\ref{def:Lip-seminorm}) induced by the translation-invariant \(W_1\) metric (see Equation~\eqref{eq:Wp}) on this lattice.

\paragraph{Objective.}

The first objective is to lift Wasserstein stability from diagrams to the harmonic side: characters on the virtual persistence diagram group should define an explicit dual (as in Equation~\eqref{eq:character-formula}), and Lipschitz bounds on stable functions should be expressible in terms of differences of phase on virtual persistence diagrams.

The second objective is to construct heat-kernel methods through heat spectral multipliers---Equation~\eqref{eq:heat-spectral-multiplier}---on the dual torus, determined by the graph Laplacian in Equation~\eqref{eq:graph-laplacian} and its symbol~(see Equation~\eqref{eq:fourier-multiplier}). The goal is to use heat damping to suppress the modes that are unstable with respect to the $W_1$ transport metric.

The third objective is to make the resulting kernels usable in computations. We construct random Fourier feature maps that approximate the heat kernels through the sampling law in Equation~\eqref{eq:heat-sampling-density} and feature map (see Equation~\eqref{eq:heat-rff-map}), and apply these approximations as finite-dimensional topological features in a machine learning problem.

\paragraph{Main Results.}


Characters, in the sense of Definition~\ref{def:characters-dual}, on the virtual persistence diagram group $K(X,A)$ determine the phase functions \(\phi_\theta\) from Equation~\eqref{eq:phase-function}. Lemma~\ref{lem:char-lip-comparison} identifies the Lipschitz constant~(Definition~\ref{def:Lip-seminorm}) of $\chi_\theta$ exactly with the Lipschitz constant of $\phi_\theta$.

\begin{lemma}[Lemma~\ref{lem:char-lip-comparison}]
For every $\theta\in\mathbb T^{|X\setminus A|}$,
\[
\mathrm{Lip}(\chi_\theta)
=
\mathrm{Lip}(\phi_\theta).
\]
\end{lemma}

Observe that the phase function $\phi_\theta$ maps persistence intervals to circular coordinates, and that the character $\chi_\theta$ aggregates these phases into a single Fourier mode~(Definition~\ref{def:characters-dual}) on the virtual persistence diagram group. Because a character may aggregate phase contributions from arbitrarily many persistence intervals distributed throughout a virtual persistence diagram, it is not a priori clear that the resulting global Fourier mode should preserve the same Lipschitz stability as the phase map. Lemma~\ref{lem:char-lip-comparison} shows that this does not occur: the Lipschitz stability of the resulting Fourier mode is exactly the intrinsic phase Lipschitz geometry on the persistence space.

The Dirichlet energy \(\lambda(\theta)\) defined in Equation~\eqref{eq:lambda-dirichlet} from the graph Laplacian of Definition~\ref{def:graph-laplacian} measures the instability of the character \(\chi_\theta\) through the intrinsic differences of phases appearing in Equation~\eqref{eq:edgewise-character-lipschitz}. Lemma~\ref{lem:lambda-vs-L} compares the graph representation's Dirichlet energy (Definition~\ref{def:graph-dirichlet}) $\lambda(\theta)$ with the Lipschitz scale of $\chi_\theta$, measured by the seminorm from Definition~\ref{def:Lip-seminorm}.

\begin{lemma}[Lemma~\ref{lem:lambda-vs-L}]
For every $\theta\in\mathbb T^{|X\setminus A|}$,
\[
\frac{4\,w_{\min}\,d_{\min}^2}{\pi^2}\,
\mathrm{Lip}(\chi_\theta)^2
\le
\lambda(\theta)
\le
w_{\max}\,M\,d_{\max}^2\,
\mathrm{Lip}(\chi_\theta)^2.
\]
\end{lemma}

Here $w_{\min},w_{\max}$ denote the extremal edge weights of the graph model, $d_{\min},d_{\max}$ denote the extremal edge lengths in the quotient metric, and $M$ is the number of edges. The quantity $\lambda(\theta)$ is the Dirichlet energy of the phase function $\phi_\theta$: it measures the total weighted phase variation across the graph model of the persistence space. The edge weights may coincide with the transport geometry, but they may also encode a separate diffusion cost, and the constants in the comparison record the interaction between these two structures. The factor $M$ in the upper bound appears because the Dirichlet energy sums phase variation over all graph edges, whereas the Lipschitz bound measures only the largest difference of phases. Consequently, the heat factor $e^{-t\lambda(\theta)}$ suppresses Fourier modes whose circular coordinates are unstable on the graph representation of the geometry of a finite virtual persistence diagram.

The heat measures defined in Equation~\eqref{eq:heat-measure} give weights to characters according to the Dirichlet energy \(\lambda(\theta)\) of their phase functions. Theorem~\ref{thm:heat-lip} shows that the resulting heat-weighted reproducing kernel Hilbert space (RKHS) $\mathcal H_t$ consists of functions whose Lipschitz stability is controlled by the same Fourier modes identified above, with high-energy (possibly unstable) modes suppressed by the heat factor $e^{-t\lambda(\theta)}$.

\begin{theorem}[Spectral form; Theorem~\ref{thm:heat-lip}]
For every $t>0$ and every $f\in\mathcal H_t$,
\[
\mathrm{Lip}(f)
\le
\|f\|_{\mathcal H_t}
\left(
\int_{\mathbb T^{|X\setminus A|}}
\mathrm{Lip}(\chi_\theta)^2\,
e^{-t\lambda(\theta)}
\,d\mu(\theta)
\right)^{1/2}.
\]
\end{theorem}

Functions in the heat RKHS $\mathcal H_t$ are constructed by superposing Fourier modes on the virtual persistence diagram group. A priori, summing many stable characters could still produce an unstable function on persistence diagrams. Theorem~\ref{thm:heat-lip} shows that this does not occur: the Wasserstein Lipschitz stability of the full function is controlled by a heat-weighted average of the characterwise Lipschitz bounds. Since the heat factor $e^{-t\lambda(\theta)}$ exponentially suppresses modes with large Dirichlet energy, functions in $\mathcal H_t$ concentrate on Fourier modes whose circular coordinate representations are stable across the virtual persistence diagram geometry.

Corollary~\ref{cor:heat-lip-geom} rewrites the heat RKHS bound as in Equation~\eqref{eq:heat-rkhs-lip-bound} using only the phase Lipschitz constants on \((X/A,\overline d_1)\). In this form, the heat factor appears as Gaussian decay in \(\mathrm{Lip}(\phi_\theta)^2\).

\begin{corollary}[Geometric form; Corollary~\ref{cor:heat-lip-geom}]
For every $t>0$ and every $f\in\mathcal H_t$,
\[
\mathrm{Lip}(f)
\le
\|f\|_{\mathcal H_t}
\left(
\int_{\mathbb T^{|X\setminus A|}}
\mathrm{Lip}(\phi_\theta)^2
\exp\!\left(
-t\,\frac{4w_{\min}d_{\min}^2}{\pi^2}\,
\mathrm{Lip}(\phi_\theta)^2
\right)
\,d\mu(\theta)
\right)^{1/2}.
\]
\end{corollary}

This formulation expresses the Wasserstein Lipschitz bound for functions in the heat RKHS entirely in terms of the geometry of the phase maps (Equation~\eqref{eq:phase-function}) on $(X/A,\overline d_1)$. The quantity controlling the bound is the Gaussian-weighted average $L^2e^{-tCL^2}$, where $L=\mathrm{Lip}(\phi_\theta)$ and $C=4w_{\min}d_{\min}^2/\pi^2$. Consequently, a function in the heat RKHS is stable when its Fourier representation concentrates on circular coordinate systems whose difference in phases is small with respect to changes in the persistence diagram geometry.

\subsection{Related work}

Generalized persistence diagrams arise from applying M\"obius inversion to
rank-type invariants. Patel develops generalized persistence diagrams for
poset-indexed persistence modules with multiplicities in abelian groups
\cite{patel2018generalized}. Kim and M\'emoli develop generalized rank
invariants and generalized persistence diagrams for poset-indexed modules in
settings where interval decompositions need not exist
\cite{kimmemoli2021generalized}. Betthauser, Bubenik, and Edwards introduce
graded persistence diagrams through M\"obius inversion of graded rank
functions, yielding signed multiplicities in the graded setting
\cite{betthauser2021graded}. In applications, persistent homology has also
been used in image segmentation through topological features and topology-aware loss
functions \cite{9186664,10378018,QAISER20191}.

The metric-pair formalism for persistence diagram monoids and virtual
persistence diagram group provides the case of virtual persistence diagrams with finitely many intervals supporting later
relative measure constructions
\cite{MR4414770,Bubenik2022VirtualPD}. Bubenik and Elchesen develop relative
Borel and Radon measures on metric pairs and construct relative Wasserstein
distances for them \cite{bubenik2025relativeoptimaltransport}. They also prove the 
relative Kantorovich--Rubinstein and Monge--Kantorovich duality theorems. In
particular, relative \(1\)-finite and locally \(1\)-finite Radon measures are
characterized through duality with Lipschitz and compactly supported
Lipschitz spaces. Bubenik and Ross then construct Schauder bases for compactly
supported Lipschitz spaces on polyhedral pairs using nested triangulations~\cite{bubenik2025schauderbasismultiparameterpersistence}.
Evaluating signed multiparameter persistence diagrams on these basis elements
gives stable sequence-space embeddings and a minimality result in the ambient
relative Radon measure setting.

\subsection{Contributions}

Since $K(X,A)$ is a discrete LCA group, Pontryagin duality
(Proposition~\ref{prop:dual}) identifies its dual with the torus
$\widehat{K(X,A)}\cong\mathbb T^{|X\setminus A|}$ through the explicit
character formula (Equation~\eqref{eq:character-formula}) and pairing
(Equation~\eqref{eq:dual-pairing}). Each character $\chi_\theta$ determines a phase map $\phi_\theta$ assigning circular coordinates to persistence intervals, and Lemma~\ref{lem:char-lip-comparison} identifies the Wasserstein Lipschitz seminorm of the resulting Fourier mode exactly with the intrinsic phase Lipschitz seminorm on the persistence diagram geometry. Corollary~\ref{cor:edgewise-char} further reduces these Lipschitz bounds to differences of phase on a finite graph model of $(X/A,\overline d_1)$.

We then introduce graph Dirichlet energies through the weighted graph Laplacian from Definition~\ref{def:graph-laplacian} and Dirichlet form from Definition~\ref{def:graph-dirichlet} on the quotient persistence space. The resulting energy $\lambda(\theta)$ measures total weighted phase variation of the circular coordinates associated to $\chi_\theta$, and Lemma~\ref{lem:lambda-vs-L} compares this energy directly with the Wasserstein Lipschitz scale on characters. This identifies the heat multiplier \(e^{-t\lambda(\theta)}\) from Equation~\eqref{eq:heat-measure} as a mechanism for suppressing Fourier modes whose circular coordinates are unstable across the persistence diagram geometry.

From these heat multipliers, we construct translation-invariant heat kernels and reproducing kernel Hilbert spaces on virtual persistence diagrams. Lemma~\ref{lem:rkhs-lip} and Theorem~\ref{thm:heat-lip} lift the characterwise stability bounds to arbitrary Reproducing Kernel Hilbert Space (RKHS) functions, while Corollaries~\ref{cor:heat-lip-spectral}--\ref{cor:heat-lip-geom} express the resulting Lipschitz estimates in both spectral and geometric forms. In particular, the geometric formulation shows that the Wasserstein stability of functions on the Fourier side can be read directly from the phase geometry of the original persistence diagram space.

To obtain computable finite-dimensional features, we introduce the heat-weighted random Fourier feature map (Equation~\eqref{eq:heat-rff-map}), following the general construction in Definition~\ref{def:rff-general}, by sampling characters from the heat law (Equation~\eqref{eq:heat-sampling-density}) on the dual torus. Lemma~\ref{lem:heat-rff-unbiased} proves that these random features are unbiased approximations of the heat kernel, and Theorem~\ref{thm:heat-rff-spectral} extends the same heat-weighted Lipschitz control to the finite-dimensional feature maps.

Finally, Section~\ref{sec:experiments} applies these constructions to machine learning with virtual persistence diagrams. We define a heat-kernel topological loss based on virtual diagram differences and compare its random Fourier feature approximants against Wasserstein-based topological losses in a synthetic image segmentation experiment.

\subsection{Organization}\label{subsec:intro-organization}

The paper is organized as follows.

Section~\ref{sec:background} reviews persistent homology, Wasserstein stability, virtual persistence diagrams, and the harmonic-analysis and RKHS tools used throughout the paper.

Section~\ref{subsec:LCA} develops the harmonic analysis of the virtual persistence diagram group, introduces phase functions and Fourier modes on persistence diagrams, and proves the phase and Lipschitz bounds supporting the theory.

Section~\ref{sec:stable-multipliers} introduces the graph Dirichlet energy and heat multipliers, and constructs heat kernels that suppress Fourier modes with rapidly oscillating circular coordinates.

Section~\ref{subsec:rkhs-layer} develops the associated RKHSs and proves global Wasserstein Lipschitz bounds for functions generated by the heat kernels.

Section~\ref{subsec:heat-rff} constructs heat-weighted random Fourier features and proves kernel approximation and Lipschitz stability results for the resulting finite-dimensional feature maps.

Section~\ref{sec:experiments} applies these constructions to a machine learning problem and compares the resulting heat-kernel losses with a Wasserstein-based baseline for image segmentation.

Section~\ref{sec:conclusion} concludes with a discussion of the mathematical consequences, limitations, and possible extensions of the theory.

\section{Background and notation}\label{sec:background}

This section collects the topological and analytic tools used in the rest of the paper. On the topological side, we recall the classical one-parameter formulation of persistent homology. We then specialize to the virtual persistence diagram framework of Bubenik and Elchesen (Section~\ref{subsec:VPD}), formulated in the category of metric pairs and Lipschitz maps and equipped with a translation-invariant $1$-Wasserstein metric via the Grothendieck completion of the diagram monoid.

On the analytic side, we recall basic notions from the theory of positive definite kernels and reproducing kernel Hilbert spaces (Section~\ref{subsec:rkhs-toolkit}), together with standard facts from harmonic analysis on discrete locally compact abelian (LCA) groups (Section~\ref{subsec:harmonic-toolkit}). These tools will later be applied to the virtual persistence diagram group $K(X,A)$ to construct heat kernels and phase-based Lipschitz bounds on functions of persistence diagrams.

We will frequently measure the regularity of maps between metric spaces via their Lipschitz seminorm.

\begin{definition}\label{def:Lip-seminorm}
Let $(M,\rho_M)$ and $(N,\rho_N)$ be metric spaces and let
$f:M\to N$. The \emph{Lipschitz seminorm} of $f$ is
\[
  \mathrm{Lip}(f)
  :=
  \sup_{\alpha\neq\beta}
  \frac{\rho_N\bigl(f(\alpha),f(\beta)\bigr)}
       {\rho_M(\alpha,\beta)}
  \in [0,\infty].
\]
\end{definition}

This is the smallest constant $L$ for which
\[
\rho_N\bigl(f(\alpha),f(\beta)\bigr)
\le
L\,\rho_M(\alpha,\beta)
\]
for all $\alpha,\beta\in M$. If $\mathrm{Lip}(f)\le K<\infty$, we say that
$f$ is $K$-Lipschitz.

\paragraph{Historical Context.}

Persistent homology gives a topological summary of filtered data by recording the birth and death of homological features across a filtration \cite{892133,Zomorodian2005ComputingPH,Oudot2015PersistenceT}. Classically, a persistence diagram is a finite multiset of birth--death pairs $(b,d)\in\overline{\mathbb R}_{<}^2$, together with the diagonal $\Delta=\{(x,x):x\in\mathbb R\}$ taken with infinite multiplicity for optimal transport matchings. The bottleneck and Wasserstein distances give stability bounds for these diagrams under perturbations of the underlying filtration \cite{CohenSteiner2007StabilityPD,Oudot2015PersistenceT}.

Persistence diagrams do not always arise from continuous or infinite-valued filtrations. Finite indexing sets occur naturally in digital imaging and voxel data, temporal network filtrations with bounded integral times, and categorical data. In this case, persistence diagrams are supported on a finite metric space, and the associated virtual persistence diagram group is a finitely generated lattice.

Algebraic approaches to persistence extend the classical diagram formalism through M\"obius inversion of rank-type invariants, producing generalized persistence diagrams with signed multiplicities \cite{Patel_2018, Kim2021GeneralizedPD}. Ordinary persistence diagrams form a commutative monoid under pointwise addition, but signed multiplicities require passage to formal differences and hence to a group completion. The virtual persistence diagram framework of Bubenik and Elchesen gives a compatible transport metric on the resulting abelian group \cite{Bubenik2022VirtualPD}.

\subsection{Virtual persistence diagrams}\label{subsec:VPD}

\begin{figure}[ht]
    \centering
    \includegraphics[width=0.75\linewidth]{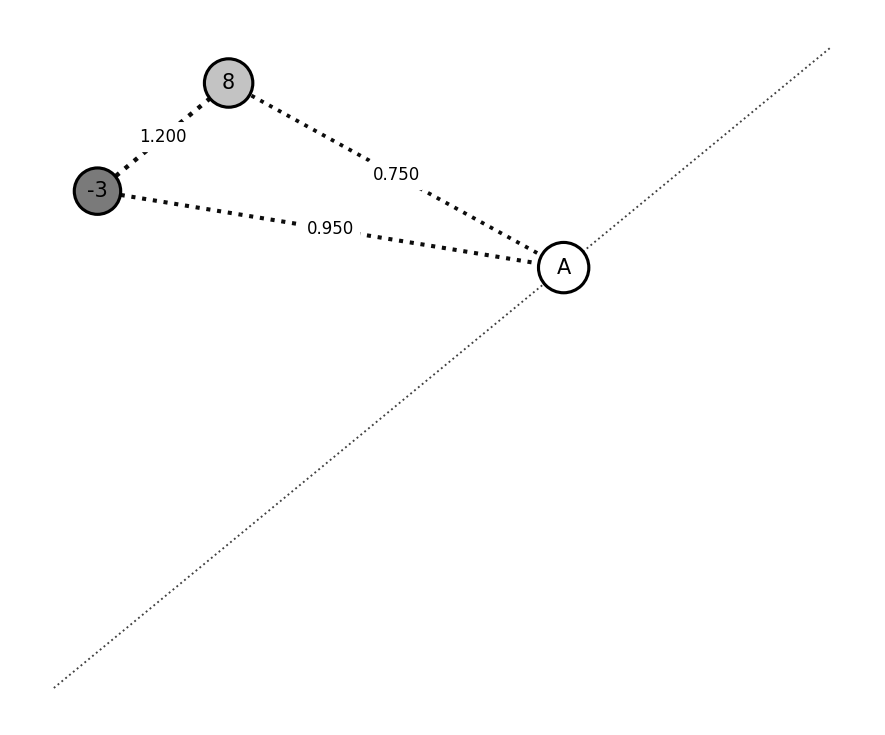}
    \caption{ 
    Visualization of a virtual persistence diagram: the vertex $A$ denotes the collapsed diagonal point, edges are weighted by the metric on the (discrete) persistence diagram, and the vertex labels record the (possibly negative) multiplicities.}
    \label{fig:toy-example}
\end{figure}

We follow Bubenik--Elchesen \cite{Bubenik2022VirtualPD} and work in the category $\mathbf{Lip}$ of metric spaces and Lipschitz maps. 
A \emph{metric pair} is a triple $(X,d,A)$ consisting of a metric space $(X,d)$ and a distinguished subset $A\subseteq X$.

\begin{definition}\label{def:pstrength}
Let $(X,d,A)$ be a metric pair and let $p\in[1,\infty]$. 
Write $d(x,A):=\inf_{a\in A} d(x,a)$ and let $\|(\cdot,\cdot)\|_p$ denote the $\ell^p$ norm on $\mathbb{R}^2$. 
The \emph{$p$-strengthening} \cite{Bubenik2022VirtualPD} of $d$ with respect to $A$ is
\begin{equation}\label{eq:pstrength}
d_p(x,y)\ :=\ \min\!\bigl\{\,d(x,y),\ \|(d(x,A),\,d(y,A))\|_p\,\bigr\},\qquad x,y\in X.
\end{equation}
\end{definition}

\begin{definition}\label{def:quotient}
Let \(q:X\to X/A\) be the quotient map that identifies all points of
\(A\) to a single basepoint \([A]\). The \emph{quotient metric induced by
the \(p\)-strengthening} is the metric \(\overline d_p\) on \(X/A\)
defined by
\begin{equation}\label{eq:pstrength-quotient}
\overline d_p(q(x),q(y)):=d_p(x,y),
\qquad x,y\in X.
\end{equation}
Equivalently, \(d_p=q^\ast\overline d_p\). In particular,
\(d_p\le d\) on \(X\), and \(\overline d_p\) metrizes the quotient
\(X/A\) \cite{Bubenik2022VirtualPD}.
\end{definition}

\begin{remark}\label{rem:finite-quotient-representations}
When \(X/A\) is finite, the pointed metric space
\((X/A,\overline d_1)\) with basepoint \([A]\) is completely determined
by the pairwise distances \(\overline d_1(u,v)\) for \(u,v\in X/A\).
Equivalently, these data may be encoded either as a distance matrix or
as the weighted complete graph on the vertex set \(X/A\) with edge weight
\(\overline d_1(u,v)\) between distinct vertices \(u\) and \(v\). We
will use the distance-matrix and complete-graph representations
interchangeably.
\end{remark}

We use the free commutative monoid on $X$ to model persistence diagrams as formal sums.

\begin{definition}\label{def:finite-diagrams}
Let \(D(X)\) denote the free commutative monoid on \(X\), identified with
the set of finitely supported functions \(X\to\mathbb N\). Equivalently,
elements of \(D(X)\) may be written as finite formal sums
\[
\alpha=\sum_i n_i x_i,
\qquad n_i\in\mathbb N,\ x_i\in X.
\]
For a metric pair \((X,d,A)\), define
\[
D(X,A)\ :=\ D(X)/D(A)\ \cong\ D(X\setminus A),
\]
whose elements are the finite persistence diagrams on \((X,A)\)
\cite{Bubenik2022VirtualPD}.
\end{definition}

Distances between diagrams are given by Wasserstein metrics built from matchings.

\begin{definition}\label{def:Wp-diagrams}
Let \(\pi_1,\pi_2:X\times X\to X\) be the coordinate projections and
write \((\pi_i)_*:D(X\times X)\to D(X)\) for the induced monoid maps. For \(\alpha,\beta\in D(X,A)\), a \emph{matching}
\cite{Bubenik2022VirtualPD} between \(\alpha\) and \(\beta\) is any
\(\sigma\in D(X\times X)\) such that
\[
(\pi_1)_*\sigma\ =\ \alpha\ \ (\mathrm{mod}\ D(A)),\qquad
(\pi_2)_*\sigma\ =\ \beta\ \ (\mathrm{mod}\ D(A)).
\]
For \(p\in[1,\infty)\), the \emph{\(p\)-Wasserstein distance} is
defined by
\begin{equation}\label{eq:Wp}
W_p[d](\alpha,\beta)
\ :=\
\inf_{\sigma=\sum_{i=1}^n (x_i,y_i)}
\Bigl(\sum_{i=1}^n d(x_i,y_i)^p\Bigr)^{1/p},
\end{equation}
where the infimum ranges over all matchings \(\sigma\) between
\(\alpha\) and \(\beta\). For \(p=\infty\), define
\begin{equation}\label{eq:Winfty}
W_\infty[d](\alpha,\beta)
\ :=\
\inf_{\sigma=\sum_{i=1}^n (x_i,y_i)}
\max_{1\le i\le n} d(x_i,y_i).
\end{equation}
\end{definition}

The key structural property we need is translation-invariance of $W_p$ on the monoid $D(X,A)$.

\begin{theorem}\label{thm:TI-criterion}
For a metric pair $(X,d,A)$ and $p\in[1,\infty]$, the following are equivalent~\cite{Bubenik2022VirtualPD}:
\begin{enumerate}
  \item $W_p[d]$ is translation-invariant on $D(X,A)$;
  \item the quotient metric $\overline d_p$ on $X/A$ is a $p$-metric, i.e.
  \[
  \overline d_p(\overline x,\overline y)\ \le\ \bigl\|(\overline d_p(\overline x,\overline z),\overline d_p(\overline z,\overline y))\bigr\|_p
  \]
  for all $\overline x,\overline y,\overline z\in X/A$. 
\end{enumerate}
In particular, $W_1[d]$ is always translation-invariant.
\end{theorem}

\begin{remark}
Unless otherwise stated, we work with the $1$-Wasserstein distance $W_1$. 
For $p>1$, translation-invariance need not hold, so the Grothendieck-group construction recalled below does not in general apply directly to $W_p$.
\end{remark}

We next recall the Grothendieck completion of a cancellative monoid and the induced metric extension.

\begin{definition}\label{def:groth}
Let $(M,+)$ be a cancellative commutative monoid. 
Define an equivalence relation on $M\times M$ by
\begin{equation}\label{eq:groth-equivalence}
(a,b)\sim(c,e)\ \Longleftrightarrow\ \exists\,k\in M:\ a+e+k=b+c+k.
\end{equation}
Write $a-b$ for the equivalence class of $(a,b)$ and set $K(M):=(M\times M)/\!\sim$ with
\[
(a-b)+(c-e)\ :=\ (a+c)-(b+e),\qquad 0:=0-0,\qquad -(a-b):=b-a.
\]
The canonical map $u:M\to K(M)$ given by $u(a)=a-0$ is a monoid homomorphism with the usual universal property: any monoid map $M\to H$ into an abelian group $H$ factors uniquely through a group homomorphism $K(M)\to H$.
\end{definition}

\begin{proposition}\label{prop:metric-lift}
Let $(M,+)$ be a cancellative commutative monoid equipped with a translation-invariant metric $d$. 
Then
\begin{equation}\label{eq:groth-metric}
\rho(a-b,\ c-e)\ :=\ d(a+e,\ c+b)
\end{equation}
is a well-defined translation-invariant metric on $K(M)$, and $u:M\to K(M)$ is $1$-Lipschitz. 
Moreover, $\rho$ is the unique translation-invariant metric on $K(M)$ extending $d$ in the sense that $\rho(u(a),u(b))=d(a,b)$ for all $a,b\in M$~\cite{Bubenik2022VirtualPD}.
\end{proposition}

\begin{remark}
Specializing to $M=D(X,A)$ and $d=W_1[d]$, Bubenik and Elchesen denote the Grothendieck group $K(D(X,A))$ by $K(X,A)$ and call its elements \emph{virtual persistence diagrams} on $(X,A)$ \cite{Bubenik2022VirtualPD}. 
We adopt this terminology; the detailed specialization will be used later, but the only input needed in this section is the abstract metric-lift construction above.
\end{remark}

\subsection{Reproducing Kernel Hilbert Spaces}\label{subsec:rkhs-toolkit}

Positive definite kernels provide a way to encode similarity between points by
means of inner products in an implicit feature space.  Reproducing kernel
Hilbert spaces make this correspondence precise and underlie the classical
\emph{kernel trick}: linear methods in a Hilbert space of functions can be
implemented using only kernel evaluations $k(x,y)$, without ever writing
feature maps explicitly; see, for example, \cite{Berlinet2004RKHS}.

\begin{definition}\label{def:pd-kernel}
Let $X$ be a set. A function $k:X\times X\to\mathbb{C}$ is a
\emph{positive definite kernel} if for every $n\in\mathbb{N}$, every choice of
points $x_1,\dots,x_n\in X$, and every $c_1,\dots,c_n\in\mathbb{C}$,
\[
  \sum_{i,j=1}^n k(x_i,x_j)\,c_i\overline{c_j}\ \ge\ 0.
\]
Equivalently, every Gram matrix $[\,k(x_i,x_j)\,]_{i,j=1}^n$ is positive
semidefinite.  In particular, such a kernel can always be written in the form
\[
  k(x,y)\ =\ \langle \Phi(x),\Phi(y)\rangle_{\mathcal H}
\]
for some Hilbert space $\mathcal H$ and feature map $\Phi:X\to\mathcal H$.
\end{definition}

The RKHS viewpoint fixes $\mathcal H$ canonically as a space of functions on
$X$, in which evaluation is compatible with the kernel.

\begin{definition}\label{def:rkhs}
Let $X$ be a set. A \emph{reproducing kernel Hilbert space} (RKHS) on $X$ is a
Hilbert space $\mathcal{H}$ of complex-valued functions on $X$ such that, for
every $x\in X$, the evaluation functional
\[
  \mathcal{H}\longrightarrow\mathbb{C},\qquad f\longmapsto f(x),
\]
is continuous.
\end{definition}

By the Riesz representation theorem, continuity of evaluation implies that for
each $x\in X$ there is a unique function $k(\cdot,x)\in\mathcal H$ with
\begin{equation}\label{eq:reproducing-property}
  f(x)\ =\ \langle f,\ k(\cdot,x)\rangle_{\mathcal H}
  \qquad\text{for all }f\in\mathcal H.
\end{equation}
The function $k$ is a positive definite kernel, and the following theorem
describes the resulting correspondence.

\begin{theorem}\label{thm:rkhs}
Let $X$ be a set and $k:X\times X\to\mathbb{C}$ a positive definite kernel.
Then there exists a unique RKHS $\mathcal{H}_k$ on $X$ such that:
\begin{enumerate}
\item for each $x\in X$, the function $k(\cdot,x)$ lies in $\mathcal{H}_k$;
\item for all $f\in\mathcal{H}_k$ and $x\in X$,
  \[
    f(x)\ =\ \langle f,\ k(\cdot,x)\rangle_{\mathcal{H}_k}
    \quad\text{\emph{(reproducing property).}}
  \]
\end{enumerate}
Conversely, every RKHS of functions on $X$ arises in this way from a unique
positive definite kernel $k$.
\end{theorem}

In particular, the map $x\mapsto k(\cdot,x)$ embeds $X$ into $\mathcal H_k$ with
\[
  k(x,y)\ =\ \langle k(\cdot,x),k(\cdot,y)\rangle_{\mathcal H_k}.
\]
Working with $\mathcal H_k$ thus amounts to working with the kernel $k$:
inner products, norms, and linear constructions in feature space can all be
expressed purely in terms of kernel evaluations.

\subsection{Prerequisite harmonic analysis}\label{subsec:harmonic-toolkit}

We recall standard notions from abstract harmonic analysis on locally compact abelian groups; see, for example, \cite{Folland2015CourseAH}.

\begin{definition}
A \emph{locally compact abelian group} (LCA group) is a topological group $G$ that is abelian and locally compact Hausdorff. A \emph{Haar measure} on $G$ is a nonzero Radon measure $\mu$ on $G$ that is left invariant:
\[
\mu(xE)=\mu(E)\qquad\text{for all Borel }E\subseteq G\text{ and all }x\in G.
\]
Haar measures exist and are unique up to a positive scalar multiple. When $G$ is discrete, one may take $\mu$ to be the counting measure, and when $G$ is compact, there is a unique Haar probability measure.
\end{definition}

\begin{definition}\label{def:characters-dual}
Let $G$ be an LCA group. A \emph{(unitary) character} of $G$ is a continuous group homomorphism
\[
\chi : G \longrightarrow \mathbb T := \{z\in\mathbb{C} : |z|=1\}.
\]
Each character is a one-dimensional unitary representation of $G$.  
The \emph{Pontryagin dual} $\widehat G$ is the set of all characters, endowed with pointwise multiplication and the compact--open topology.

A character will also be called a Fourier mode: this means that it is one
single frequency component used in Fourier expansions on \(G\).
\end{definition}

\begin{example}\label{ex:ZN-dual}
For $G=\mathbb{Z}^N$ with the discrete topology, the dual group is the $N$-torus
\[
\widehat G\ \cong\ \mathbb{T}^N:=\mathbb{R}^N/2\pi\mathbb{Z}^N.
\]
Writing $\theta=(\theta_1,\dots,\theta_N)\in[0,2\pi)^N$ and $k=(k_1,\dots,k_N)\in\mathbb{Z}^N$, the corresponding character is
\[
\chi_\theta(k)\ :=\ \exp\!\big(i\langle k,\theta\rangle\big),
\qquad
\langle k,\theta\rangle\ :=\ \sum_{j=1}^N k_j\theta_j.
\]
Under this identification, the Haar measure on $\widehat G$ is the normalized Lebesgue measure on $\mathbb{T}^N$.
\end{example}

The coordinates $\theta_j$ may be viewed as circular phase coordinates, and the character $\chi_\theta(k)$ as the aggregate phase associated to the group element $k$.

\begin{definition}\label{def:fourier}
Let $G$ be an LCA group with Haar measure $\mu$. For $f\in L^1(G,\mu)$, the \emph{Fourier transform} $\widehat f:\widehat G\to\mathbb{C}$ is
\begin{equation}\label{eq:fourier-transform}
\widehat f(\chi)\ :=\ \int_G \chi(x)\,f(x)\,d\mu(x),
\qquad \chi\in\widehat G.
\end{equation}
When $G$ is discrete with counting measure and we index characters by $\theta\in\widehat G$, this specializes to
\begin{equation}\label{eq:discrete-fourier-transform}
\widehat f(\theta)\ :=\ \sum_{k\in G} \chi_\theta(k)\,f(k).
\end{equation}
The Fourier transform extends uniquely to a unitary operator
\[
\mathcal{F}:L^2(G,\mu)\longrightarrow L^2(\widehat G,\widehat\mu),
\]
where $\widehat\mu$ is Haar measure on $\widehat G$ (Plancherel theorem).
\end{definition}

\begin{definition}\label{def:FS}
Let $G$ be an LCA group and let $\nu$ be a finite complex Borel measure on $\widehat G$. The \emph{Fourier--Stieltjes transform} of $\nu$ is the function $F_\nu:G\to\mathbb{C}$ defined by
\begin{equation}\label{eq:fourier-stieltjes}
F_\nu(\alpha)\ :=\ \int_{\widehat G} \chi(\alpha)\,d\nu(\chi),
\qquad \alpha\in G.
\end{equation}
We write $|\nu|$ for the total variation measure of $\nu$.
\end{definition}

\begin{proposition}\label{prop:FS-translation}
Let $G$ and $\nu$ be as above. The kernel
\begin{equation}\label{eq:translation-invariant-kernel}
k_\nu(\alpha,\beta)\ :=\ F_\nu(\alpha-\beta),\qquad \alpha,\beta\in G,
\end{equation}
is translation-invariant in the sense that
\[
k_\nu(\alpha+\gamma,\beta+\gamma)\ =\ k_\nu(\alpha,\beta)
\qquad\text{for all }\gamma\in G.
\]
\end{proposition}

\begin{definition}\label{def:pd-function}
A function $\varphi:G\to\mathbb{C}$ is \emph{positive definite} if, for every $n\in\mathbb{N}$, points $\alpha_1,\dots,\alpha_n\in G$, and coefficients $c_1,\dots,c_n\in\mathbb{C}$,
\[
\sum_{i,j=1}^n \varphi(\alpha_i-\alpha_j)\,c_i\overline{c_j}\ \ge\ 0.
\]
If $\varphi$ is positive definite, the kernel $k(\alpha,\beta):=\varphi(\alpha-\beta)$ is positive definite on $G\times G$ in the sense of Definition~\ref{def:pd-kernel}.
\end{definition}

\begin{theorem}[Bochner's theorem]\label{thm:bochner}
Let $G$ be an LCA group. A continuous function $\varphi:G\to\mathbb{C}$ is positive definite if and only if there exists a finite positive Borel measure $\nu$ on $\widehat G$ such that
\begin{equation}\label{eq:bochner-representation}
\varphi(\alpha)\ =\ \int_{\widehat G} \chi(\alpha)\,d\nu(\chi),
\qquad \alpha\in G.
\end{equation}
Equivalently, a continuous kernel $k:G\times G\to\mathbb{C}$ is positive definite and translation-invariant if and only if $k(\alpha,\beta)=\varphi(\alpha-\beta)$ with $\varphi$ of the above form.
\end{theorem}

A positive definite function can then be said to have such a representation via Bochner's theorem above.

\begin{definition}\label{def:fourier-multiplier}
Let $G$ be an LCA group and let $T$ be a bounded translation-invariant operator on $L^2(G,\mu)$.  A measurable function $m:\widehat G\to\mathbb C$ is called the \emph{Fourier multiplier}, or \emph{symbol}, of $T$ if
\begin{equation}\label{eq:fourier-multiplier}
\widehat{Tf}(\chi)
=
m(\chi)\widehat f(\chi)
\end{equation}
for all $f\in L^2(G,\mu)$.
\end{definition}

A Fourier multiplier acts by reweighting characters on the dual group according to frequency, where frequency measures how rapidly the phase of a character changes as one moves through the group. If \(L\) is a nonnegative translation-invariant operator on an LCA group \(G\), then its Fourier multiplier \(m\) is often called the \emph{symbol} of \(L\). In particular, the heat semigroup \(e^{-tL}\) has the \emph{heat spectral multiplier}
\begin{equation}\label{eq:heat-spectral-multiplier}
m_t(\chi)
=
e^{-t m(\chi)},
\qquad \chi\in\widehat G,\ t>0.
\end{equation}
Thus, heat kernels arise by exponentially damping characters according to the symbol \(m\).

\subsubsection{Prerequisite Approximation Theory} \label{subsubsec:approximation-theory}

Random Fourier features approximate Fourier--Stieltjes kernels by replacing the integral in Bochner's theorem (Theorem~\ref{thm:bochner}) over the dual group with a finite empirical average of sampled characters.

Let \(G\) be an LCA group, let \(\nu\) be a finite positive Borel measure on \(\widehat G\), and suppose the associated translation-invariant Bochner kernel
\begin{equation}\label{eq:bochner-kernel}
k_\nu(\alpha,\beta)
:=
\int_{\widehat G}\chi(\alpha-\beta)\,d\nu(\chi)
\end{equation}
is real-valued.

\begin{definition}\label{def:rff-general}
Fix \(R\in\mathbb N\), and let \(\chi^{(1)},\dots,\chi^{(R)}\) be independent samples from the probability measure \(\nu/\nu(\widehat G)\) on \(\widehat G\). The associated random Fourier feature map is
\begin{equation}\label{eq:rff-map}
\Phi_R:G\longrightarrow\mathbb R^{2R},
\end{equation}
defined by
\begin{equation}\label{eq:rff-definition}
\Phi_R(\alpha)
:=
\sqrt{\frac{\nu(\widehat G)}{R}}
\bigl(
\Re\chi^{(r)}(\alpha),
\Im\chi^{(r)}(\alpha)
\bigr)_{r=1}^R.
\end{equation}
\end{definition}

The inner products \(\langle\Phi_R(\alpha),\Phi_R(\beta)\rangle \) give Monte Carlo approximations of the kernel values \(k_\nu(\alpha,\beta)\), thereby replacing the Fourier--Stieltjes integral with a finite-dimensional Euclidean feature map.

\subsection{Prerequisite Potential Theory}\label{subsec:potential-toolkit}

A weighted graph naturally carries an energy measuring variation across edges. Let \(H=(V,E,w)\) be a finite connected weighted graph with symmetric edge weights \(w_{uv}=w_{vu}\ge 0\), extended by \(w_{uv}=0\) whenever \(\{u,v\}\notin E\).

\begin{definition}\label{def:graph-laplacian}
The \emph{graph Laplacian} \(L:\ell^2(V)\to\ell^2(V)\) is defined by
\begin{equation}\label{eq:graph-laplacian}
(Lf)(u)
:=
\sum_{v\in V}
w_{uv}\,\bigl(f(u)-f(v)\bigr),
\qquad u\in V.
\end{equation}
\end{definition}

\begin{definition}\label{def:graph-dirichlet}
Its \emph{Dirichlet form} is
\begin{equation}\label{eq:dirichlet-form}
\mathcal E(f,g)
:=
\frac12
\sum_{u,v\in V}
w_{uv}\,
\bigl(f(u)-f(v)\bigr)
\overline{\bigl(g(u)-g(v)\bigr)},
\end{equation}
for \(f,g\in\ell^2(V)\). We write \( \mathcal E(f):=\mathcal E(f,f). \)
\end{definition}

The identity \(\mathcal E(f)=\langle f,Lf\rangle_{\ell^2(V)}\) shows that the Dirichlet energy is generated by the graph Laplacian (Definition~\ref{def:graph-laplacian}). Since \(H\) is connected, \(\mathcal E(f)=0\) if and only if \(f\) is constant. Thus \(\mathcal E(f)\) increases as \(f\) increases in frequency across edges with large weight.

\subsection{Image segmentation}\label{subsec:image-segmentation}

A grayscale digital image on a finite rectangular grid is a function
\(I:\Omega\to\{0,\dots,2^N-1\}\), where
\(\Omega\subset\mathbb Z^2\) is the pixel lattice and \(N\) is the bit
depth. Analogously, voxel data are functions on finite subsets of
\(\mathbb Z^3\). Since the intensity set is finite and totally ordered,
thresholding across all intensity levels produces a finite filtration by
cubical complexes.

For a threshold \(t\), the associated cubical sublevel complex
is generated by the pixels or voxels with intensity at most \(t\),
together with all of their faces. Persistent homology of this filtration
encodes connected components, loops, and higher-dimensional homological
features across intensity scales.

\begin{definition}\label{def:fixed-cubical-filtration}
A \emph{fixed cubical filtration} is a choice of underlying grid,
cubical-complex construction, and a thresholding rule that is held fixed
across all images and masks under consideration.
\end{definition}

In binary image segmentation, the goal is to predict a
foreground-background segmentation from an input image \(I\). A
ground-truth mask is a binary-valued function
\(y:\Omega\to\{0,1\}\), where \(y(p)=1\) indicates foreground and
\(y(p)=0\) indicates background. A model prediction is a soft mask
\(\hat y:\Omega\to[0,1]\), where \(\hat y(p)\) is the predicted
foreground score at \(p\in\Omega\). Given a threshold
\(\tau\in[0,1]\), the corresponding binary predicted mask is
\begin{equation}\label{eq:thresholded-mask}
\hat y_\tau(p)
:=
\mathbf 1_{\{\hat y(p)\ge\tau\}}.
\end{equation}

Two standard overlap measures for binary masks are the Dice
coefficient
\begin{equation}\label{eq:dice-score}
\mathrm{Dice}(y,\hat y_\tau)
:=
\frac{
2\,
\bigl|
\{p\in\Omega:y(p)=1\text{ and }\hat y_\tau(p)=1\}
\bigr|
}{
\bigl|
\{p\in\Omega:y(p)=1\}
\bigr|
+
\bigl|
\{p\in\Omega:\hat y_\tau(p)=1\}
\bigr|
},
\end{equation}
and the intersection-over-union score
\begin{equation}\label{eq:iou}
\mathrm{IoU}(y,\hat y_\tau)
:=
\frac{
\bigl|
\{p\in\Omega:y(p)=1\text{ and }\hat y_\tau(p)=1\}
\bigr|
}{
\bigl|
\{p\in\Omega:y(p)=1\text{ or }\hat y_\tau(p)=1\}
\bigr|
}.
\end{equation}
Both quantities take values in \([0,1]\), with larger values indicating
greater agreement between the binary predicted and ground-truth masks.

For optimization with soft predictions, a standard segmentation
objective is the soft Dice loss
\begin{equation}\label{eq:dice-loss}
\mathcal L_{\mathrm{Dice}}(y,\hat y)
:=
1-
\frac{2\langle y,\hat y\rangle+1}
{\|y\|_1+\|\hat y\|_1+1},
\end{equation}
which is a differentiable relaxation of the Dice coefficient for soft
masks \(\hat y:\Omega\to[0,1]\). In contrast, a topological loss
compares persistence diagrams induced by a fixed cubical filtration in
the sense of Definition~\ref{def:fixed-cubical-filtration}.

\section{Pontryagin duality}\label{subsec:LCA}

We restrict attention to filtrations with finite indexing sets. The associated number of possible nontrivial intervals is therefore finite, and we write \(X\setminus A=\{x_1,\dots,x_N\}\).
We fix $p=1$ and write $\overline d_1$ for the quotient metric on $X/A$
(see Definition~\ref{def:quotient}). Let $W_1$ denote the corresponding
$1$-Wasserstein distance on $D(X,A)$ induced by $(X/A,\overline d_1)$
as in Definition~\ref{def:Wp-diagrams}. Via the Grothendieck
completion (Proposition~\ref{prop:metric-lift}), we equip $K(X,A)$ with
the translation-invariant metric
\[
\rho(a-b,\ c-d)\ :=\ W_1(a+d,\ c+b),\qquad a,b,c,d\in D(X,A).
\]
Since $X\setminus A$ is finite, we identify
\begin{equation}\label{eq:KXA-lattice}
K(X,A)\ \cong\ \mathbb Z^{X\setminus A}
\end{equation}
by sending the class of each $x_j\in X\setminus A$ to the standard basis
vector $e_j$ and the collapsed basepoint $[A]\in X/A$ to $0\in K(X,A)$.

\begin{lemma}\label{prop:K-discrete-lc}
$(K(X,A),\rho)$ is a discrete metric group, hence a locally compact Hausdorff space.
\end{lemma}

\begin{proof}
Since $X/A$ is finite and $\overline d_1$ is a metric, there is a minimal
nonzero ground distance
\begin{equation}\label{eq:dmin}
d_{\min}\ :=\ \min\!\Bigl(\ \min_{j\neq k}\overline d_1(x_j,x_k),\ 
                           \min_j \overline d_1(x_j,[A])\ \Bigr)\ >\ 0.
\end{equation}
By the definitions of $W_1$ and $\rho$, every nonzero value
$\rho(\kappa,\lambda)$ is a sum of finitely many terms, each at least
$d_{\min}$, so $\rho(\kappa,\lambda)\ge d_{\min}$ whenever
$\kappa\neq\lambda$. Thus each ball $B_\rho(\kappa,d_{\min}/2)$ is a
singleton, so $(K(X,A),\rho)$ is discrete. Any discrete metric space is
locally compact Hausdorff.
\end{proof}

\subsection{Characters}\label{subsec:characters}

\begin{figure}[ht]
    \centering
    \includegraphics[width=0.55\linewidth]{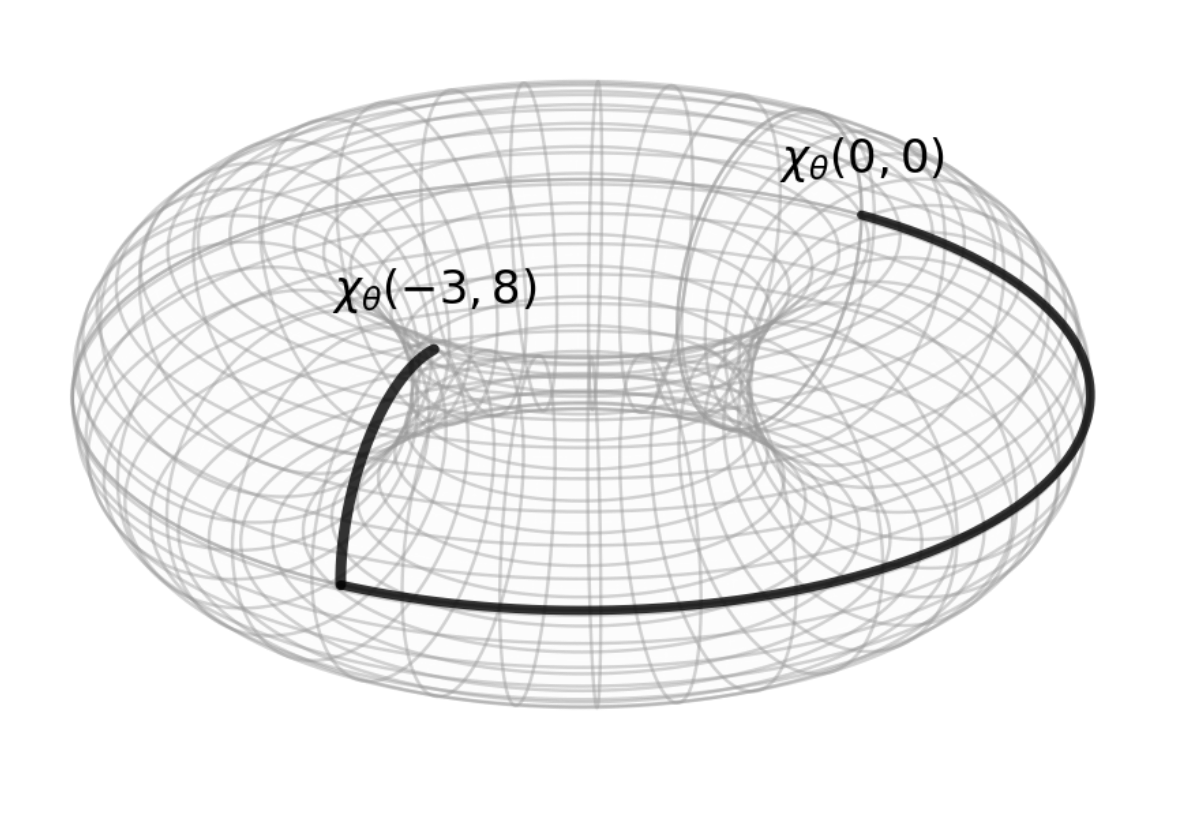}
    \caption{
    The embedding of the virtual persistence diagram from Figure~\ref{fig:toy-example} into the Pontryagin dual torus. The points $\chi_\theta(0,0)$ and $\chi_\theta(-3,8)$ denote the character evaluations of the empty diagram and the virtual persistence diagram from Figure~\ref{fig:toy-example}, respectively. The black curve represents the resulting character difference function, whose arclength equals the distance from the trivial diagram.}
    \label{fig:torus-embedding}
\end{figure}

The point of passing to the dual is to interpret Fourier coordinates on \(K(X,A)\) geometrically. A character \(\chi_\theta\) is determined by assigning a circular phase \(\theta_j\) to each interval \(x_j\in X/A\), and evaluating \(\chi_\theta\) on a virtual persistence diagram aggregates these phase contributions according to the diagram multiplicities. Thus, a Fourier mode on \(K(X,A)\) is equivalently a circular coordinate system on the finite persistence space \(X/A\). The estimates below show that the relevant notion of frequency is phase instability: a mode is high frequency precisely when nearby intervals receive very different phases.

\begin{figure}[ht]
    \centering
    \includegraphics[width=\linewidth]{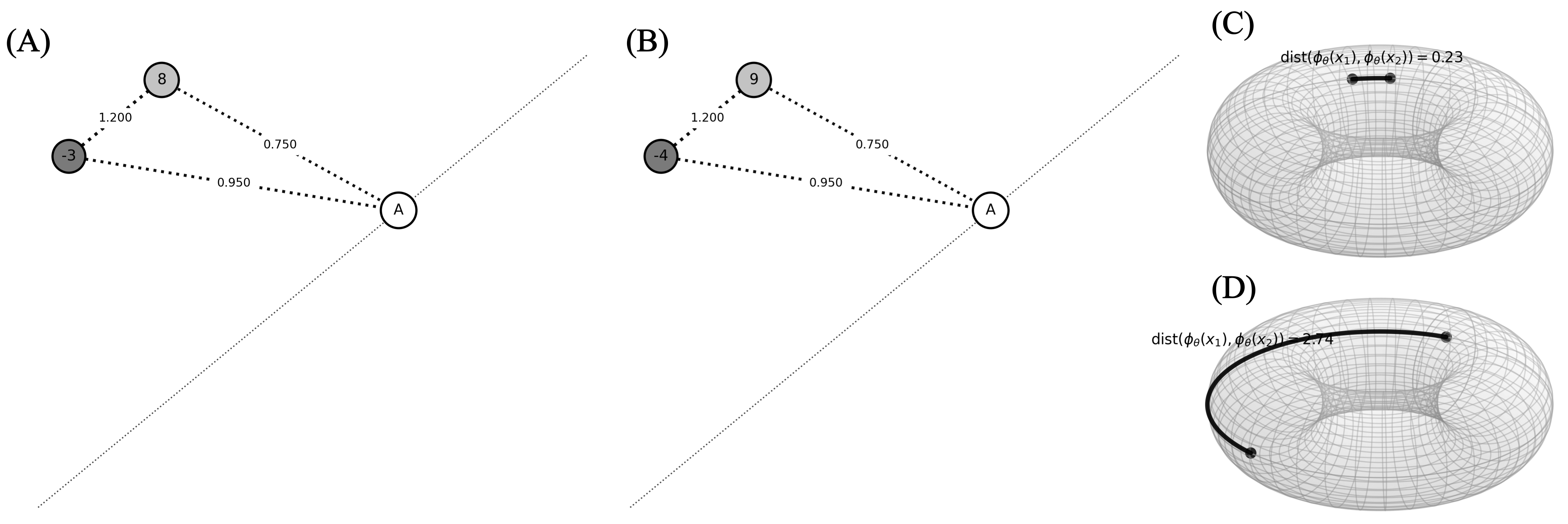}
    \caption{(A) The toy virtual persistence diagram from Figure~\ref{fig:toy-example}. (B) A nearby perturbation obtained by changing the interval multiplicities. (C) A low-frequency character produces only a small phase displacement between the two diagrams, indicating stable behavior under perturbation. (D) A high-frequency character produces a much larger phase displacement for the same perturbation, illustrating the instability of high-frequency Fourier modes.}
    \label{fig:vpd-frequency-stability}
\end{figure}

Since $(K(X,A),\rho)$ is a discrete abelian group (Lemma~\ref{prop:K-discrete-lc}), it is an LCA group in the sense of Definition~\ref{def:characters-dual}, and its Pontryagin dual $\widehat{K(X,A)}$ is compact. We now fix an explicit parametrization of characters.

By the lattice identification in Equation~\eqref{eq:KXA-lattice}, we regard $K(X,A)$ as the free abelian group $\mathbb Z^{X\setminus A}$ by sending the class of each $x_j\in X\setminus A$ to a basis vector $e_j$ and the collapsed basepoint $[A]$ to $0$. Writing an element $k\in K(X,A)$ as a finite sum $k=\sum_j k_j e_j$, we have the following standard description of the dual group; see, for example, \cite{Folland2015CourseAH}.

\begin{proposition}\label{prop:dual} Write $\theta=(\theta_1,\dots,\theta_{|X\setminus A|})$ with $\theta_j\in[0,2\pi)$. Every character $\chi$ on $K(X,A)$ is of the form
\begin{equation}\label{eq:character-formula}
\chi_\theta\!\left(\sum_{j=1}^{|X\setminus A|} n_j e_j\right)
\ =\
\exp\!\left(i\sum_{j=1}^{|X\setminus A|} n_j\,\theta_j\right),
\qquad n_j\in\mathbb Z,
\end{equation}
and this identifies $\widehat{K(X,A)}$ with the $|X\setminus A|$-torus $\mathbb T^{|X\setminus A|}=\mathbb R^{|X\setminus A|}/2\pi\mathbb Z^{|X\setminus A|}$. The pairing between
$K(X,A)$ and its dual is
\begin{equation}\label{eq:dual-pairing}
\langle k,\theta\rangle\ :=\ \sum_{j=1}^{|X\setminus A|} k_j\theta_j\quad(\mathrm{mod}\ 2\pi),
\end{equation}
so that $\chi_\theta(k)=e^{i\langle k,\theta\rangle}$. Haar measure on
$K(X,A)$ is a counting measure, and Haar measure on
$\widehat{K(X,A)}\cong\mathbb T^{|X \setminus A|}$ is normalized Lebesgue measure
$(2\pi)^{-|X\setminus A|}d\theta$ on $[0,2\pi)^{|X\setminus A|}$, so that $\mu\left(\mathbb T^{|X\setminus A|}\right)=1$.
\end{proposition}

\begin{proof}
It is classical that $\widehat{\mathbb Z}\cong\mathbb T$ via
$n\mapsto e^{in\theta}$ with Haar measure $(2\pi)^{-1}d\theta$ on
$[0,2\pi)$, and that Pontryagin duality preserves finite products
\cite{Folland2015CourseAH}. Since $K(X,A)\cong\mathbb Z^{|X\setminus A|}$ is a
finite product of copies of $\mathbb Z$, its dual is a finite product of
circles, i.e.\ $\widehat{K(X,A)}\cong\mathbb T^{|X\setminus A|}$. The displayed formula for
$\chi_\theta$ is the product of the one-dimensional characters, and the
pairing and Haar normalization are inherited from the one-dimensional case.
\end{proof}

For a fixed probability measure $\nu$ on $\widehat{K(X,A)}$, Bochner's theorem (Theorem~\ref{thm:bochner}) applied to the Fourier-Stieltjes transform (Equation~\eqref{eq:fourier-stieltjes}) produces a translation-invariant positive-definite kernel on $K(X,A)$ by averaging the characters $\chi_\theta$. In our setting, $\nu$ will be chosen with density proportional to the heat spectral multiplier from Equation~\eqref{eq:heat-spectral-multiplier}, generated by the discrete Laplacian of the weighted complete graph corresponding to $(X/A,\overline d_1)$ (Remark~\ref{rem:finite-quotient-representations}; Definition~\ref{def:graph-laplacian}).

To control stability of the resulting features with respect to the diagram metric, we need to quantify how sensitive each mode $\chi_\theta$ is to perturbations in $(K(X,A),\rho)$. We measure this via the Lipschitz seminorm $\mathrm{Lip}(\cdot)$ from Definition~\ref{def:Lip-seminorm}.

For each $\theta\in\mathbb T^{|X\setminus A|}$ we define the \emph{phase
function}
\begin{equation}\label{eq:phase-function}
\phi_\theta: X/A\longrightarrow S^1,\qquad
\phi_\theta([A])=1,\quad \phi_\theta(x_j)=e^{i\theta_j}.
\end{equation}
We equip \(S^1=\mathbb R/2\pi\mathbb Z\) with the quotient metric induced by
the Euclidean metric on \(\mathbb R\), namely
\[
\operatorname{dist}(\theta,\phi)
:=
\inf_{k\in\mathbb Z}
|\theta-\phi+2\pi k|.
\]
This metric realizes the length of the shorter arc between two points on
\(S^1\).
The Lipschitz seminorm
\(\mathrm{Lip}(\phi_\theta)\) measures the largest difference of phases,
using the circle metric \(\operatorname{dist}\) on \(S^1\), per unit
ground distance in \((X/A,\overline d_1)\). Explicitly, it is the
supremum of
\[
\frac{
\operatorname{dist}\bigl(\phi_\theta(x),\phi_\theta(y)\bigr)
}{
\overline d_1(x,y)
}
\]
over distinct \(x,y\in X/A\). Corollary~\ref{cor:edgewise-char} later identifies this supremum, for a
shortest-path graph model of \(X/A\), with the maximum of the differences of phases appearing in Equation~\eqref{eq:edgewise-character-lipschitz}.

\begin{lemma}\label{lem:char-lip}
For every $\theta\in\mathbb T^{|X\setminus A|}$,
\[
\mathrm{Lip}(\chi_\theta)
\ =\
\sup_{\gamma\ne 0}
\frac{\operatorname{dist}(\chi_\theta(\gamma),1)}{\rho(\gamma,0)}.
\]
\end{lemma}

\begin{proof}
Let $\alpha,\beta\in K(X,A)$ and set $\gamma=\alpha-\beta$.
Since $\chi_\theta$ is a character, we have
\begin{align*}
\chi_\theta(\alpha)\chi_\theta(\beta)^{-1}
&=
\chi_\theta(\alpha-\beta) \\
&=
\chi_\theta(\gamma).
\end{align*}
The geodesic distance on $S^1$ is translation-invariant, so
\[
\operatorname{dist}\bigl(\chi_\theta(\alpha),\chi_\theta(\beta)\bigr)
=
\operatorname{dist}\bigl(\chi_\theta(\gamma),1\bigr).
\]
Translation-invariance of $\rho$ gives \( \rho(\alpha,\beta)=\rho(\gamma,0). \) Taking the supremum over all $\alpha\ne\beta$ is therefore equivalent to taking the supremum over all $\gamma\ne0$, which proves the claim.
\end{proof}

\subsection{Phase identity}\label{subsec:phase-bound}

Characters $\chi_\theta$ on $K(X,A)$ determine phase functions on the interval space $(X/A,\overline d_1)$, and evaluating $\chi_\theta$ on a virtual persistence diagram aggregates these interval phases according to the diagram multiplicities. In this subsection, we make this relationship quantitative by comparing the Lipschitz seminorm of a character $\chi_\theta$ on $(K(X,A),\rho)$ with the Lipschitz seminorm of its associated phase function on $(X/A,\overline d_1)$. This identifies the relevant notion of frequency geometrically: a mode is high frequency precisely when nearby intervals receive very different phases. The resulting estimates show that the stability of the aggregated character is controlled entirely by the stability of the underlying phase map, and will later allow the kernel construction induced by the heat semigroup (The paragraph after Definition~\ref{def:fourier-multiplier}) to suppress precisely those Fourier modes with unstable phase maps.

The next lemma shows that the Lipschitz seminorm of the character agrees exactly with the Lipschitz seminorm of its associated phase function.

\begin{lemma}\label{lem:char-lip-comparison}
For every $\theta\in\mathbb T^{|X\setminus A|}$,
\[
\mathrm{Lip}(\chi_\theta)
=
\mathrm{Lip}(\phi_\theta).
\]
\end{lemma}

\begin{proof}
We first prove that
\[
\mathrm{Lip}(\chi_\theta)
\le
\mathrm{Lip}(\phi_\theta).
\]
Let $\gamma\in K(X,A)$ with $\gamma\ne0$. Choose
$\alpha,\beta\in D(X,A)$ such that $\gamma=\alpha-\beta$.
Viewing $\alpha$ and $\beta$ as atomic measures on
$(X/A,\overline d_1)$, let $m_{x,y}\ge0$ be the multiplicities of an
optimal matching for $W_1(\alpha,\beta)$. Thus
\[
\alpha(x)=\sum_y m_{x,y},\qquad
\beta(y)=\sum_x m_{x,y},
\]
and
\begin{align}
\rho(\gamma,0)
&=
W_1(\alpha,\beta) \nonumber\\
&=
\sum_{x,y}m_{x,y}\,\overline d_1(x,y).
\label{eq:rho-optimal-matching}
\end{align}

By the character formula and the homomorphism property, the value of
\(\chi_\theta\) on \(\gamma=\alpha-\beta\) is
\[
\chi_\theta(\gamma)
=
\prod_x \phi_\theta(x)^{\alpha(x)}
\prod_y \phi_\theta(y)^{-\beta(y)}.
\]
Using the marginal identities
\(\alpha(x)=\sum_y m_{x,y}\) and
\(\beta(y)=\sum_x m_{x,y}\), we rewrite this as
\begin{equation}\label{eq:character-factorization}
\chi_\theta(\gamma)
=
\prod_{x,y}
\bigl(\phi_\theta(x)\phi_\theta(y)^{-1}\bigr)^{m_{x,y}}.
\end{equation}
The geodesic metric on \(S^1\) is bi-invariant, so repeated application of the triangle inequality gives
\begin{align*}
\operatorname{dist}\bigl(\chi_\theta(\gamma),1\bigr)
&\le
\sum_{x,y}
m_{x,y}\,
\operatorname{dist}\bigl(\phi_\theta(x)\phi_\theta(y)^{-1},1\bigr)
\\
&=
\sum_{x,y}
m_{x,y}\,
\operatorname{dist}\bigl(\phi_\theta(x),\phi_\theta(y)\bigr)
\\
&\le
\mathrm{Lip}(\phi_\theta)
\sum_{x,y}
m_{x,y}\,\overline d_1(x,y)
\\
&=
\mathrm{Lip}(\phi_\theta)\,\rho(\gamma,0).
\end{align*}
Dividing by \(\rho(\gamma,0)\) and taking the supremum over
\(\gamma\ne0\) yields
\[
\mathrm{Lip}(\chi_\theta)
\le
\mathrm{Lip}(\phi_\theta).
\]

We now prove the reverse inequality. If
\(\mathrm{Lip}(\phi_\theta)=0\), then the claim is immediate.
Otherwise, since \(X/A\) is finite, there exist distinct
\(x,y\in X/A\) such that
\[
\operatorname{dist}\bigl(\phi_\theta(x),\phi_\theta(y)\bigr)
=
\mathrm{Lip}(\phi_\theta)\,
\overline d_1(x,y).
\]
Let \(\gamma\in K(X,A)\) be represented by the pair of singleton
diagrams \((x,y)\), where the singleton at the basepoint \([A]\) is
identified with the zero diagram. Then
\[
\rho(\gamma,0)
=
\overline d_1(x,y),
\qquad
\chi_\theta(\gamma)
=
\phi_\theta(x)\phi_\theta(y)^{-1}.
\]
Hence
\begin{align*}
\frac{
\operatorname{dist}\bigl(\chi_\theta(\gamma),1\bigr)
}{
\rho(\gamma,0)
}
&=
\frac{
\operatorname{dist}\bigl(\phi_\theta(x)\phi_\theta(y)^{-1},1\bigr)
}{
\overline d_1(x,y)
}
\\
&=
\frac{
\operatorname{dist}\bigl(\phi_\theta(x),\phi_\theta(y)\bigr)
}{
\overline d_1(x,y)
}
\\
&=
\mathrm{Lip}(\phi_\theta).
\end{align*}
Taking the supremum over \(\gamma\ne0\) gives
\[
\mathrm{Lip}(\chi_\theta)
\ge
\mathrm{Lip}(\phi_\theta).
\]
Combining the two inequalities proves the claim.
\end{proof}

Lemma~\ref{lem:char-lip-comparison} identifies
$\mathrm{Lip}(\chi_\theta)$ with the Lipschitz seminorm of the
associated phase function $\phi_\theta$ on
$(X/A,\overline d_1)$. In the graph-based discretizations used later,
this seminorm can be computed entirely from differences of phases on
the underlying interval space.

\subsection{Edgewise identity}\label{subsec:edgewise-bound}

Because $(X/A,\overline d_1)$ is finite, we may fix a connected weighted graph whose shortest-path metric coincides with $\overline d_1$. The next corollary shows that, with respect to such a graph model, the Lipschitz seminorm of a character $\chi_\theta$ is completely determined by the phase variation of its associated phase function along graph edges.

\begin{corollary}\label{cor:edgewise-char}
Let $\theta\in\mathbb T^{|X\setminus A|}$ and let $\phi_\theta$ be the
associated phase function. Suppose $H=(X/A,E)$ is a connected weighted
graph whose shortest-path metric coincides with $\overline d_1$. Then
\begin{equation}\label{eq:edgewise-character-lipschitz}
\mathrm{Lip}(\chi_\theta)
=
\max_{(u,v)\in E}
\frac{
\operatorname{dist}\bigl(\phi_\theta(u),\phi_\theta(v)\bigr)
}{
\overline d_1(u,v)
}.
\end{equation}
In particular, $\mathrm{Lip}(\chi_\theta)$ can be computed in
$O(|E|)$ time from the differences of phases $\phi_\theta$.
\end{corollary}

\begin{proof}
Let $f:X/A\to S^1$ be any function. Since the metric on $X/A$ is the
shortest-path metric of $H$, for any distinct $x,y\in X/A$ there exists
a path
\[
x=u_0,u_1,\dots,u_m=y
\]
such that
\[
\overline d_1(x,y)
=
\sum_{k=0}^{m-1}\overline d_1(u_k,u_{k+1}).
\]
By the triangle inequality for $\operatorname{dist}$,
\[
\operatorname{dist}\bigl(f(x),f(y)\bigr)
\le
\sum_{k=0}^{m-1}
\operatorname{dist}\bigl(f(u_k),f(u_{k+1})\bigr).
\]
Therefore,
\[
\frac{
\operatorname{dist}\bigl(f(x),f(y)\bigr)
}{
\overline d_1(x,y)
}
\le
\max_{0\le k<m}
\frac{
\operatorname{dist}\bigl(f(u_k),f(u_{k+1})\bigr)
}{
\overline d_1(u_k,u_{k+1})
}.
\]
Taking the supremum over all $x\ne y$ gives
\[
\mathrm{Lip}(f)
\le
\max_{(u,v)\in E}
\frac{
\operatorname{dist}\bigl(f(u),f(v)\bigr)
}{
\overline d_1(u,v)
}.
\]
The reverse inequality follows because every edge is a pair of points in
$X/A$. Hence
\[
\mathrm{Lip}(f)
=
\max_{(u,v)\in E}
\frac{
\operatorname{dist}\bigl(f(u),f(v)\bigr)
}{
\overline d_1(u,v)
}.
\]

Applying this identity to $f=\phi_\theta$ and using
Lemma~\ref{lem:char-lip-comparison} gives the stated formula for
$\mathrm{Lip}(\chi_\theta)$. The complexity statement follows
because the maximum is computed by a single pass over the edge set $E$.
\end{proof}

The corollary reduces the global Lipschitz seminorm of
$\chi_\theta$ on $(K(X,A),\rho)$ to a maximization over local phase
differences across graph edges. In the next section, the graph
Dirichlet energy (Definition~\ref{def:graph-dirichlet}) will be defined from the same edgewise phase
differences, but by taking a weighted sum of their squares over all
edges rather than their maximum.

\section{Heat flow and spectral multipliers on the dual}\label{sec:stable-multipliers}

The previous section identifies characters $\chi_\theta$ on the virtual diagram group $K(X,A)$ with the phase functions $\phi_\theta$ from Equation~\eqref{eq:phase-function} and shows in Corollary~\ref{cor:edgewise-char} that their Lipschitz seminorms with respect to the VPD metric are determined by differences of phases. We now pass to the dual side: characters \(\chi_\theta\) on \(K(X,A)\), parametrized by the dual torus \(\widehat{K(X,A)}\cong\mathbb T^{|X\setminus A|}\), will be weighted according to the phase variation of their associated circular coordinate systems --- in particular, that of the heat semigroup (After Definition~\ref{def:fourier-multiplier}) on the discrete Laplacian (Definition~\ref{def:graph-laplacian}) --- which plays the role of an energy associated to each mode (Definition~\ref{def:characters-dual}) $\chi_\theta$. In this way, kernels on $K(X,A)$ arise as Fourier-Stieltjes averages over the torus, in exact analogy with Bochner-type constructions on $\mathbb R^d$. In this section, we define the Laplacian via a \emph{Dirichlet form} (Definition~\ref{def:graph-dirichlet}) on the graph model of $X/A$, express it explicitly in terms of differences of phases, and then define the heat semigroup (After Definition~\ref{def:fourier-multiplier}) on it as a spectral multiplier to build translation-invariant kernels on virtual persistence diagrams.

\subsection{Graph Laplacian and Dirichlet energy}\label{subsec:dirichlet-symbol}

Our goal in this section is to construct, in the next subsection, translation-invariant positive definite kernels on the virtual diagram group $K(X,A)$ by integrating characters against a nonnegative weight on the dual torus.  By Bochner's theorem on locally compact abelian groups (Theorem~\ref{thm:bochner}), such kernels arise as Fourier-Stieltjes transforms of finite positive measures.  In the Euclidean prototype, the Gaussian kernel is obtained in exactly this way from the nonnegative eigenvalue function of the Laplacian.  We now recall that pattern and then adapt it to $X/A$.

\medskip

On $\mathbb R^d$ we start with the \emph{Dirichlet energy} of a smooth, compactly
supported function $f$,
\[
  \mathcal E(f)
  := \int_{\mathbb R^d} |\nabla f(x)|^2\,dx.
\]
Writing $\nabla f = (\partial_1 f,\dots,\partial_d f)$ and integrating by parts
componentwise (no boundary term because $f$ is compactly supported) gives
\begin{align*}
  \mathcal E(f)
  & = \sum_{j=1}^d \int_{\mathbb R^d} |\partial_j f(x)|^2\,dx \\
  & = \sum_{j=1}^d \int_{\mathbb R^d} \overline{f(x)}\,(-\partial_j^2 f)(x)\,dx \\
  & = \int_{\mathbb R^d} \overline{f(x)}\,(-\Delta f)(x)\,dx.
\end{align*}
Thus the quadratic form $f\mapsto\mathcal E(f)$ has (formal) generator
$L := -\Delta$ in the sense that
\[
  \mathcal E(f) = \langle f, Lf\rangle_{L^2(\mathbb R^d)}
  \qquad\text{for }f\in C_c^\infty(\mathbb R^d).
\]
We use the unitary Fourier transform
\[
  \widehat f(\xi)
  := (2\pi)^{-d/2} \int_{\mathbb R^d} e^{-i x\cdot\xi} f(x)\,dx.
\]
Differentiating under the integral sign yields
\[
  \widehat{\partial_j f}(\xi) = i\xi_j \widehat f(\xi),
  \qquad
  \widehat{\partial_j^2 f}(\xi) = -\xi_j^2 \widehat f(\xi),
\]
and hence
\[
  \widehat{Lf}(\xi)
  = \widehat{-\Delta f}(\xi)
  = \sum_{j=1}^d \xi_j^2 \widehat f(\xi)
  = |\xi|^2\,\widehat f(\xi).
\]
So the Fourier transform diagonalizes $L$ with eigenvalue function
$\lambda(\xi)=|\xi|^2$.

For $t>0$ we define the associated heat operator $e^{-tL}$ on such $f$ by
multiplying by the spectral multiplier $e^{-t|\xi|^2}$ (Definition~\ref{def:fourier-multiplier}) in Fourier variables:
\begin{equation}\label{eq:euclidean-heat-multiplier}
  \widehat{e^{-tL}f}(\xi) := e^{-t|\xi|^2}\,\widehat f(\xi).
\end{equation}
Let $\mathcal F^{-1}$ denote the inverse unitary Fourier transform, so that
\[
  \mathcal F^{-1}g(z)
  :=
  (2\pi)^{-d/2}
  \int_{\mathbb R^d} e^{i z\cdot\xi}g(\xi)\,d\xi.
\]
Define
\[
  p_t(z)
  :=
  (2\pi)^{-d/2}\mathcal F^{-1}\big(e^{-t|\xi|^2}\big)(z)
  =
  (2\pi)^{-d}
  \int_{\mathbb R^d} e^{i z\cdot\xi}e^{-t|\xi|^2}\,d\xi.
\]
The integral factors over coordinates into one-dimensional Gaussian integrals and gives the explicit formula
\[
  p_t(z)
  = \frac{1}{(4\pi t)^{d/2}} \exp\!\Big(-\frac{|z|^2}{4t}\Big).
\]

If we write $t = \sigma^2/2$ for some $\sigma>0$, then
\[
  p_t(z) = \frac{1}{(2\pi \sigma^2)^{d/2}} \exp\!\Big(-\frac{|z|^2}{2\sigma^2}\Big).
\]
Thus for each fixed $t>0$ the heat kernel $p_t$ is a centered Gaussian density with covariance $\sigma^2 I_d$, and up to the constant factor $(2\pi \sigma^2)^{-d/2}$ it has the same radial form as the Gaussian kernel $z\mapsto \exp(-|z|^2/(2\sigma^2))$.

By the convolution theorem for the unitary Fourier transform,
\[
  (e^{-tL}f)(x)
  = \mathcal F^{-1}\!\big(e^{-t|\xi|^2}\widehat f(\xi)\big)(x)
  = (p_t * f)(x)
  = \int_{\mathbb R^d} p_t(x-y)\,f(y)\,dy.
\]
In particular, $p_t(x-y)$ is the Fourier--Stieltjes transform of the finite positive measure on $\mathbb R^d$ with density $(2\pi)^{-d}e^{-t|\xi|^2}$ with respect to Lebesgue measure. By Bochner's theorem (Theorem~\ref{thm:bochner}), this kernel is translation-invariant and positive definite.

\medskip

We now recall the finite analogue on a weighted graph.  Let Let $V$ be a finite set and
let $H = (V,E,w)$ be a finite connected weighted graph with symmetric edge weights
$w_{uv} = w_{vu}>0$ for $\{u,v\}\in E$, extended by $w_{uv}=0$ if $\{u,v\}\notin E$.
The graph Laplacian (Definition~\ref{def:graph-laplacian}) is
\[
  (Lf)(u) := \sum_{v\in V} w_{uv}\,\bigl(f(u)-f(v)\bigr),
  \qquad u\in V,
\]
equivalently $L = D-A$ in matrix form, where $D_{uu}=\sum_{v} w_{uv}$ and
$A_{uv}=w_{uv}$.
The associated graph \emph{Dirichlet form} (Definition \ref{def:graph-dirichlet}) is
\[
  \mathcal E(f)
  := \frac12\sum_{u,v\in V} w_{uv}\,\bigl\lvert f(u)-f(v)\bigr\rvert^2
  = \langle f,Lf\rangle_{\ell^2(V)}.
\]
This is the discrete analogue of $\int \lvert \nabla f\rvert^2$: it is nonnegative,
vanishes exactly on constant functions, and increases as $f$ increases in frequency across edges with large weight.  In particular, $L$ is positive semidefinite
and all of its eigenvalues are nonnegative, so the heat semigroup (After Definition~\ref{def:fourier-multiplier}) $(e^{-tL})_{t>0}$ induced by the graph representation
and its associated kernels are well defined.

\medskip

Let \(V:=X/A\). Since \(V\) is finite, the metric
\(\overline d_1\) determines a complete graph with edge length
\(\overline d_1(u,v)\) between distinct vertices \(u,v\in V\). More
generally, we may work with any connected graph \(H=(V,E)\) whose
edges, assigned the lengths \(\overline d_1(u,v)\), induce
\(\overline d_1\) as the shortest-path metric. On the chosen edge set
\(E\), we fix symmetric positive weights \(w_{uv}=w_{vu}>0\),
interpreted as edge conductances, defining the graph Laplacian and
Dirichlet form. The intrinsic choice
\(w_{uv}=\overline d_1(u,v)\) is permitted, but we do not require the
conductances to equal the metric edge lengths. Allowing the conductances
to vary independently permits different weighted Laplacians, and hence
different diffusion dynamics, on the same metric space
\((X/A,\overline d_1)\).

For each $\theta\in\mathbb T^{|X\setminus A|}$, define \( \lambda(\theta) := \mathcal E(\phi_\theta) \) using the Dirichlet form from Equation~\eqref{eq:dirichlet-form}. Explicitly,
\begin{equation}\label{eq:lambda-dirichlet}
  \lambda(\theta)
  =
  \frac12\sum_{u,v\in V}
  w_{uv}\,\bigl|\phi_\theta(u)-\phi_\theta(v)\bigr|^2.
\end{equation}
Since
\[
  |z-z'|^2
  =
  2\bigl(1-\cos(\operatorname{dist}(z,z'))\bigr)
\]
for $z,z'\in S^1$, and since $w_{uv}=0$ whenever $\{u,v\}\notin E$,
the \emph{graph Dirichlet form} can be rewritten as
\begin{equation}\label{eq:lambda-edge-cosine}
  \lambda(\theta)
  =
  2\sum_{\{u,v\}\in E}
  w_{uv}\,
  \Bigl(1-\cos\bigl(\operatorname{dist}(\phi_\theta(u),\phi_\theta(v))\bigr)\Bigr).
\end{equation}
The function $\lambda:\mathbb T^{|X\setminus A|}\to[0,\infty)$ is
continuous, and $\lambda(0)=0$.

\begin{figure}[ht]
    \centering
    \includegraphics[width=0.85\linewidth]{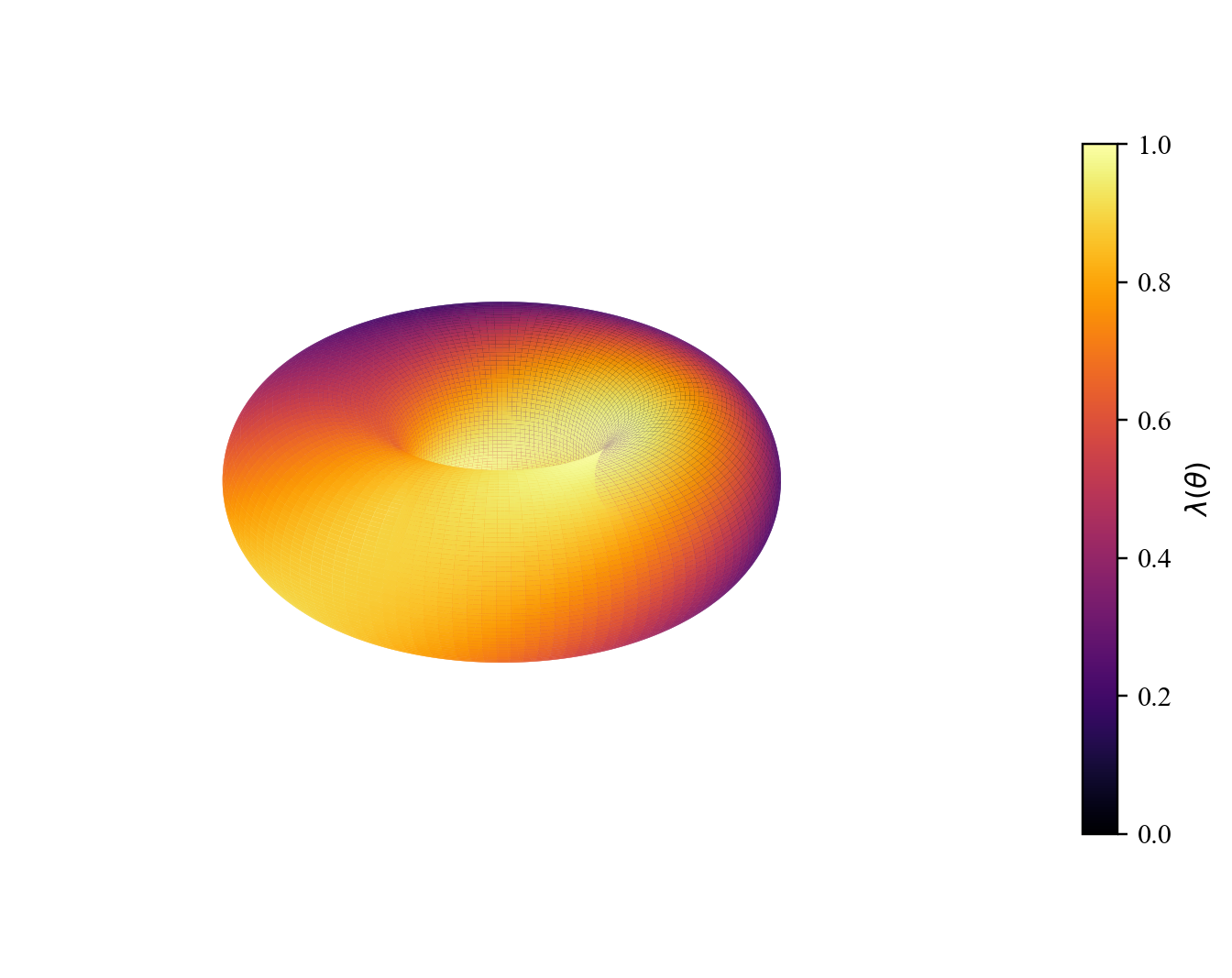}
    \caption{Heatmap of the Laplacian applied to the character embedding shown in Fig.~\ref{fig:torus-embedding}, visualized over the Pontryagin dual torus.}
    \label{fig:torus-laplacian}
\end{figure}

The quantity $\lambda(\theta)$ assigns to the character
$\chi_\theta$ the graph Dirichlet energy of its associated phase
function $\phi_\theta$. It therefore measures the total weighted phase
variation of $\phi_\theta$ across
$(X/A,\overline d_1)$ as a weighted (complete) graph.

In the next subsection, we use the heat multipliers
$\theta\mapsto e^{-t\lambda(\theta)}$ and their Fourier-Stieltjes
transforms to build translation-invariant heat kernels on $K(X,A)$ that
play the role of Gaussian kernels on virtual persistence diagrams.

\subsection{Phase-energy comparison}\label{subsec:phase-energy}

We now compare the spectral quantity $\lambda(\theta)$ with the
Lipschitz seminorm $\mathrm{Lip}(\chi_\theta)$.

For each edge $\{u,v\}\in E$, write
\begin{equation}\label{eq:edge-length-phase-difference}
\ell_{uv} := \overline d_1(u,v),
\qquad
\Delta_{uv}
:=
\operatorname{dist}\bigl(\phi_\theta(u),\phi_\theta(v)\bigr)
\in[0,\pi].
\end{equation}
Then
\begin{equation}\label{eq:lambda-delta}
\lambda(\theta)
=
2\sum_{\{u,v\}\in E}
w_{uv}\,\bigl(1-\cos(\Delta_{uv})\bigr),
\end{equation}
and
\begin{equation}\label{eq:lip-delta-edge-max}
\mathrm{Lip}(\chi_\theta)
=
\max_{\{u,v\}\in E}
\frac{\Delta_{uv}}{\ell_{uv}}.
\end{equation}

The edge weights $w_{uv}$ need not equal the metric edge lengths $\ell_{uv}$. The lengths $\ell_{uv}$ come from the metric $\overline d_1$, while the weights $w_{uv}$ determine the weighted graph Laplacian and Dirichlet energy. The phase differences $\Delta_{uv}$ describe the variation of $\phi_\theta$ across the fixed metric space $(X/A,\overline d_1)$. The lengths $\ell_{uv}$ measure this variation relative to the metric geometry, while the weights $w_{uv}$ determine the Laplacian dynamics of the same phase variation. Thus, varying $w_{uv}$ permits different Laplacian dynamics on the same metric space. The weights $w_{uv}$ may be interpreted as conductances.

We also write
\begin{align}\label{eq:edge-weight-extrema}
w_{\min}
&:=
\min\{w_{uv}: \{u,v\}\in E\}, \\
w_{\max}
&:=
\max\{w_{uv}: \{u,v\}\in E\}, \\
M &:= |E|,
\end{align}
and
\begin{equation}\label{eq:edge-length-extrema}
d_{\min}
:=
\min\{\overline d_1(u,v): \{u,v\}\in E\},
\qquad
d_{\max}
:=
\max\{\overline d_1(u,v): \{u,v\}\in E\}.
\end{equation}

\begin{lemma}\label{lem:lambda-vs-L}
For every $\theta\in\mathbb T^{|X\setminus A|}$,
\[
\frac{4\,w_{\min}\,d_{\min}^2}{\pi^2}\,
\mathrm{Lip}(\chi_\theta)^2
\le
\lambda(\theta)
\le
w_{\max}\,M\,d_{\max}^2\,
\mathrm{Lip}(\chi_\theta)^2.
\]
\end{lemma}

\begin{proof}
Fix $\theta\in\mathbb T^{|X\setminus A|}$. By
Corollary~\ref{cor:edgewise-char},
\[
\mathrm{Lip}(\chi_\theta)
=
\max_{\{u,v\}\in E}
\frac{\Delta_{uv}}{\ell_{uv}}.
\]
Choose an edge $\{u_*,v_*\}\in E$ such that
\[
\mathrm{Lip}(\chi_\theta)
=
\frac{\Delta_{u_*v_*}}{\ell_{u_*v_*}}.
\]
If
\[
\Delta_{\max}
:=
\max_{\{u,v\}\in E}\Delta_{uv},
\]
then
\[
\Delta_{\max}
\ge
\Delta_{u_*v_*}
=
\ell_{u_*v_*}\,\mathrm{Lip}(\chi_\theta)
\ge
d_{\min}\,\mathrm{Lip}(\chi_\theta).
\]
On the other hand, for every edge $\{u,v\}\in E$,
\[
\Delta_{uv}
\le
\ell_{uv}\,\mathrm{Lip}(\chi_\theta)
\le
d_{\max}\,\mathrm{Lip}(\chi_\theta),
\]
so \( \Delta_{\max} \le d_{\max}\,\mathrm{Lip}(\chi_\theta). \)

By definition,
\[
\lambda(\theta)
=
2\sum_{\{u,v\}\in E}
w_{uv}\,\bigl(1-\cos(\Delta_{uv})\bigr).
\]
For $t\in[0,\pi]$,
\[
\frac{2}{\pi^2}\,t^2
\le
1-\cos t
\le
\frac12\,t^2.
\]
Applying these inequalities edgewise gives
\[
\frac{4w_{\min}}{\pi^2}
\sum_{\{u,v\}\in E}\Delta_{uv}^2
\le
\lambda(\theta)
\le
w_{\max}
\sum_{\{u,v\}\in E}\Delta_{uv}^2.
\]
Since \( \Delta_{uv}\le\Delta_{\max} \)
for every edge,
\[
\Delta_{\max}^2
\le
\sum_{\{u,v\}\in E}\Delta_{uv}^2
\le
M\,\Delta_{\max}^2.
\]
Therefore,
\[
\frac{4w_{\min}}{\pi^2}\,\Delta_{\max}^2
\le
\lambda(\theta)
\le
w_{\max}\,M\,\Delta_{\max}^2.
\]
Combining this with the bounds on $\Delta_{\max}$ yields
\[
\frac{4\,w_{\min}\,d_{\min}^2}{\pi^2}\,
\mathrm{Lip}(\chi_\theta)^2
\le
\lambda(\theta)
\le
w_{\max}\,M\,d_{\max}^2\,
\mathrm{Lip}(\chi_\theta)^2.
\]
\end{proof}

Lemma~\ref{lem:lambda-vs-L} gives a two-sided graph-dependent
comparison between $\lambda(\theta)$ and
$\mathrm{Lip}(\chi_\theta)^2$. Thus, the heat multiplier
$e^{-t\lambda(\theta)}$ suppresses the high-frequency characters with respect to the VPD metric. When the edge
lengths and weights have bounded aspect ratios, the comparison constants
are uniform, so the Dirichlet energy (Equation~\eqref{eq:lambda-dirichlet}) is uniformly comparable to the
squared VPD Lipschitz seminorm.

\subsection{Lipschitz bounds for heat-weighted Fourier-Stieltjes transforms}\label{subsec:lipschitz-heat}

We now use the heat multipliers $\theta\mapsto e^{-t\lambda(\theta)}$
to construct translation-invariant Fourier--Stieltjes kernels on
$K(X,A)$ and study their Lipschitz behavior with respect to the VPD
metric. Since $\lambda(\theta)$ is comparable to
$\mathrm{Lip}(\chi_\theta)^2$, the heat weighting suppresses
Fourier modes (Definition~\ref{def:characters-dual}) whose associated phase functions exhibit unstable phase
variation across nearby points of $(X/A,\overline d_1)$.

\begin{definition}\label{def:heat-measure}
For $t>0$ define the \emph{heat measure} on $\mathbb T^{|X\setminus A|}$ by
\begin{equation}\label{eq:heat-measure}
  d\nu_t(\theta)\ :=\ e^{-t\lambda(\theta)}\,d\mu(\theta),
  \qquad \theta\in\mathbb T^{|X\setminus A|},
\end{equation}
where $\mu$ is the normalized Haar measure
on $\mathbb T^{|X\setminus A|}$.
The corresponding Fourier--Stieltjes transform is
\begin{equation}\label{eq:heat-fourier-stieltjes}
  F_{\nu_t}(\alpha)
  \ :=\ \int_{\mathbb T^{|X\setminus A|}}\chi_\theta(\alpha)\,e^{-t\lambda(\theta)}\,d\mu(\theta),
  \qquad \alpha\in K(X,A),
\end{equation}
and, as in the Fourier--Stieltjes construction (Equation~\eqref{eq:fourier-stieltjes}), it induces a translation-invariant
positive definite kernel on $K(X,A)$ by
$k_{\nu_t}(\alpha,\beta):=F_{\nu_t}(\alpha-\beta)$.
\end{definition}

Thus $t\mapsto\nu_t$ is the heat flow on the dual torus, and
$t\mapsto F_{\nu_t}$ is the induced family of translation-invariant kernels
on the virtual diagram group.

\medskip

We first record a general Lipschitz estimate for Fourier--Stieltjes transforms with respect to the VPD metric.

\begin{figure}[ht]
    \centering

    \begin{minipage}{0.32\linewidth}
        \centering
        \includegraphics[width=\linewidth]{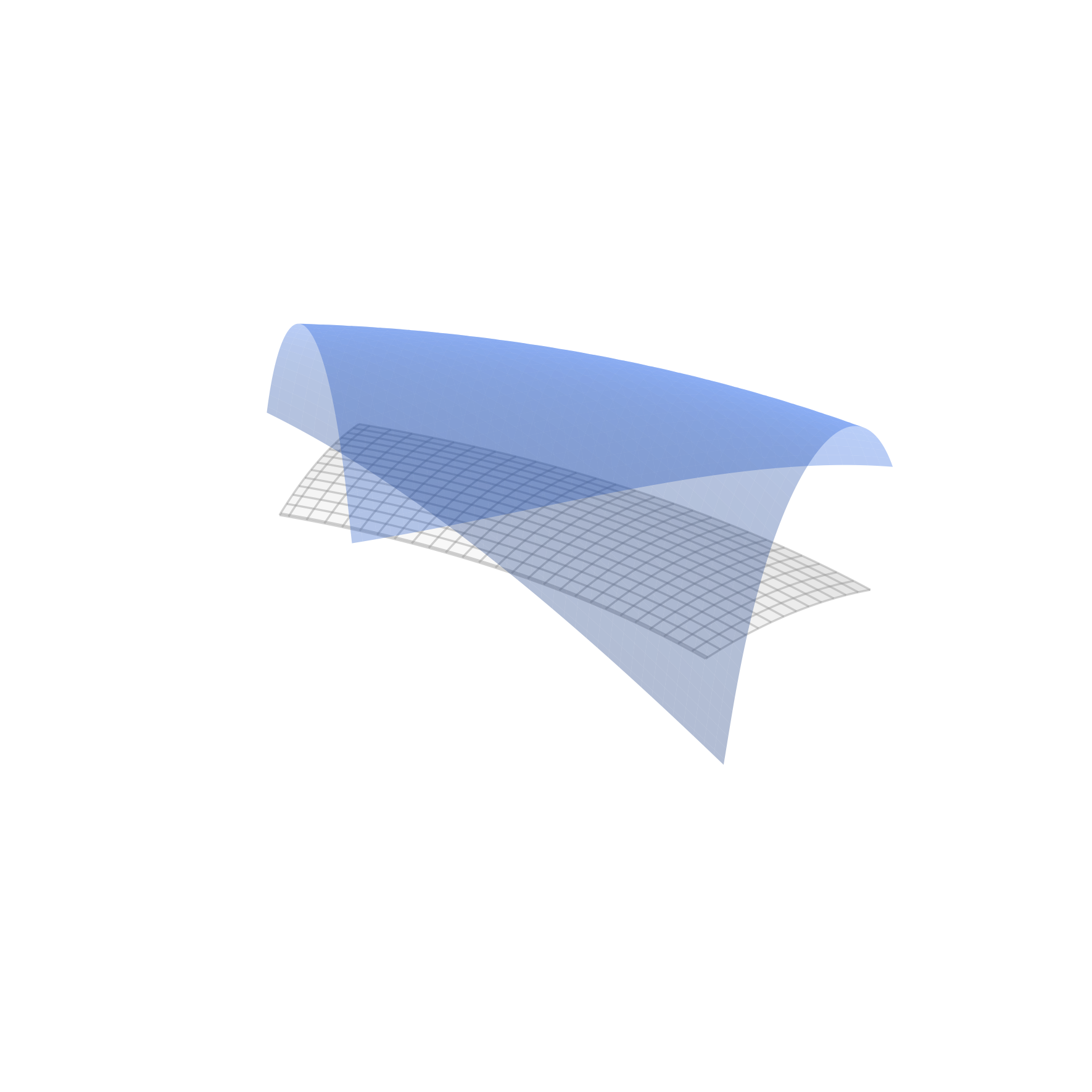}
    \end{minipage}\hfill
    \begin{minipage}{0.32\linewidth}
        \centering
        \includegraphics[width=\linewidth]{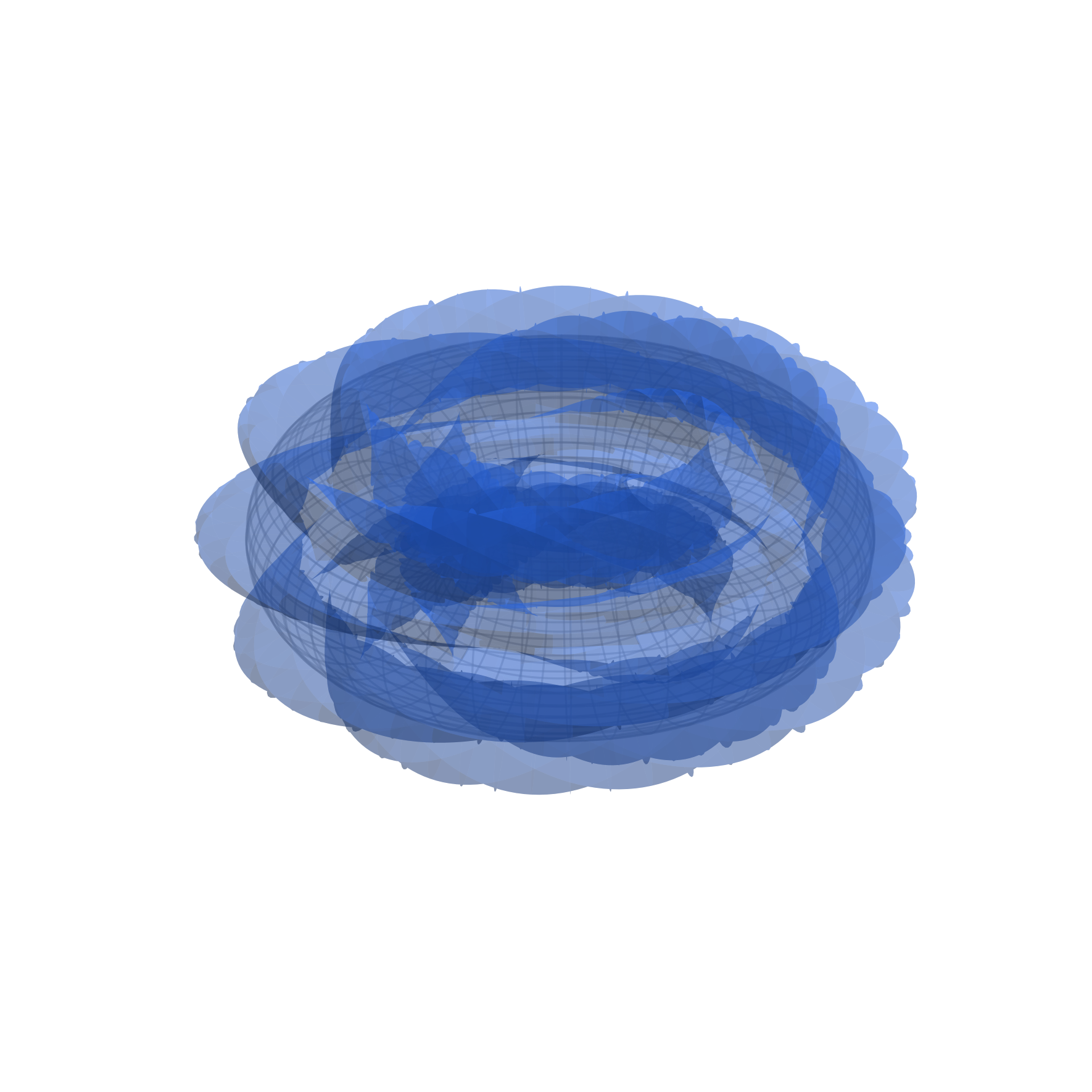}
    \end{minipage}\hfill
    \begin{minipage}{0.32\linewidth}
        \centering
        \includegraphics[width=\linewidth]{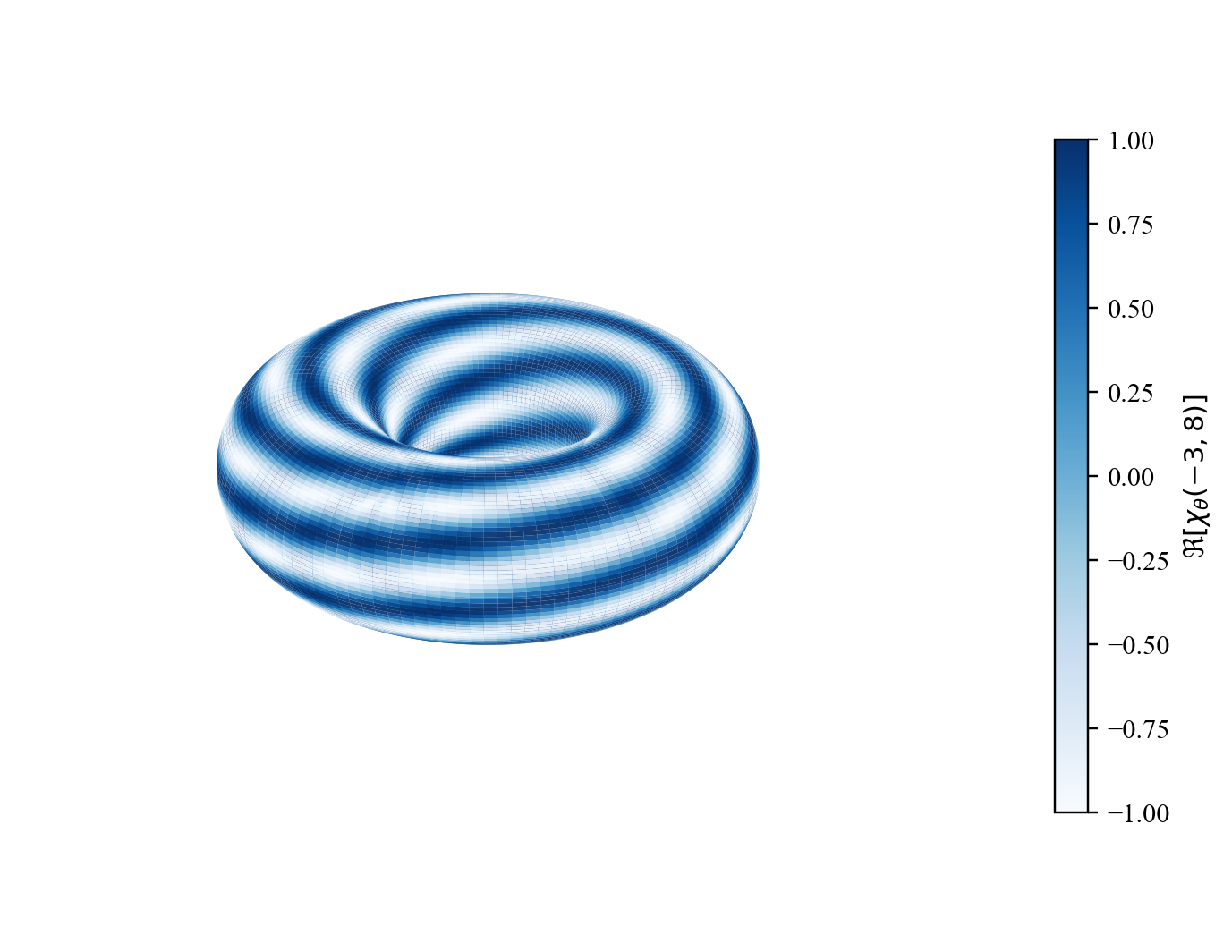}
    \end{minipage}

    \vspace{0.8em}

    \begin{minipage}{0.32\linewidth}
        \centering
        \includegraphics[width=\linewidth]{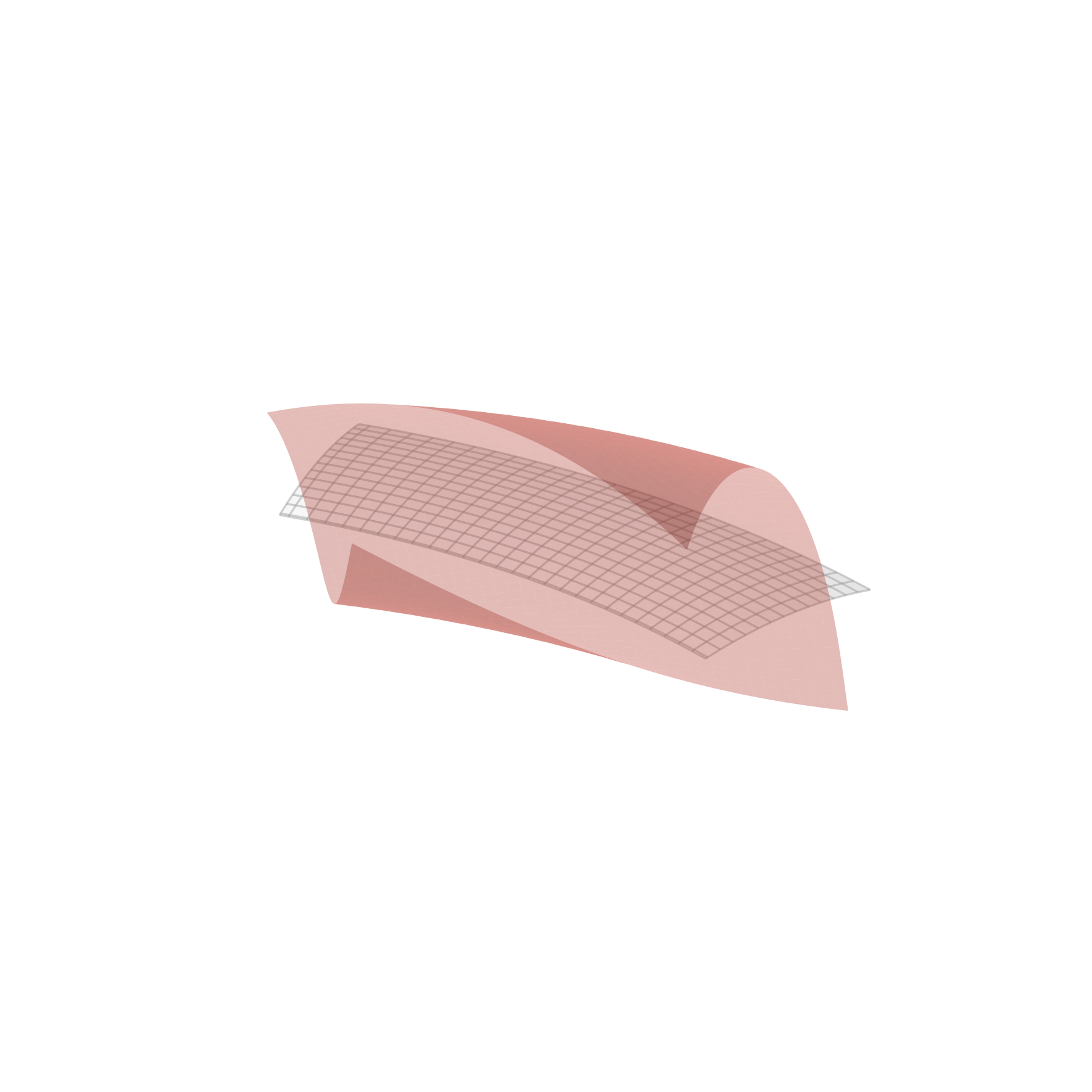}
    \end{minipage}\hfill
    \begin{minipage}{0.32\linewidth}
        \centering
        \includegraphics[width=\linewidth]{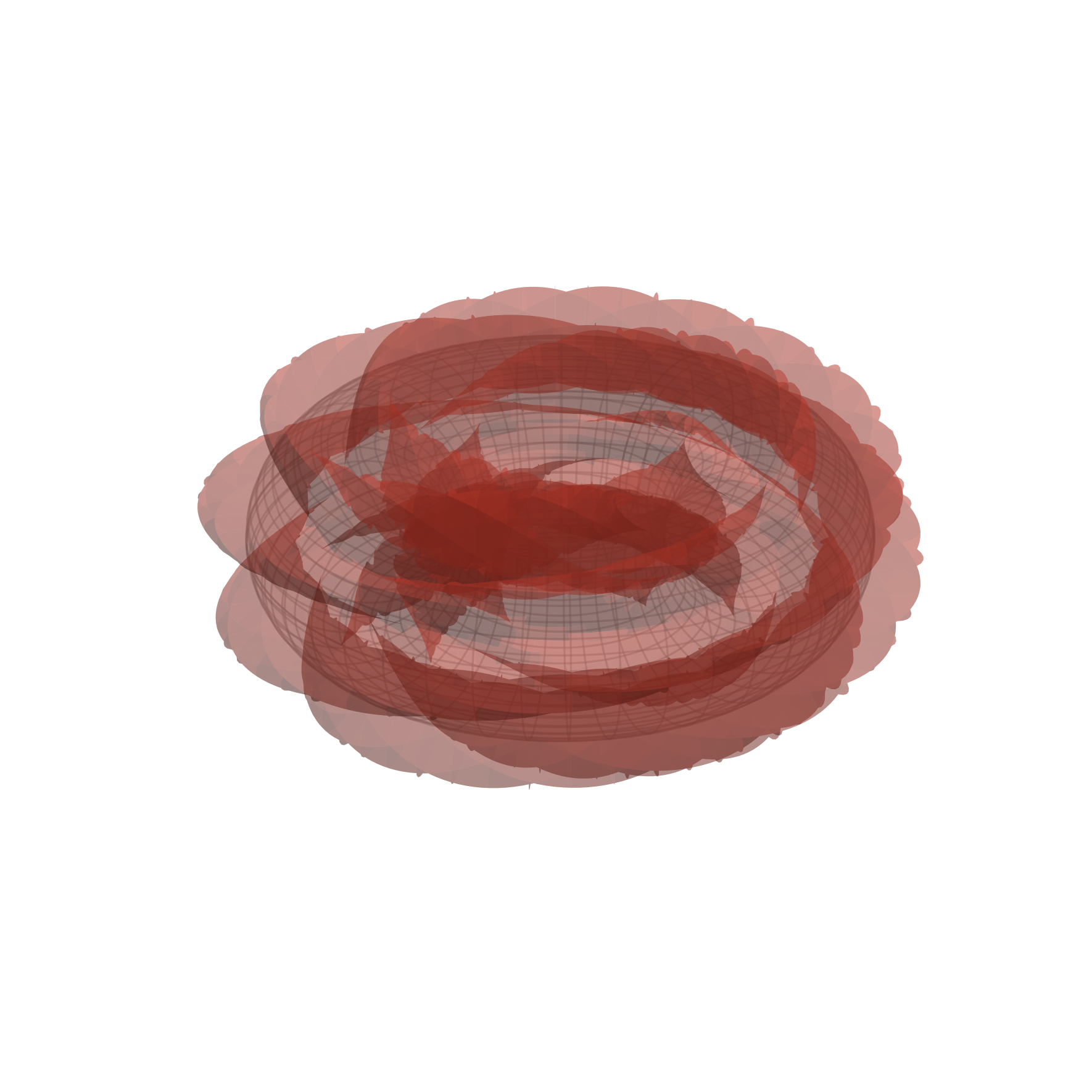}
    \end{minipage}\hfill
    \begin{minipage}{0.32\linewidth}
        \centering
        \includegraphics[width=\linewidth]{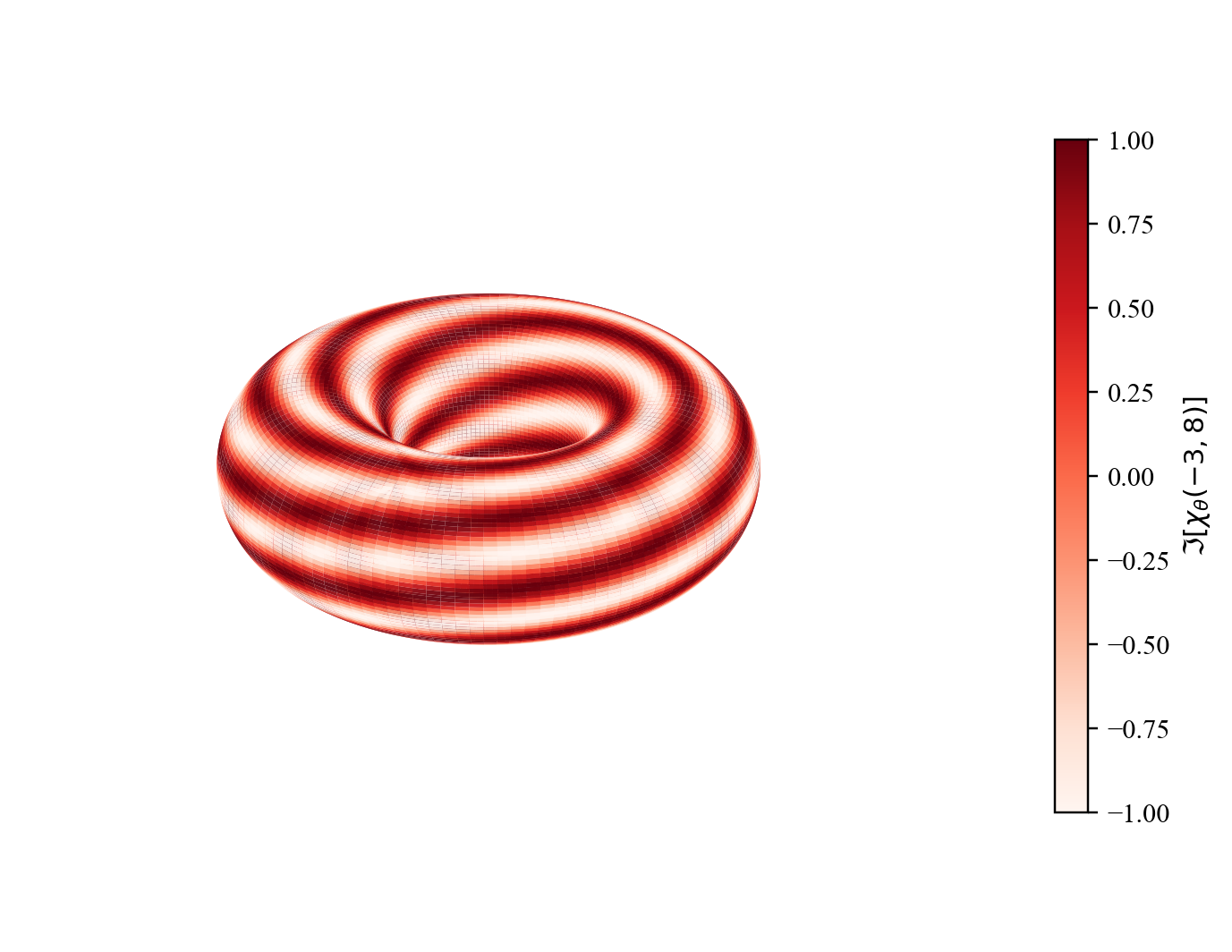}
    \end{minipage}

    \caption{
    Fourier transform of the virtual persistence diagram from Fig.~\ref{fig:toy-example}, evaluated over the Pontryagin dual torus from Fig.~\ref{fig:torus-embedding}.
    Columns show patch, surface, and heatmap visualizations; rows show the real and imaginary parts, respectively.}
    \label{fig:fourier-vpd}
\end{figure}

\begin{lemma}\label{lem:multiplier-stability}
Let $\nu$ be a finite complex Borel measure on
$\mathbb T^{|X\setminus A|}$ and let $F_\nu$ be its
Fourier--Stieltjes transform as in Definition~\ref{def:FS}. Then
\[
  \mathrm{Lip}(F_\nu)
  \le
  \int_{\mathbb T^{|X\setminus A|}}
  \mathrm{Lip}(\chi_\theta)\,d|\nu|(\theta).
\]
\end{lemma}

\begin{proof}
Let $\alpha,\beta\in K(X,A)$ and set $\gamma:=\alpha-\beta$. Since
$\chi_\theta$ is a character,
\[
F_\nu(\alpha)-F_\nu(\beta)
=
\int_{\mathbb T^{|X\setminus A|}}
\chi_\theta(\beta)\bigl(\chi_\theta(\gamma)-1\bigr)\,d\nu(\theta).
\]
Using $|\chi_\theta(\beta)|=1$ and the total-variation inequality,
\[
\bigl|F_\nu(\alpha)-F_\nu(\beta)\bigr|
\le
\int_{\mathbb T^{|X\setminus A|}}
\bigl|\chi_\theta(\gamma)-1\bigr|\,d|\nu|(\theta).
\]
For points on $S^1$, the chordal distance is bounded above by the geodesic distance, so
\[
\bigl|\chi_\theta(\gamma)-1\bigr|
\le
\operatorname{dist}\bigl(\chi_\theta(\gamma),1\bigr).
\]
By Lemma~\ref{lem:char-lip},
\[
\operatorname{dist}\bigl(\chi_\theta(\gamma),1\bigr)
\le
\mathrm{Lip}(\chi_\theta)\,\rho(\gamma,0).
\]
Since $\rho$ is translation-invariant, \( \rho(\gamma,0)=\rho(\alpha,\beta). \) Therefore,
\[
\bigl|F_\nu(\alpha)-F_\nu(\beta)\bigr|
\le
\rho(\alpha,\beta)
\int_{\mathbb T^{|X\setminus A|}}
\mathrm{Lip}(\chi_\theta)\,d|\nu|(\theta).
\]
Dividing by $\rho(\alpha,\beta)$ and taking the supremum over
$\alpha\ne\beta$ proves the claim.
\end{proof}

Specializing to the heat measures $\nu_t$ gives the desired Lipschitz
control for the heat-weighted transforms.

\begin{figure}[ht]
    \centering

    \begin{minipage}{0.30\linewidth}
        \centering
        \includegraphics[width=\linewidth]{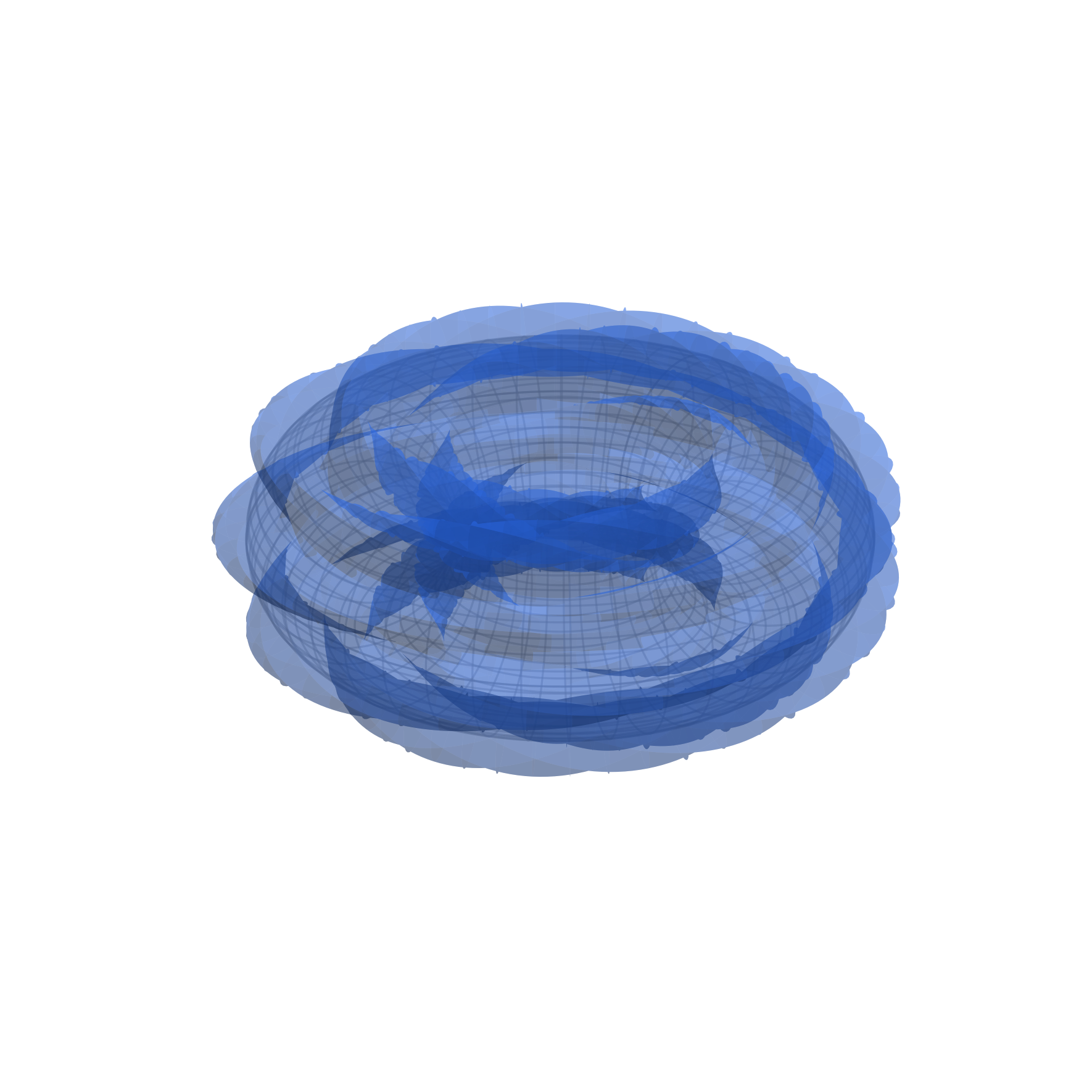}
    \end{minipage}\hfill
    \begin{minipage}{0.30\linewidth}
        \centering
        \includegraphics[width=\linewidth]{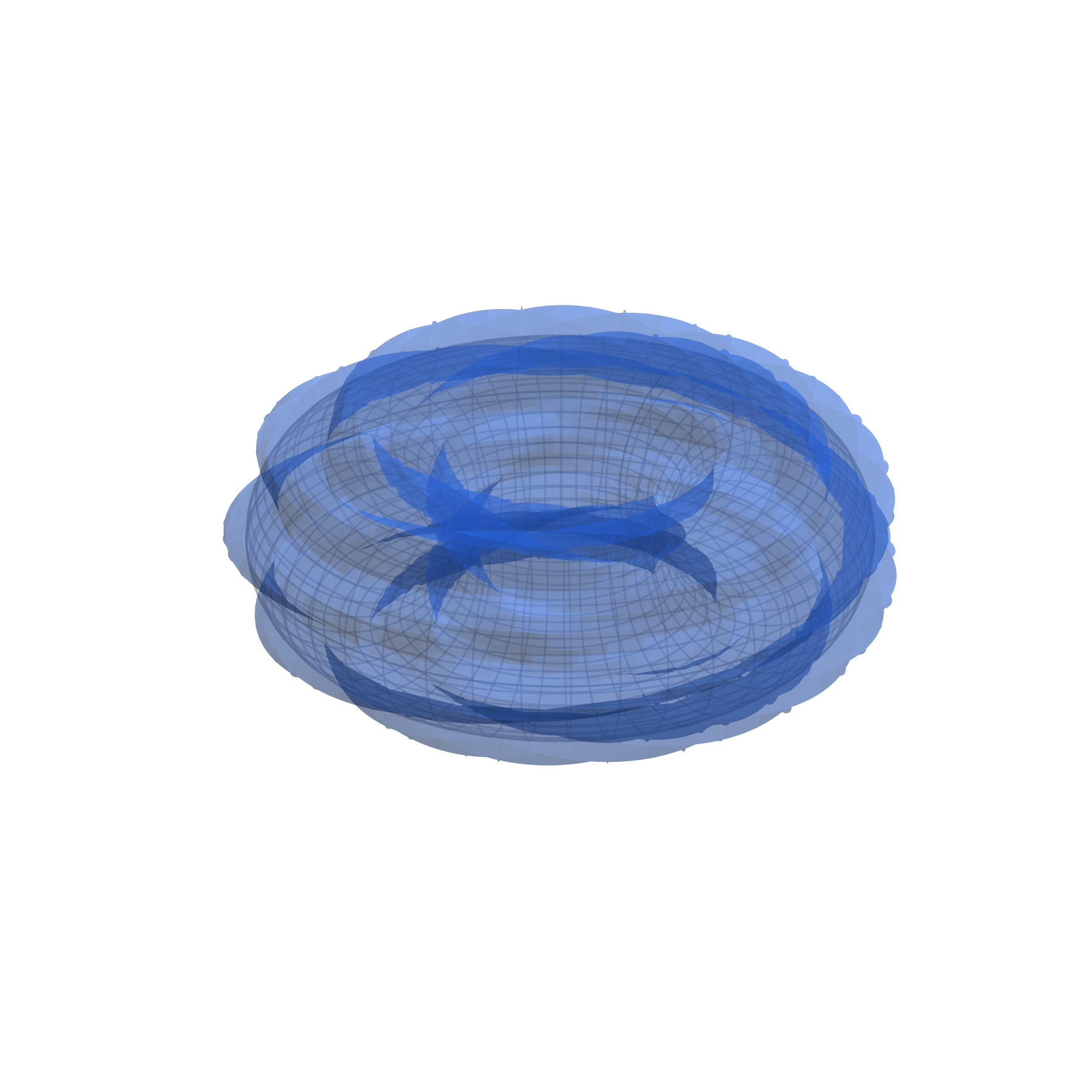}
    \end{minipage}\hfill
    \begin{minipage}{0.30\linewidth}
        \centering
        \includegraphics[width=\linewidth]{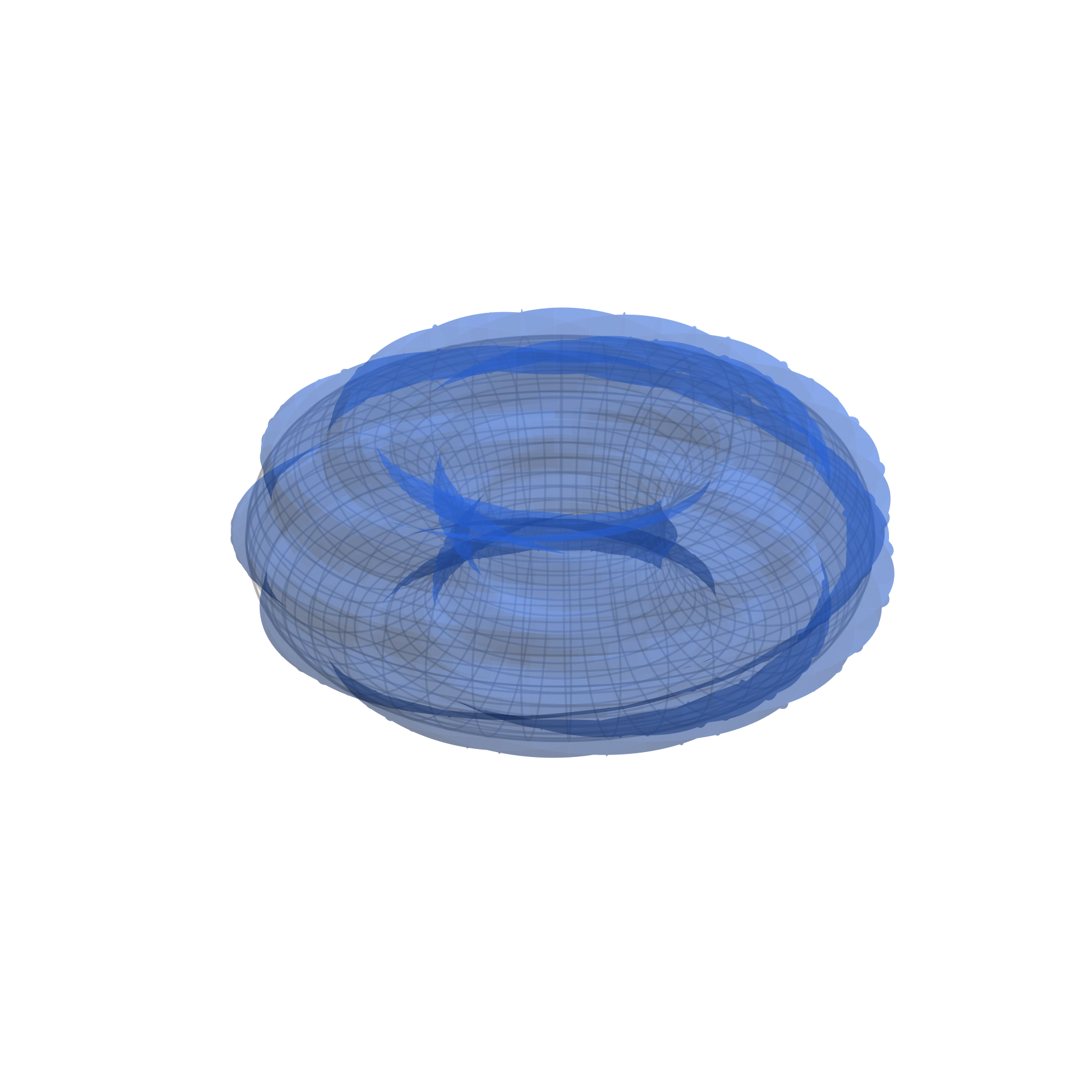}
    \end{minipage}

    \vspace{0.8em}

    \begin{minipage}{0.30\linewidth}
        \centering
        \includegraphics[width=\linewidth]{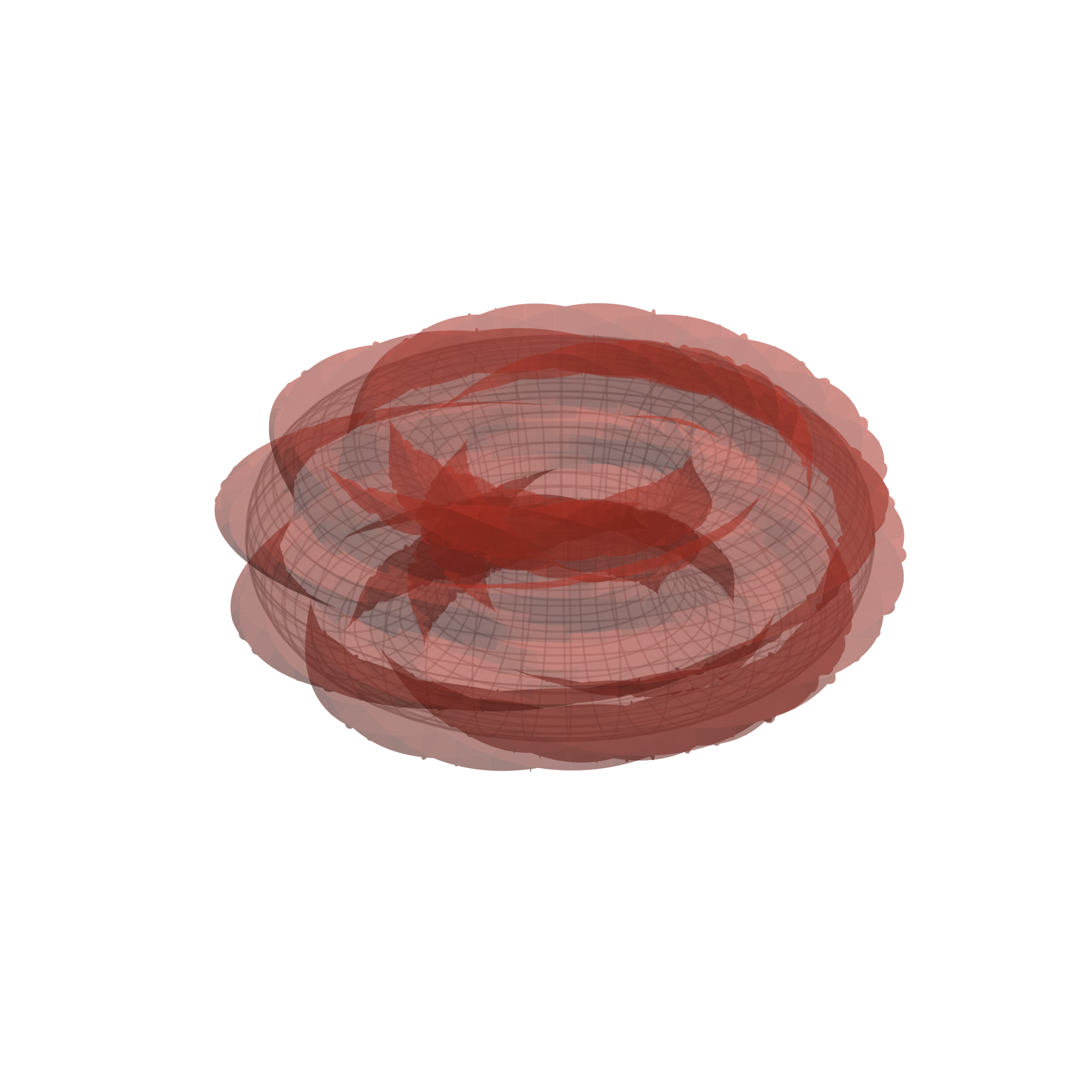}
    \end{minipage}\hfill
    \begin{minipage}{0.30\linewidth}
        \centering
        \includegraphics[width=\linewidth]{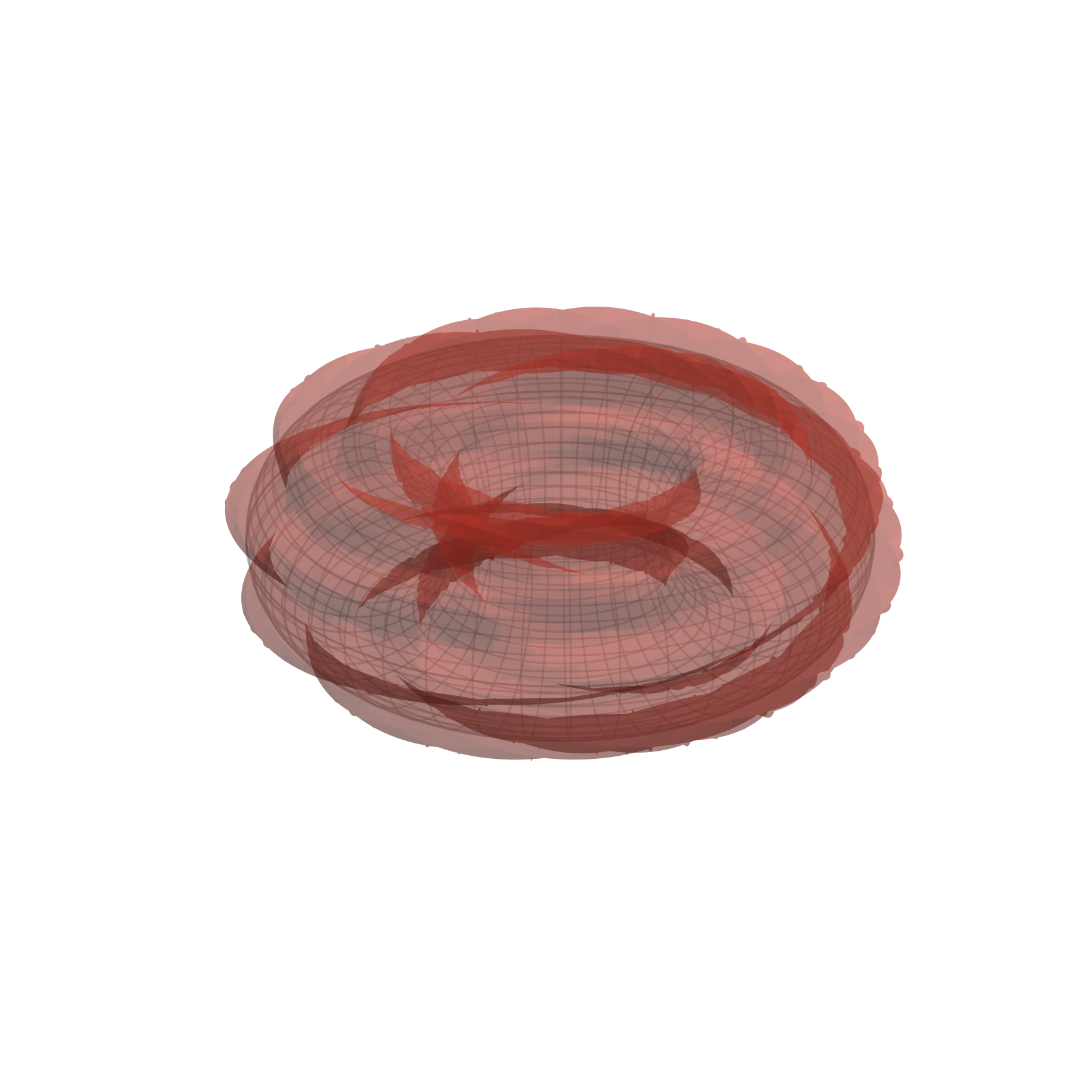}
    \end{minipage}\hfill
    \begin{minipage}{0.30\linewidth}
        \centering
        \includegraphics[width=\linewidth]{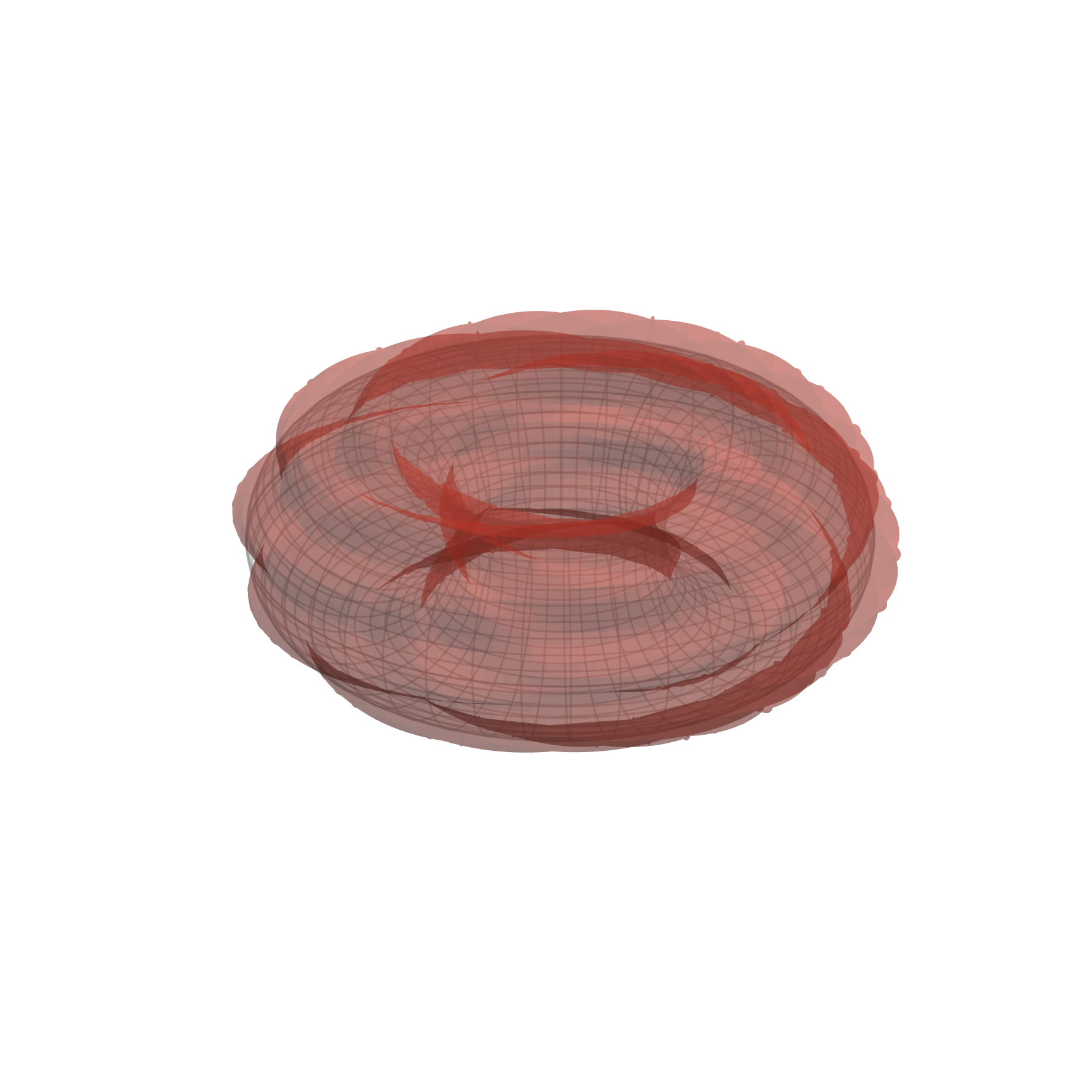}
    \end{minipage}

    \caption{
    Heat flow of the Fourier transform from Fig.~\ref{fig:fourier-vpd} on the Pontryagin dual torus from Fig.~\ref{fig:torus-embedding}, originating from the virtual persistence diagram in Fig.~\ref{fig:toy-example}. Columns correspond to times $t=0.0,\ 0.0625,$ and $0.125$. The two rows show the real and imaginary parts as surfaces.
    }
    \label{fig:heat-kernel-surfaces}
\end{figure}

\begin{figure}[ht]
    \centering

    \begin{minipage}{0.30\linewidth}
        \centering
        \includegraphics[width=\linewidth]{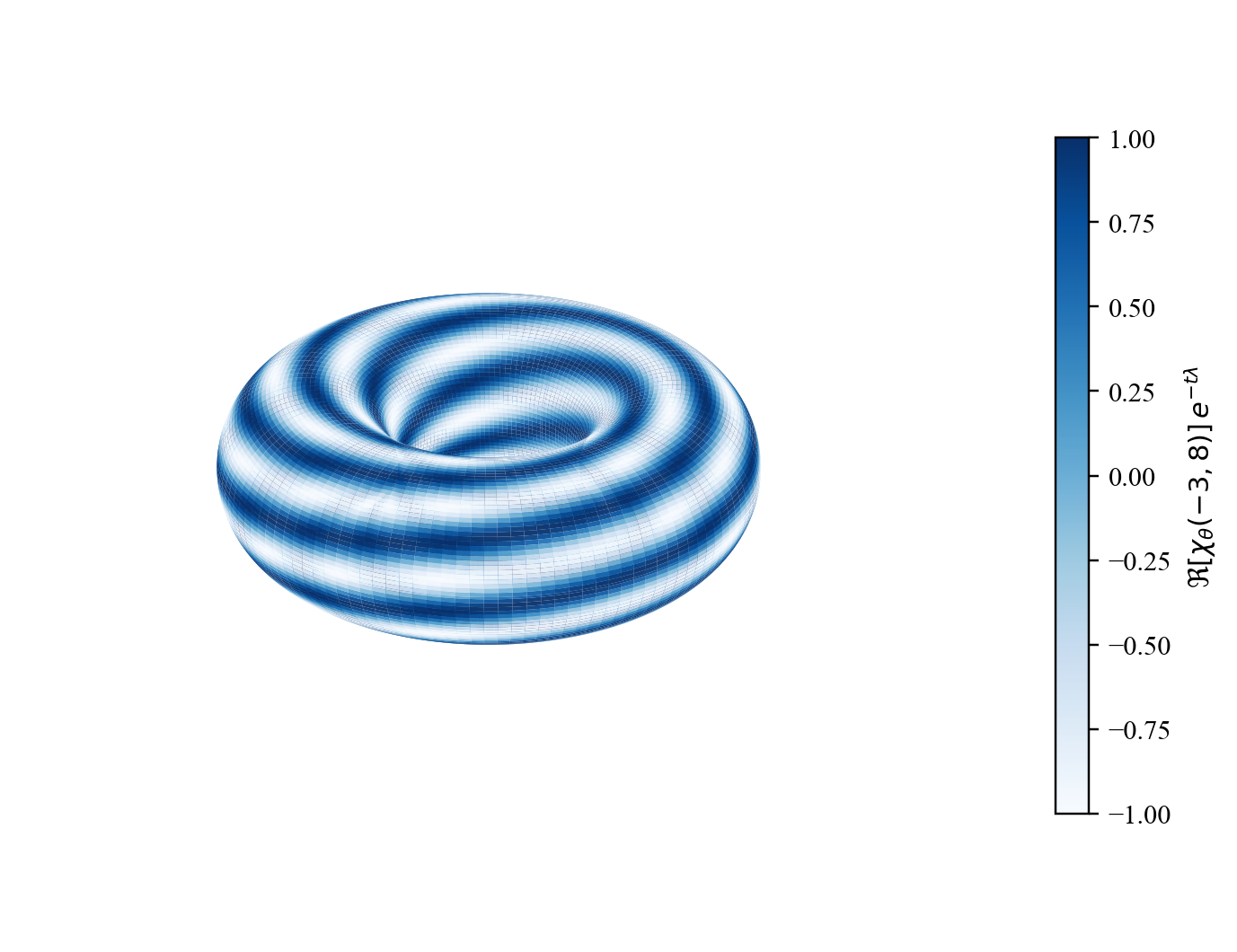}
    \end{minipage}\hfill
    \begin{minipage}{0.30\linewidth}
        \centering
        \includegraphics[width=\linewidth]{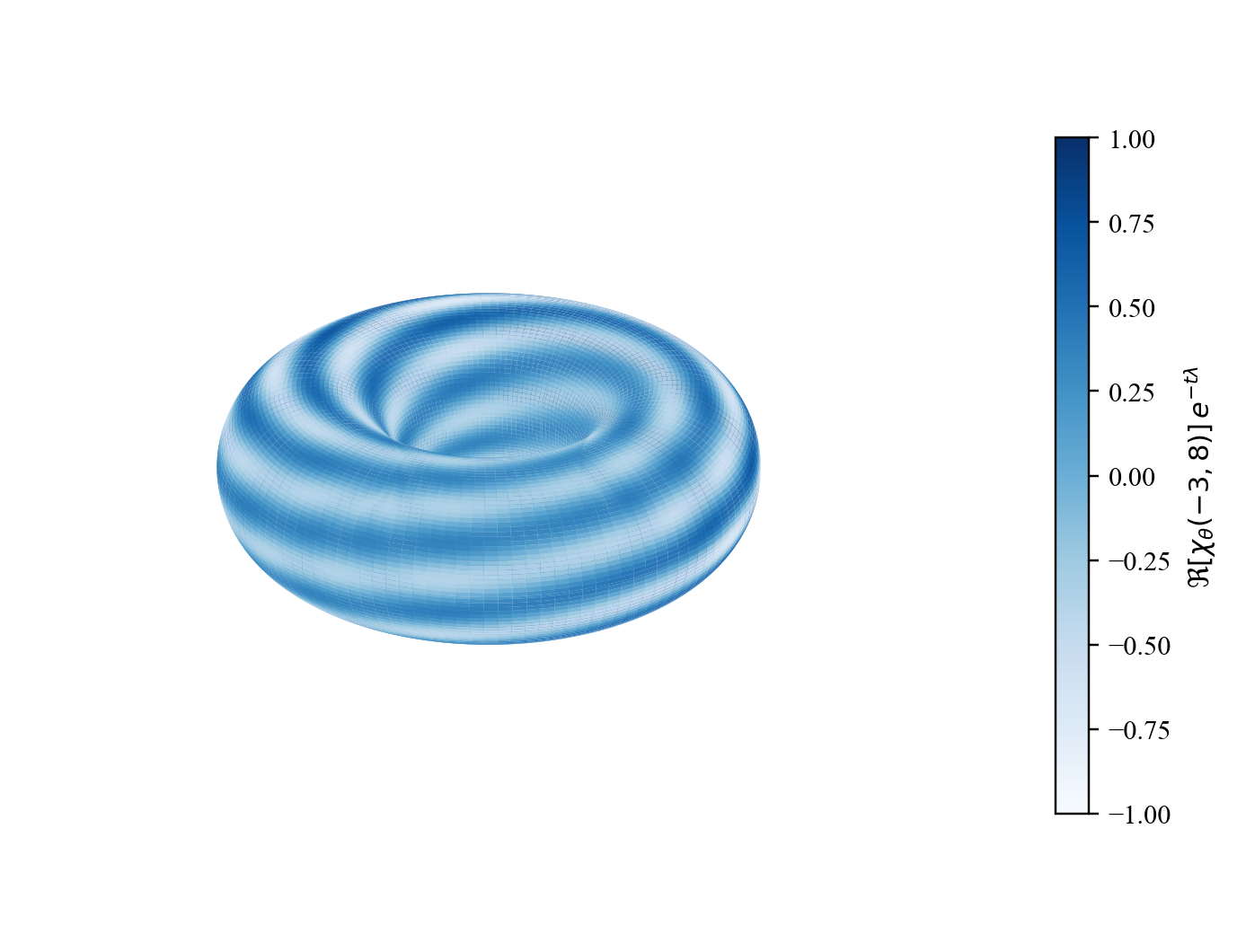}
    \end{minipage}\hfill
    \begin{minipage}{0.30\linewidth}
        \centering
        \includegraphics[width=\linewidth]{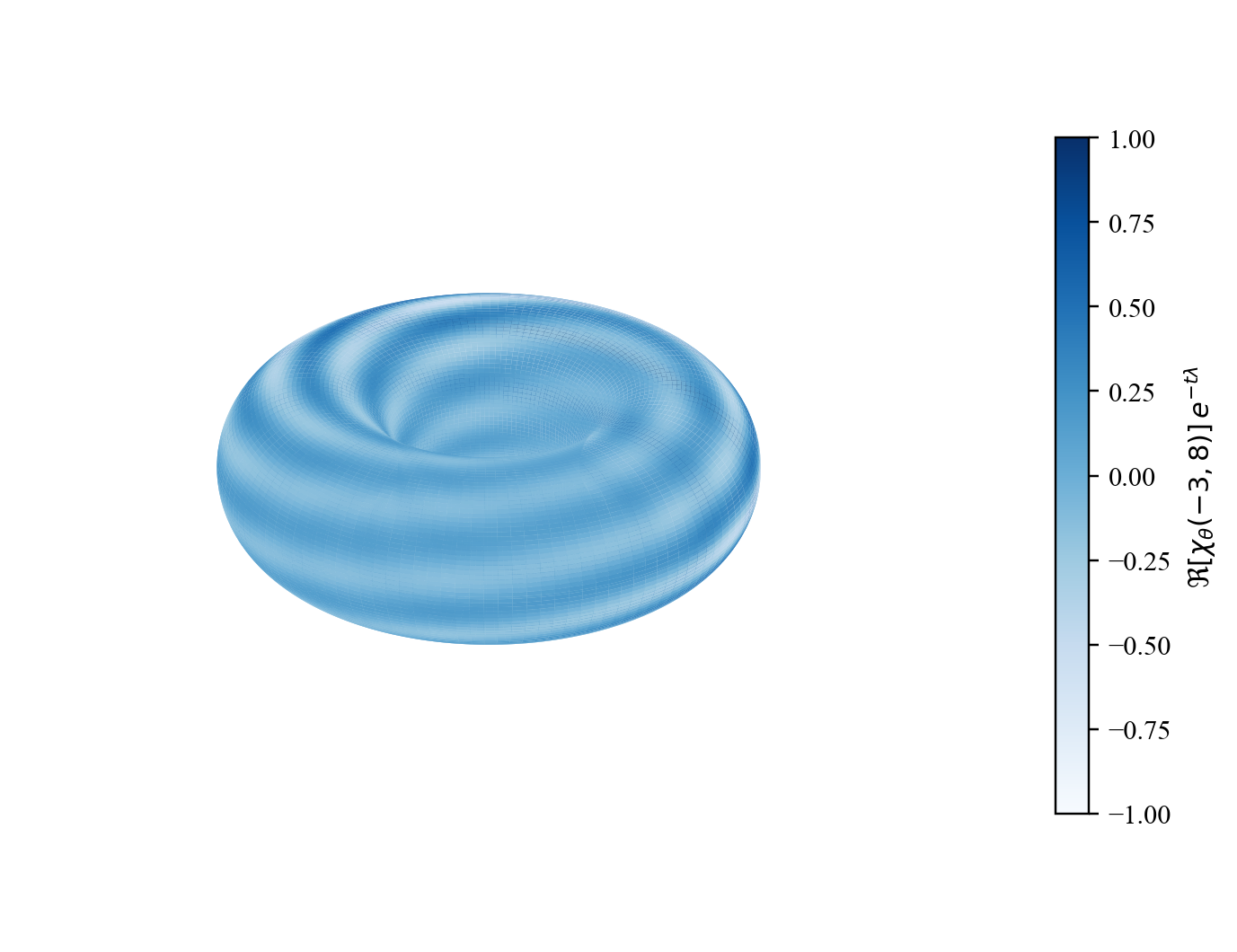}
    \end{minipage}

    \vspace{0.8em}

    \begin{minipage}{0.30\linewidth}
        \centering
        \includegraphics[width=\linewidth]{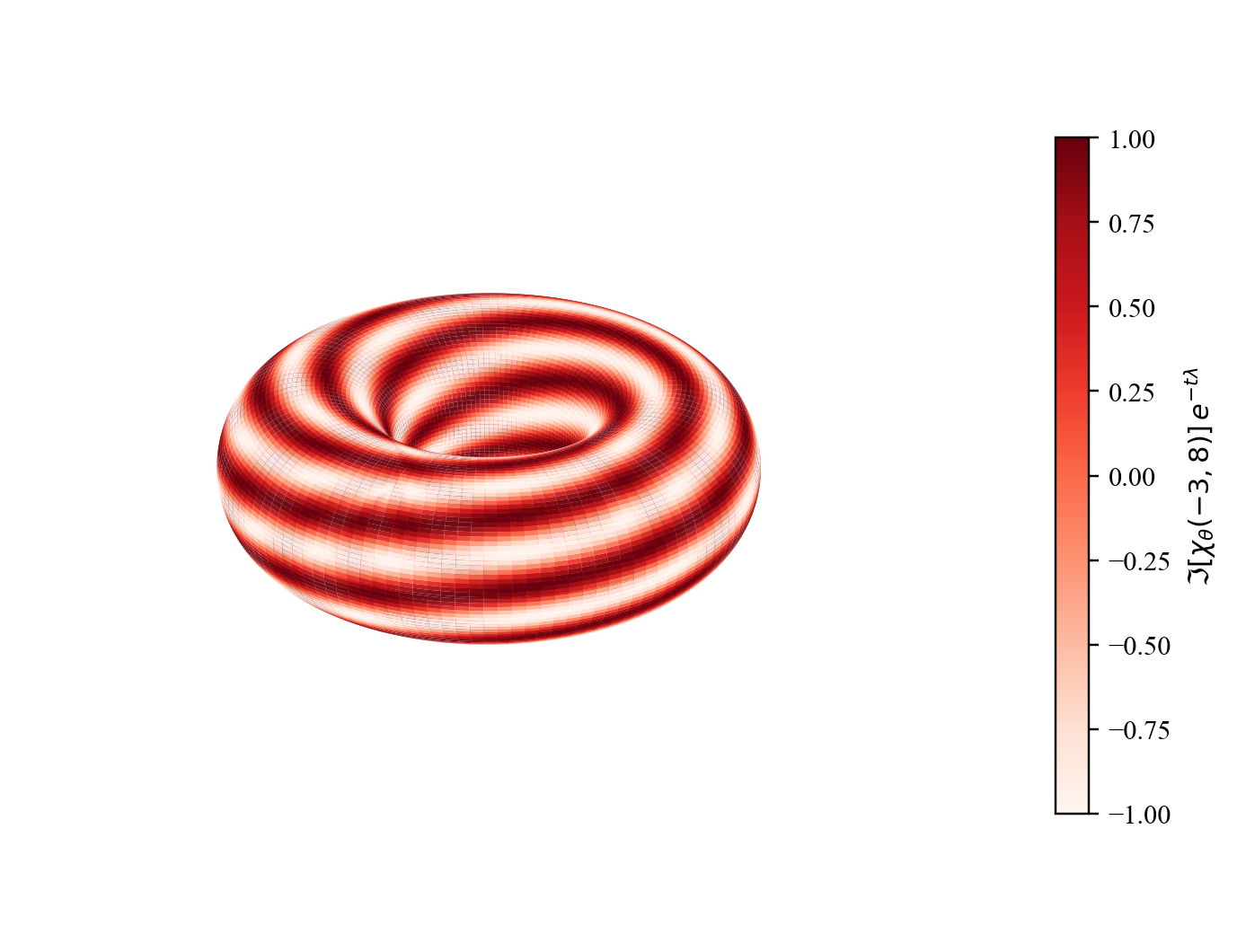}
    \end{minipage}\hfill
    \begin{minipage}{0.30\linewidth}
        \centering
        \includegraphics[width=\linewidth]{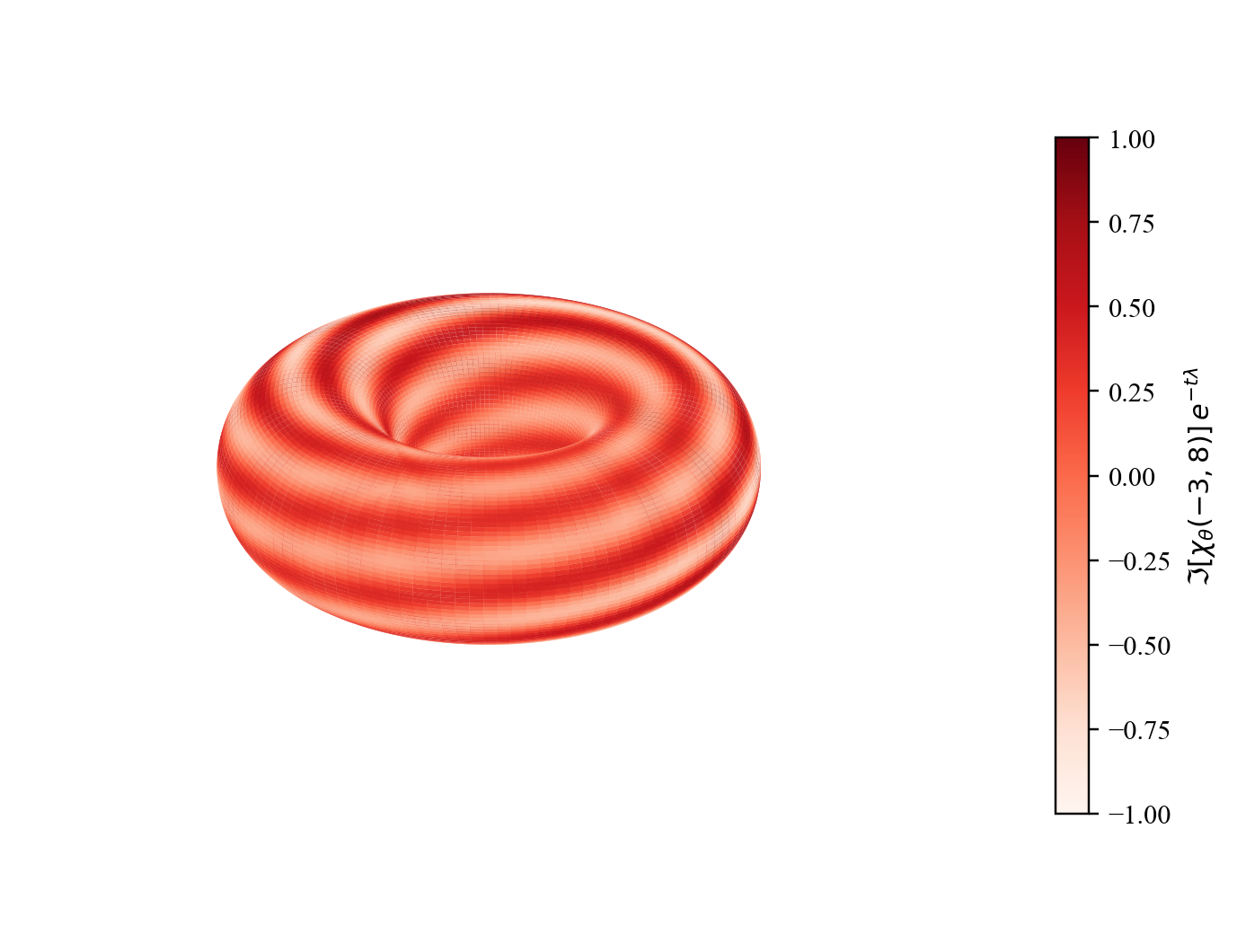}
    \end{minipage}\hfill
    \begin{minipage}{0.30\linewidth}
        \centering
        \includegraphics[width=\linewidth]{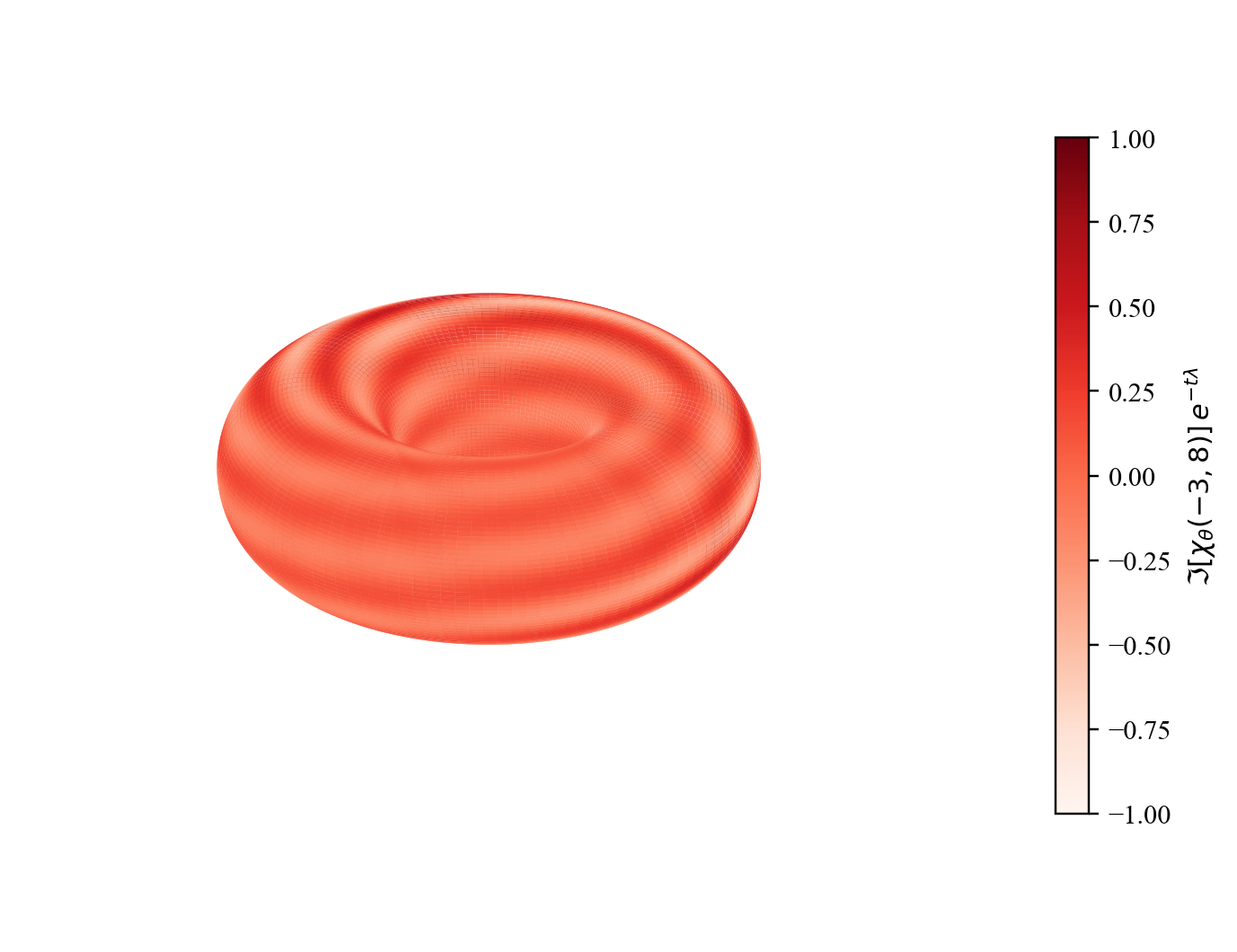}
    \end{minipage}

    \caption{
    Heat flow of the Fourier transform from Fig.~\ref{fig:fourier-vpd} on the Pontryagin dual torus from Fig.~\ref{fig:torus-embedding}, originating from the virtual persistence diagram in Fig.~\ref{fig:toy-example}. Columns correspond to times $t=0.0,\ 0.125,$ and $0.25$.  The two rows show the corresponding real and imaginary parts rendered as heatmaps.
}
    \label{fig:heat-kernel-heatmaps}
\end{figure}

\begin{proposition}\label{prop:lip-heat-transform}
For $t>0$ let $\nu_t$ be the heat measure from
Definition~\ref{def:heat-measure}.
Then
\[
  \mathrm{Lip}\bigl(F_{\nu_t}\bigr)
  \ \le\
  \int_{\mathbb T^{|X\setminus A|}}\!\mathrm{Lip}(\chi_\theta)\,
                        e^{-t\lambda(\theta)}\,d\mu(\theta),
\]
and the right-hand side is nonincreasing in $t$.
\end{proposition}

\begin{proof}
Applying Lemma~\ref{lem:multiplier-stability} with
$d\nu_t(\theta)=e^{-t\lambda(\theta)}\,d\mu(\theta)$ gives
\[
  \mathrm{Lip}(F_{\nu_t})
  \ \le\
  \int_{\mathbb T^{|X\setminus A|}}\!\mathrm{Lip}(\chi_\theta)\,d|\nu_t|(\theta).
\]
Since $e^{-t\lambda(\theta)}\ge 0$ for all $\theta$, we have
$d|\nu_t|=d\nu_t$, which yields the displayed inequality.

For monotonicity, let $0<t_1<t_2$.
For every $\theta\in\mathbb T^{|X\setminus A|}$ we have $\lambda(\theta)\ge 0$ and hence
$e^{-t_2\lambda(\theta)}\le e^{-t_1\lambda(\theta)}$.
Multiplying by the nonnegative factor $\mathrm{Lip}(\chi_\theta)$ and
integrating against $\mu$ shows that
\[
  \int_{\mathbb T^{|X\setminus A|}}\!\mathrm{Lip}(\chi_\theta)\,
                       e^{-t_2\lambda(\theta)}\,d\mu(\theta)
  \ \le\
  \int_{\mathbb T^{|X\setminus A|}}\!\mathrm{Lip}(\chi_\theta)\,
                       e^{-t_1\lambda(\theta)}\,d\mu(\theta),
\]
so the bound on $\mathrm{Lip}(F_{\nu_t})$ is nonincreasing in~$t$.
\end{proof}

By Lemma~\ref{lem:lambda-vs-L},
\[
\lambda(\theta)
\ge
\frac{4w_{\min}d_{\min}^2}{\pi^2}\,
\mathrm{Lip}(\chi_\theta)^2.
\]
Hence
\[
\mathrm{Lip}(\chi_\theta)
\le
\Bigl(\frac{\pi^2}{4w_{\min}d_{\min}^2}\Bigr)^{1/2}
\lambda(\theta)^{1/2}.
\]
Combining this estimate with Proposition~\ref{prop:lip-heat-transform}
gives the Dirichlet energy-based bound
\begin{equation}\label{eq:energy-lip-heat-transform}
\mathrm{Lip}\bigl(F_{\nu_t}\bigr)
\le
\Bigl(\frac{\pi^2}{4w_{\min}d_{\min}^2}\Bigr)^{1/2}
\int_{\mathbb T^{|X\setminus A|}}
\lambda(\theta)^{1/2}
e^{-t\lambda(\theta)}\,d\mu(\theta).
\end{equation}
This expresses the Lipschitz control of the heat-weighted
Fourier--Stieltjes transform directly in terms of the Dirichlet energy:
as $t$ increases, the heat factor increasingly damps modes (Definition~\ref{def:characters-dual}) with large Dirichlet
energy.

\section{Reproducing kernel Hilbert spaces}\label{subsec:rkhs-layer}

We now lift the Lipschitz estimates from individual characters and
Fourier--Stieltjes transforms to arbitrary functions in the associated
reproducing kernel Hilbert spaces. For any finite positive Borel measure
\(\nu\) on \(\mathbb T^{|X\setminus A|}\), Bochner's theorem
(Theorem~\ref{thm:bochner}) determines the translation-invariant kernel
\begin{equation}\label{eq:rkhs-bochner-kernel}
k_\nu(\alpha,\beta)
=
\int_{\mathbb T^{|X\setminus A|}}
\chi_\theta(\alpha-\beta)\,d\nu(\theta),
\end{equation}
and Theorem~\ref{thm:rkhs} associates to \(k_\nu\) the corresponding
RKHS \(\mathcal H_\nu\) of functions on \(K(X,A)\).

In this section we show that the heat measures $\nu_t$ from
Equation~\eqref{eq:heat-measure} produce RKHSs whose functions are
globally Lipschitz (Definition~\ref{def:Lip-seminorm}) with respect to the VPD metric. The resulting bounds
are controlled by the same spectral quantities governing the individual
characters, namely their VPD Lipschitz seminorms and Dirichlet energies (Equation~\eqref{eq:lambda-dirichlet}).
Consequently, the resulting RKHS functions are built primarily from
Fourier modes with small Dirichlet energy. This provides a class of
stable topological features on virtual persistence diagrams and prepares
the ground for the finite-dimensional random Fourier feature
constructions in Equation~\eqref{eq:heat-rff-map}.

\subsection{Heat-weighted RKHS on virtual persistence diagrams}\label{subsec:heat-rkhs}

For a finite positive Borel measure $\nu$ on $\mathbb T^{|X\setminus A|}$, Bochner's theorem (Theorem~\ref{thm:bochner})
identifies the corresponding RKHS $\mathcal H_\nu$ on $G$ as the closure of
the span of the feature maps
\begin{equation}\label{eq:rkhs-feature-map}
  \Phi_\nu(\alpha)(\theta)\ :=\ \chi_\theta(\alpha),
  \qquad \alpha\in G,\ \theta\in\mathbb T^{|X\setminus A|}
\end{equation}
inside $L^2(\nu)$, with kernel
\begin{equation}\label{eq:rkhs-kernel-realization}
  k_\nu(\alpha,\beta)
  =
  \int_{\mathbb T^{|X\setminus A|}}\chi_\theta(\alpha-\beta)\,d\nu(\theta)
  =
  \langle \Phi_\nu(\alpha),\Phi_\nu(\beta)\rangle_{L^2(\nu)}.
\end{equation}
In this realization, $\mathcal H_\nu$ is generated by the same
characters $\chi_\theta$ whose phase variation determines their VPD
Lipschitz behavior, and the kernel encodes their spectral
superposition with weight~$\nu$.

The next lemma lifts the pointwise Lipschitz bounds for characters to
arbitrary functions in $\mathcal H_\nu$, with constants expressed directly in
terms of the Lipschitz seminorms of $\chi_\theta$.

\begin{lemma}\label{lem:rkhs-lip}
Let $\nu$ be a finite positive Borel measure on
$\mathbb T^{|X\setminus A|}$, and let $\mathcal H_\nu$ be the
associated RKHS on $G$. Then every $f\in\mathcal H_\nu$ satisfies
\begin{equation}\label{eq:rkhs-lip-bound}
  \mathrm{Lip}(f)
  \le
  \|f\|_{\mathcal H_\nu}
  \Bigg(
    \int_{\mathbb T^{|X\setminus A|}}
    \mathrm{Lip}(\chi_\theta)^2\,d\nu(\theta)
  \Bigg)^{1/2}.
\end{equation}
\end{lemma}

\begin{proof}
By the reproducing property,
\[
f(\alpha)-f(\beta)
=
\langle
f,\,
k_\nu(\cdot,\alpha)-k_\nu(\cdot,\beta)
\rangle_{\mathcal H_\nu},
\qquad
\alpha,\beta\in G,
\]
so the Cauchy--Schwarz inequality gives
\[
|f(\alpha)-f(\beta)|
\le
\|f\|_{\mathcal H_\nu}\,
\|k_\nu(\cdot,\alpha)-k_\nu(\cdot,\beta)\|_{\mathcal H_\nu}.
\]

Expanding the squared norm and using Bochner's theorem representation (Theorem~\ref{thm:bochner}) of
$k_\nu$ (Equation~\eqref{eq:rkhs-kernel-realization}) gives
\begin{align}
\|k_\nu(\cdot,\alpha)-k_\nu(\cdot,\beta)\|_{\mathcal H_\nu}^2
&=
k_\nu(\alpha,\alpha)
+k_\nu(\beta,\beta)
-2\Re\,k_\nu(\alpha,\beta) \nonumber\\
&=
\int_{\mathbb T^{|X\setminus A|}}
\bigl|\chi_\theta(\alpha)-\chi_\theta(\beta)\bigr|^2
\,d\nu(\theta).
\label{eq:kernel-difference-norm}
\end{align}

For points on $S^1$, chordal distance is bounded above by geodesic
distance. Hence
\[
\bigl|\chi_\theta(\alpha)-\chi_\theta(\beta)\bigr|
\le
\operatorname{dist}\bigl(
\chi_\theta(\alpha),
\chi_\theta(\beta)
\bigr).
\]
By the definition of $\mathrm{Lip}(\chi_\theta)$,
\[
\operatorname{dist}\bigl(
\chi_\theta(\alpha),
\chi_\theta(\beta)
\bigr)
\le
\mathrm{Lip}(\chi_\theta)\,\rho(\alpha,\beta).
\]
Therefore,
\[
\|k_\nu(\cdot,\alpha)-k_\nu(\cdot,\beta)\|_{\mathcal H_\nu}^2
\le
\rho(\alpha,\beta)^2
\int_{\mathbb T^{|X\setminus A|}}
\mathrm{Lip}(\chi_\theta)^2
\,d\nu(\theta),
\]
and hence
\[
\|k_\nu(\cdot,\alpha)-k_\nu(\cdot,\beta)\|_{\mathcal H_\nu}
\le
\rho(\alpha,\beta)
\Bigg(
\int_{\mathbb T^{|X\setminus A|}}
\mathrm{Lip}(\chi_\theta)^2
\,d\nu(\theta)
\Bigg)^{1/2}.
\]

Substituting this estimate into the Cauchy--Schwarz bound gives
\[
|f(\alpha)-f(\beta)|
\le
\rho(\alpha,\beta)\,
\|f\|_{\mathcal H_\nu}
\Bigg(
\int_{\mathbb T^{|X\setminus A|}}
\mathrm{Lip}(\chi_\theta)^2
\,d\nu(\theta)
\Bigg)^{1/2}.
\]
Taking the supremum over $\alpha\ne\beta$ proves the claim.
\end{proof}

Thus every function in $\mathcal H_\nu$ inherits Wasserstein stability from the character family $\{\chi_\theta\}$: if the characters used in the Fourier representation have stable phase behaviour on persistence intervals, then every function with bounded $\mathcal H_\nu$-norm is automatically Lipschitz, with Lipschitz constant controlled by the $\nu$-weighted average of the characterwise quantities $\mathrm{Lip}(\chi_\theta)^2$.

\medskip

We now specialize to the heat measures on the dual introduced in
Definition~\ref{def:heat-measure}. For each $t>0$ the measure
\[
d\nu_t(\theta)
:=
e^{-t\lambda(\theta)}\,d\mu(\theta)
\]
weights each character by the heat factor determined by its Dirichlet
energy. Since $\lambda(\theta)$ is comparable to
$\mathrm{Lip}(\chi_\theta)^2$, increasing $t$ suppresses modes with
unstable phase variation, equivalently those with large VPD Lipschitz
seminorm. Let $\mathcal H_t:=\mathcal H_{\nu_t}$ be the corresponding
RKHS on $G$.

\begin{theorem}\label{thm:heat-lip}
For every $t>0$ and $f\in\mathcal H_t$,
\begin{equation}\label{eq:heat-rkhs-lip-bound}
  \mathrm{Lip}(f)
  \le
  \|f\|_{\mathcal H_t}
  \Bigg(\int_{\mathbb T^{|X\setminus A|}}
    \mathrm{Lip}(\chi_\theta)^2
    e^{-t\lambda(\theta)}\,d\mu(\theta)\Bigg)^{1/2}
\end{equation}
and the integral prefactor
\begin{equation}\label{eq:heat-lip-prefactor}
  t\longmapsto
  \int_{\mathbb T^{|X\setminus A|}}
    \mathrm{Lip}(\chi_\theta)^2
    e^{-t\lambda(\theta)}\,d\mu(\theta)
\end{equation}
is finite and nonincreasing on $(0,\infty)$.
\end{theorem}

\begin{proof}
Apply Lemma~\ref{lem:rkhs-lip} with $\nu=\nu_t$. This gives, for every
$f\in\mathcal H_t$,
\[
  \mathrm{Lip}(f)\ \le\ \|f\|_{\mathcal H_t}\,
  \Bigg(\int_{\mathbb T^{|X\setminus A|}}\mathrm{Lip}(\chi_\theta)^2\,
                        e^{-t\lambda(\theta)}\,d\mu(\theta)\Bigg)^{1/2},
\]
which is the first claim.

For finiteness and monotonicity of the prefactor, note that
$\mathrm{Lip}(\chi_\theta)^2\,e^{-t\lambda(\theta)}\ge 0$ and
$\mu$ is a probability measure, so the integral is finite for each $t>0$
as soon as $\theta\mapsto\mathrm{Lip}(\chi_\theta)^2$ is $\mu$-integrable,
which follows from the edgewise formula (Equation~\eqref{eq:edgewise-character-lipschitz}).
Moreover, $\lambda(\theta)\ge 0$ for all $\theta$, so for each fixed $\theta$
the map $t\mapsto e^{-t\lambda(\theta)}$ is nonincreasing on $(0,\infty)$.
Multiplying by the nonnegative factor $\mathrm{Lip}(\chi_\theta)^2$
preserves pointwise monotonicity, and integration against $\mu$ preserves
the order. Hence, the prefactor is nonincreasing in $t$.
\end{proof}

In words, the heat weight $e^{-t\lambda(\theta)}$
suppresses characters with unstable phase variation across nearby
persistence intervals. Since characters aggregate phase information over
entire virtual persistence diagrams, instability could in principle
accumulate under Fourier superposition. Theorem~\ref{thm:heat-lip}
shows that this does not occur in the heat-weighted RKHS: the Lipschitz seminorm of an RKHS function is controlled by a
heat-weighted average of the characterwise Lipschitz scales. As $t$
increases, contributions from high Dirichlet-energy characters are increasingly
damped, so the resulting functions concentrate on stable Fourier modes.

\subsection{Geometric corollaries}\label{subsec:heat-geom}

Theorem~\ref{thm:heat-lip} bounds the Lipschitz seminorm of
functions in $\mathcal H_t$ by a heat-weighted average of the seminorms
$\mathrm{Lip}(\chi_\theta)$. We now rewrite this estimate in
spectral and geometric forms using the character-phase identity in Lemma~\ref{lem:char-lip-comparison}
and the phase to Dirichlet-energy comparison in Lemma~\ref{lem:lambda-vs-L}.

\begin{corollary}[Spectral form]\label{cor:heat-lip-spectral}
For every $t>0$ and $f\in\mathcal H_t$,
\begin{equation}\label{eq:heat-lip-spectral}
  \mathrm{Lip}(f)
  \le
  \frac{\pi}{2d_{\min}\sqrt{w_{\min}}}\,\,
  \|f\|_{\mathcal H_t}
  \Bigg(\int_{\mathbb T^{|X \setminus A|}}
    \lambda(\theta)e^{-t\lambda(\theta)}\,d\mu(\theta)\Bigg)^{1/2}
\end{equation}
with $w_{\min},d_{\min}$ as in Lemma~\ref{lem:lambda-vs-L}.
\end{corollary}

\begin{proof}
By Lemma~\ref{lem:lambda-vs-L},
\[
\lambda(\theta)
\ge
\frac{4w_{\min}d_{\min}^2}{\pi^2}\,
\mathrm{Lip}(\chi_\theta)^2.
\]
Equivalently,
\[
  \mathrm{Lip}(\chi_\theta)^2
  \ \le\ \frac{\pi^2}{4w_{\min}d_{\min}^2}\,\lambda(\theta).
\]
Therefore,
\[
  \int_{\mathbb T^{|X\setminus A|}}
    \mathrm{Lip}(\chi_\theta)^2\,e^{-t\lambda(\theta)}\,d\mu(\theta)
  \ \le\
  \frac{\pi^2}{4w_{\min}d_{\min}^2}\,\,
  \int_{\mathbb T^{|X\setminus A|}}\lambda(\theta)\,e^{-t\lambda(\theta)}\,d\mu(\theta).
\]
Thus, for every $f\in\mathcal H_t$,
\[
\begin{aligned}
  \mathrm{Lip}(f)
  &\ \le\ \|f\|_{\mathcal H_t}\,
         \Bigg(\int_{\mathbb T^{|X\setminus A|}}
                 \mathrm{Lip}(\chi_\theta)^2\,
                 e^{-t\lambda(\theta)}\,d\mu(\theta)\Bigg)^{1/2}\\[4pt]
  &\ \le\ \|f\|_{\mathcal H_t}\,
         \Bigg(\frac{\pi^2}{4w_{\min}d_{\min}^2}\,\,
                 \int_{\mathbb T^{|X\setminus A|}}\lambda(\theta)\,
                                    e^{-t\lambda(\theta)}\,d\mu(\theta)
         \Bigg)^{1/2}\\[4pt]
  &\ =\ \frac{\pi}{2d_{\min}\sqrt{w_{\min}}}\,\,
         \|f\|_{\mathcal H_t}\,
         \Bigg(\int_{\mathbb T^{|X\setminus A|}}\lambda(\theta)\,
                                e^{-t\lambda(\theta)}\,d\mu(\theta)\Bigg)^{1/2},
\end{aligned}
\]
which is the claimed inequality.
\end{proof}

In this form, the Lipschitz constant is controlled by a
heat-weighted average of the Dirichlet energies
$\lambda(\theta)$. Since $\lambda(\theta)$ measures the phase
instability of the character $\chi_\theta$, the factor
$\lambda(\theta)e^{-t\lambda(\theta)}$ quantifies how much instability
remains after heat damping at scale $t$. Corollary~\ref{cor:heat-lip-spectral}
therefore shows that functions in the heat RKHS are $\rho$-stable when
their Fourier representations concentrate on low Dirichlet-energy, slowly varying
modes.

\medskip

The next corollary bounds the heat factor using the comparison between
Dirichlet energy and phase Lipschitz scale from
Lemma~\ref{lem:lambda-vs-L}. The resulting estimate depends only on the
Lipschitz seminorms of the phase functions on the finite metric space
$(X/A,\overline d_1)$ and expresses the heat weight through Gaussian
decay in \(\mathrm{Lip}(\phi_\theta)^2\).

\begin{corollary}[Geometric form]\label{cor:heat-lip-geom}
For every $t>0$ and $f\in\mathcal H_t$,
\begin{equation}\label{eq:heat-lip-geometric}
  \mathrm{Lip}(f)
  \le
  \|f\|_{\mathcal H_t}
  \Bigg(\int_{\mathbb T^{|X \setminus A|}}
    \mathrm{Lip}(\phi_\theta)^2
    \exp\!\Big(-t\,\frac{4w_{\min}d_{\min}^2}{\pi^2}\,\,
                 \mathrm{Lip}(\phi_\theta)^2\Big)
    \,d\mu(\theta)
  \Bigg)^{1/2}
\end{equation}
where $\phi_\theta$ is the phase function from
Lemma~\ref{lem:char-lip-comparison}.
\end{corollary}

\begin{proof}
By Lemma~\ref{lem:char-lip-comparison},
\[
  \mathrm{Lip}(\chi_\theta)
  =
  \mathrm{Lip}(\phi_\theta).
\]
Combining this identity with Lemma~\ref{lem:lambda-vs-L} gives
\[
  \lambda(\theta)
  \ge
  \frac{4w_{\min}d_{\min}^2}{\pi^2}\,\,
  \mathrm{Lip}(\phi_\theta)^2.
\]
Hence
\[
  e^{-t\lambda(\theta)}
  \le
  \exp\!\Big(-t\,\frac{4w_{\min}d_{\min}^2}{\pi^2}\,\,
               \mathrm{Lip}(\phi_\theta)^2\Big).
\]
Substituting this estimate and the identity
\[
  \mathrm{Lip}(\chi_\theta)^2
  =
  \mathrm{Lip}(\phi_\theta)^2
\]
into the integral in Theorem~\ref{thm:heat-lip} gives
\begin{align*}
  &\int_{\mathbb T^{|X\setminus A|}}
    \mathrm{Lip}(\chi_\theta)^2\,
    e^{-t\lambda(\theta)}\,d\mu(\theta) \\
  &\le
  \int_{\mathbb T^{|X\setminus A|}}
    \mathrm{Lip}(\phi_\theta)^2
    \exp\!\Big(-t\,\frac{4w_{\min}d_{\min}^2}{\pi^2}\,\,
                 \mathrm{Lip}(\phi_\theta)^2\Big)
    \,d\mu(\theta).
\end{align*}
Combining this bound with Theorem~\ref{thm:heat-lip} proves the claim.
\end{proof}

This form isolates the dependence on the stability of phases of the quotient metric space $(X/A,\overline d_1)$: the integrand depends only
on the Lipschitz seminorms
$\mathrm{Lip}(\phi_\theta)$ of phase maps with the corresponding heat kernel
weights. Consequently, characters whose phase functions exhibit unstable
variation across nearby persistence intervals are increasingly
suppressed as $t$ grows. The heat RKHS $\mathcal H_t$ therefore favors
functions built from stable Fourier modes, namely characters whose
associated phase functions vary slowly with respect to the quotient
metric $\overline d_1$.

\begin{remark}\label{rem:gaussian-heat}
On $\mathbb R^d$ with the Euclidean Laplacian $\Delta$, the heat semigroup (After Definition~\ref{def:fourier-multiplier})
$e^{t\Delta}$ has Fourier multiplier $e^{-t|\xi|^2}$ and fundamental solution
\[
  K_t(x,y)\ =\ \frac{1}{(4\pi t)^{d/2}}
               \exp\Big(-\frac{|x-y|^2}{4t}\Big),
\]
the standard Gaussian radial basis kernel. The corresponding RKHS consists
of functions whose Fourier transforms satisfy
\[
\int_{\mathbb R^d} |\widehat f(\xi)|^2 e^{t|\xi|^2}\,d\xi < \infty,
\]
up to the normalization determined by the Gaussian kernel, and
Corollary~\ref{cor:heat-lip-spectral} reduces to the familiar statement
that Gaussian RKHS functions are Lipschitz, with a constant controlled by
a heat-weighted second moment of $|\xi|$.

In our setting, $k_t=k_{\nu_t}$ plays the analogous role on $G$:
$\lambda(\theta)$ takes the place of $|\xi|^2$, the weight
$e^{-t\lambda(\theta)}$ is the heat multiplier on the dual torus
$\widehat G\cong\mathbb T^{|X\setminus A|}$, and Corollaries~\ref{cor:heat-lip-spectral}
and~\ref{cor:heat-lip-geom} quantify how this spectral damping translates
into Lipschitz regularity with respect to the VPD metric~$\rho$.  The resulting bounds are the analytic input we will use later to control the stability of learned functionals of virtual persistence diagrams under
$W_1$-perturbations of the underlying diagrams.
\end{remark}

\section{Random Fourier features}\label{subsec:heat-rff}

The heat kernels associated to the heat measures
( Equation~\eqref{eq:heat-measure}) define a family of translation-invariant RKHSs
$\{\mathcal H_t\}_{t>0}$ on the virtual diagram group
$G:=K(X,A)$, with Lipschitz control given by
Theorem~\ref{thm:heat-lip} and
Corollaries~\ref{cor:heat-lip-spectral}-\ref{cor:heat-lip-geom}. We now
construct finite-dimensional random feature maps that approximate the
heat kernels by sampling characters from the heat law on the dual torus.
The resulting embeddings map virtual persistence diagrams into Euclidean space through heat-weighted circular coordinate systems.

\subsection{Sampling from the heat law}\label{subsec:heat-rff-sampling}

We now pass from the integral representation of $k_t$ to an explicit
finite-dimensional feature map by Monte Carlo sampling of characters from the
heat measure. The construction is the usual cosine-sine lift of complex
exponentials, but here the sampling law is the heat law on the dual torus and
the domain is the virtual diagram group $K(X,A)$.

\begin{definition}\label{def:heat-rff}
Fix $t>0$ and $R\in\mathbb N$. Let
$\theta^{(1)},\dots,\theta^{(R)}$ be independent samples from the probability
measure on $\mathbb T^{|X\setminus A|}$ with density
\begin{equation}\label{eq:heat-sampling-density}
  \theta\ \longmapsto\ \frac{e^{-t\lambda(\theta)}}{\nu_t(\mathbb T^{|X\setminus A|})}
  \quad\text{with respect to }\mu
\end{equation}
that is, from the normalized heat measure $\nu_t/\nu_t(\mathbb T^{|X\setminus A|})$. Define
the feature map
\begin{equation}\label{eq:heat-rff-map}
  \Phi_{t,R}:K(X,A)\longrightarrow\mathbb R^{2R},
  \qquad
  \Phi_{t,R}(\alpha)
  :=\sqrt{\frac{\nu_t(\mathbb T^{|X\setminus A|})}{R}}\,
    \bigl(\cos\langle\alpha,\theta^{(r)}\rangle,\
          \sin\langle\alpha,\theta^{(r)}\rangle\bigr)_{r=1}^R.
\end{equation}
\end{definition}

The feature map $\Phi_{t,R}$ samples characters from the heat law on the dual torus and records each character evaluation through its cosine and sine coordinates. The scaling is chosen so that inner products of the resulting feature vectors approximate the heat kernel $k_t$.

\begin{remark}\label{rem:heat-law-sampling}
Since \(\lambda(\theta)\ge 0\), we have
\(0<e^{-t\lambda(\theta)}\le 1\) for all
\(\theta\in\mathbb T^{|X\setminus A|}\). To sample from
\(\nu_t/\nu_t(\mathbb T^{|X\setminus A|})\), draw
\(\Theta\sim\mu\) by sampling each torus coordinate independently
uniformly from \([0,2\pi)\), draw
\(U\sim\mathrm{Unif}[0,1]\) independently, and accept \(\Theta\) when
\(U\le e^{-t\lambda(\Theta)}\). The accepted samples then have law
\(\nu_t/\nu_t(\mathbb T^{|X\setminus A|})\).
\end{remark}

\subsection{Kernel approximation and unbiasedness}

The next lemma records the heat-weighted analog of the standard
unbiasedness property of random Fourier features~\cite{10.5555/2981562.2981710}:
The inner product of two feature vectors is an unbiased Monte Carlo estimator
of $k_t(\alpha,\beta)$.

\begin{lemma}\label{lem:heat-rff-unbiased}
For all $\alpha,\beta\in K(X,A)$,
\begin{equation}\label{eq:heat-rff-unbiased}
  \mathbb E\big[\langle\Phi_{t,R}(\alpha),\Phi_{t,R}(\beta)\rangle\big]
  =
  k_t(\alpha,\beta).
\end{equation}
\end{lemma}

\begin{proof}
By Definition~\ref{def:heat-rff},
\[
\begin{aligned}
\mathbb E\big[\langle\Phi_{t,R}(\alpha),\Phi_{t,R}(\beta)\rangle\big]
&=
\frac{\nu_t(\mathbb T^{|X\setminus A|})}{R}
\sum_{r=1}^R
\mathbb E\Big[
\cos\langle\alpha,\theta^{(r)}\rangle\,
\cos\langle\beta,\theta^{(r)}\rangle\\
&\hphantom{=
\frac{\nu_t(\mathbb T^{|X\setminus A|})}{R}
\sum_{r=1}^R
\mathbb E\Big[}
\quad
+
\sin\langle\alpha,\theta^{(r)}\rangle\,
\sin\langle\beta,\theta^{(r)}\rangle
\Big].
\end{aligned}
\]
Since the $\theta^{(r)}$ are i.i.d., the sum reduces to a single
expectation, and the cosine angle-difference identity gives
\[
\begin{aligned}
\mathbb E\big[\langle\Phi_{t,R}(\alpha),\Phi_{t,R}(\beta)\rangle\big]
&=
\nu_t(\mathbb T^{|X\setminus A|})\,
\mathbb E\big[
\cos\langle\alpha-\beta,\theta^{(1)}\rangle
\big]\\
&=
\nu_t(\mathbb T^{|X\setminus A|})
\int_{\mathbb T^{|X\setminus A|}}
\cos\langle\alpha-\beta,\theta\rangle\,
\frac{e^{-t\lambda(\theta)}}{\nu_t(\mathbb T^{|X\setminus A|})}
\,d\mu(\theta)\\
&=
\int_{\mathbb T^{|X\setminus A|}}
\cos\langle\alpha-\beta,\theta\rangle\,
e^{-t\lambda(\theta)}
\,d\mu(\theta).
\end{aligned}
\]

Replacing $\theta$ by $-\theta$ conjugates each value of the associated
phase function, so the Dirichlet energy (Equation~\eqref{eq:lambda-dirichlet}) satisfies
$\lambda(-\theta)=\lambda(\theta)$. Since Haar measure is invariant under
$\theta\mapsto-\theta$, the imaginary part of
$e^{i\langle\alpha-\beta,\theta\rangle}$ integrates to zero. Therefore
\[
\begin{aligned}
\int_{\mathbb T^{|X\setminus A|}}
\cos\langle\alpha-\beta,\theta\rangle\,
e^{-t\lambda(\theta)}
\,d\mu(\theta)
&=
\int_{\mathbb T^{|X\setminus A|}}
e^{i\langle\alpha-\beta,\theta\rangle}\,
e^{-t\lambda(\theta)}
\,d\mu(\theta)\\
&=
k_t(\alpha,\beta),
\end{aligned}
\]
where the final equality is the Bochner representation of $k_t$.
\end{proof}

Thus, at the level of kernels, $(\alpha,\beta)\mapsto\langle
\Phi_{t,R}(\alpha),\Phi_{t,R}(\beta)\rangle$ is a Monte Carlo approximation to
$k_t(\alpha,\beta)$ that becomes accurate as $R$ grows, while keeping all
computations in the finite-dimensional Euclidean space $\mathbb R^{2R}$.

\subsection{Lipschitz control for a fixed draw}

We now link the random features back to the VPD metric $\rho$. For a fixed
draw of frequencies $\{\theta^{(r)}\}$, the following lemma bounds the
Lipschitz constant of $\Phi_{t,R}$ in terms of the characterwise
Lipschitz seminorms from Lemma~\ref{lem:char-lip-comparison}.

\begin{lemma}\label{lem:heat-rff-lip}
For any fixed sample $\{\theta^{(r)}\}_{r=1}^R$,
\begin{equation}\label{eq:heat-rff-fixed-draw-lip}
\mathrm{Lip}(\Phi_{t,R})
\le
\sqrt{2\,\nu_t(\mathbb T^{|X \setminus A|})}\,
\Bigg(\frac1R\sum_{r=1}^R
\mathrm{Lip}(\chi_{\theta^{(r)}})^2\Bigg)^{1/2}.
\end{equation}
\end{lemma}

\begin{proof}
Fix $\theta\in\mathbb T^{|X\setminus A|}$. Then, for all
$\alpha,\beta\in K(X,A)$,
\begin{align*}
|\cos\langle\alpha,\theta\rangle-\cos\langle\beta,\theta\rangle|
&=
|\Re(\chi_\theta(\alpha)-\chi_\theta(\beta))| \\
&\le
|\chi_\theta(\alpha)-\chi_\theta(\beta)| \\
&\le
\operatorname{dist}\bigl(\chi_\theta(\alpha),\chi_\theta(\beta)\bigr) \\
&\le
\mathrm{Lip}(\chi_\theta)\,\rho(\alpha,\beta).
\end{align*}
The same argument with imaginary parts gives the corresponding bound for
sine.
Thus each of the two coordinates
\[
\alpha\ \longmapsto\ \cos\langle\alpha,\theta\rangle,\qquad
\alpha\ \longmapsto\ \sin\langle\alpha,\theta\rangle
\]
is $\mathrm{Lip}(\chi_\theta)$-Lipschitz on $(K(X,A),\rho)$. For
$\theta^{(r)}$ as in Definition~\ref{def:heat-rff}, the $2R$ coordinates of
$\Phi_{t,R}$ are therefore each
$\sqrt{\nu_t(\mathbb T^{|X \setminus A|})/R}\,
\mathrm{Lip}(\chi_{\theta^{(r)}})$-Lipschitz.
For a Euclidean-valued map $F=(f_i)_i$ one has
\begin{equation}\label{eq:coordinate-lip-euclidean-map}
\|F(\alpha)-F(\beta)\|_2^2
\le
\rho(\alpha,\beta)^2\sum_i \mathrm{Lip}(f_i)^2,
\end{equation}
where the sum runs over coordinates $f_i$. Applying this to
$\Phi_{t,R}$ gives
\[
\begin{aligned}
\mathrm{Lip}(\Phi_{t,R})^2
&\le
\frac{\nu_t(\mathbb T^{|X\setminus A|})}{R}\sum_{r=1}^R
\bigl(\mathrm{Lip}(\chi_{\theta^{(r)}})^2
+\mathrm{Lip}(\chi_{\theta^{(r)}})^2\bigr)\\
&=
2\,\nu_t(\mathbb T^{|X\setminus A|})
\Bigl(\frac1R\sum_{r=1}^R
\mathrm{Lip}(\chi_{\theta^{(r)}})^2\Bigr),
\end{aligned}
\]
and taking square roots yields the claimed inequality.
\end{proof}

Thus, the Lipschitz constant of $\Phi_{t,R}$ is controlled by
the empirical average of the characterwise Lipschitz seminorms.

\subsection{Concentration of the instability scales}

We now identify the large-sample limit of the empirical averages
$\frac1R\sum_{r=1}^R \mathrm{Lip}(\chi_{\theta^{(r)}})^2$
under the heat law.

\begin{lemma}\label{lem:heat-rff-lln}
With $\{\theta^{(r)}\}_{r=1}^R$ as in Definition~\ref{def:heat-rff}, one has
\begin{equation}\label{eq:heat-rff-lln}
  \frac1R\sum_{r=1}^R \mathrm{Lip}(\chi_{\theta^{(r)}})^2
  \xrightarrow{\ \mathbb P\ }
  \frac{1}{\nu_t(\mathbb T^{|X\setminus A|})}
  \int_{\mathbb T^{|X\setminus A|}}\mathrm{Lip}(\chi_\theta)^2\,
  e^{-t\lambda(\theta)}\,d\mu(\theta)
\end{equation}
as $R\to\infty$.
\end{lemma}

\begin{proof}
The random variables
$\mathrm{Lip}(\chi_{\theta^{(r)}})^2$ are i.i.d.\ with finite mean
\[
\begin{aligned}
\mathbb E\big[\mathrm{Lip}(\chi_{\theta^{(1)}})^2\big]
&=\int_{\mathbb T^{|X\setminus A|}}\mathrm{Lip}(\chi_\theta)^2\,
  \frac{e^{-t\lambda(\theta)}}{\nu_t(\mathbb T^{|X \setminus A|})}\,d\mu(\theta)\\
&=\frac{1}{\nu_t(\mathbb T^{|X\setminus A|})}
  \int_{\mathbb T^{|X\setminus A|}}\mathrm{Lip}(\chi_\theta)^2\,
  e^{-t\lambda(\theta)}\,d\mu(\theta),
\end{aligned}
\]
which is finite by the edgewise formula in Equation~\eqref{eq:edgewise-character-lipschitz} and the definition of $\nu_t$. By the law of large numbers, the empirical averages
\[
  \frac1R\sum_{r=1}^R \mathrm{Lip}(\chi_{\theta^{(r)}})^2
\]
converge in probability to this mean as $R\to\infty$.
\end{proof}

Combining Lemmas~\ref{lem:heat-rff-lip}
and~\ref{lem:heat-rff-lln} shows that the Lipschitz constants of the
random feature maps are governed by the same heat-weighted spectral
quantities that appear in Theorem~\ref{thm:heat-lip}. We now express
these bounds directly in terms of the Dirichlet energy
$\lambda(\theta)$.

\subsection{Spectral stability of the random features}

We now express the Lipschitz stability scale of
$\Phi_{t,R}$ directly in terms of the Dirichlet energy
$\lambda(\theta)$ from Equation~\eqref{eq:lambda-dirichlet}.

\begin{theorem}\label{thm:heat-rff-spectral}
Fix $t>0$. As $R\to\infty$, the Lipschitz constants of the feature maps
\[
  \Phi_{t,R}:K(X,A)\longrightarrow\mathbb R^{2R}
\]
satisfy
\begin{equation}\label{eq:heat-rff-spectral-bound}
  \mathrm{Lip}(\Phi_{t,R})
  \le
  \frac{\pi}{d_{\min}\sqrt{2\,w_{\min}}}
  \Bigg(\int_{\mathbb T^{|X\setminus A|}}
    \lambda(\theta)e^{-t\lambda(\theta)}\,d\mu(\theta)\Bigg)^{1/2}
  +o_{\mathbb P}(1),
\end{equation}
where $w_{\min}$ and $d_{\min}$ are as in
Lemma~\ref{lem:lambda-vs-L}.
\end{theorem}

\begin{proof}
By Lemmas~\ref{lem:heat-rff-lip} and~\ref{lem:heat-rff-lln},
\begin{align*}
\mathrm{Lip}(\Phi_{t,R})^2
&\le
2\,\nu_t(\mathbb T^{|X\setminus A|})
\left(
\frac{1}{\nu_t(\mathbb T^{|X\setminus A|})}
\int_{\mathbb T^{|X\setminus A|}}
\mathrm{Lip}(\chi_\theta)^2
e^{-t\lambda(\theta)}\,d\mu(\theta)
+o_{\mathbb P}(1)
\right)\\
&=
2\int_{\mathbb T^{|X\setminus A|}}
\mathrm{Lip}(\chi_\theta)^2
e^{-t\lambda(\theta)}\,d\mu(\theta)
+o_{\mathbb P}(1).
\end{align*}
By Lemma~\ref{lem:lambda-vs-L},
\[
  \mathrm{Lip}(\chi_\theta)^2
  \le
  \frac{\pi^2}{4\,w_{\min}\,d_{\min}^2}\,\lambda(\theta),
\]
so
\[
  \mathrm{Lip}(\Phi_{t,R})^2
  \le
  \frac{\pi^2}{2\,w_{\min}\,d_{\min}^2}
  \int_{\mathbb T^{|X\setminus A|}}
  \lambda(\theta)e^{-t\lambda(\theta)}\,d\mu(\theta)
  +o_{\mathbb P}(1).
\]
Taking square roots gives Equation~\eqref{eq:heat-rff-spectral-bound}.
\end{proof}

Theorem~\ref{thm:heat-rff-spectral} shows that the
heat weighting suppresses characters with large Dirichlet energy, and
hence large VPD Lipschitz seminorm, in the random feature
representation. Consequently, the feature maps $\Phi_{t,R}$ retain the
Wasserstein stability structure of the heat kernel while reducing
computations to finite-dimensional Euclidean linear algebra on
$\mathbb R^{2R}$.

\section{Experiments}\label{sec:experiments}

In digital image processing, cubical filtrations are naturally finite because pixel or voxel intensities take values in a finite ordered set.  If a filtration uses \(n\) distinct intensity levels, then every nontrivial persistence interval is determined by a pair of filtration indices, so the number of possible interval types is \(\binom{n}{2} = \frac{n(n-1)}{2}\). For standard \(8\)-bit grayscale images, \(n=256\), giving \(\binom{256}{2}=32640\) possible intervals.  More generally, if intensities are represented using \(N\) bits, then the number of interval types is \(\binom{2^N}{2}=2^{N-1}(2^N-1)\). Thus, persistence diagrams from digital images or voxel data admit only finitely many possible persistence intervals and therefore define finite-rank virtual persistence diagrams.

We will define a topological loss on the Grothendieck group $K(X,A)$
associated with the persistence-diagram monoid $D(X,A)$
via the Grothendieck construction (Equation~\eqref{eq:groth-equivalence}). For a ground-truth mask $y$ and soft mask
$\hat y$, let $D_y,D_{\hat y}\in D(X,A)$ denote their
$H_0\oplus H_1$ persistence diagrams under the fixed cubical filtration
from Definition~\ref{def:fixed-cubical-filtration},
regarded as elements of $K(X,A)$ via the canonical inclusion of
$D(X,A)$ into its group completion, and set
$\gamma:=D_{\hat y}-D_y\in K(X,A)$. Let $k_t$ be the
translation-invariant heat kernel on $K(X,A)$ from
Definition~\ref{def:heat-measure} with reproducing kernel
Hilbert space $\mathcal H_t$ induced by the heat measure
from Equation~\eqref{eq:heat-measure}. We represent the topological loss arising from
virtual persistence diagrams as the square of the RKHS semimetric:
\begin{equation}\label{eq:Ltopo-ideal}
\begin{aligned}
\mathcal L_{\mathrm{topo}}(\gamma)
&:= \bigl\|k_t(\gamma,\cdot)-k_t(0,\cdot)\bigr\|_{\mathcal H_t}^2 \\
&= k_t(\gamma,\gamma)+k_t(0,0)-2k_t(\gamma,0) \\
&= 2\bigl(k_t(0,0)-k_t(\gamma,0)\bigr),
\end{aligned}
\end{equation}
where the final equality uses translation-invariance of $k_t$ on
$K(X,A)$.

In the finite-dimensional implementation we approximate
$k_t$ by inner products of the random Fourier feature map
$\Phi_{t,R}:K(X,A)\to\mathbb R^{2R}$ from
Equation~\eqref{eq:heat-rff-map}. The map $\Phi_{t,R}$ samples
heat-weighted circular coordinate systems on the dual torus and records
their associated cosine-sine coordinates on virtual persistence diagrams.
By the unbiasedness identity (Equation~\eqref{eq:heat-rff-unbiased}), the inner
product
\(
\langle \Phi_{t,R}(\alpha),\Phi_{t,R}(\beta)\rangle
\)
is an unbiased Monte Carlo estimator of \(k_t(\alpha,\beta)\). Applying
this approximation to the three kernel terms in
Equation~\eqref{eq:Ltopo-ideal} gives
\begin{equation}\label{eq:Ltopo-rff}
\begin{aligned}
\mathcal L_{\mathrm{topo}}(\gamma)
&\approx
\|\Phi_{t,R}(\gamma)\|_2^2
+\|\Phi_{t,R}(0)\|_2^2
-2\langle\Phi_{t,R}(\gamma),\Phi_{t,R}(0)\rangle \\
&=
\|\Phi_{t,R}(\gamma)-\Phi_{t,R}(0)\|_2^2.
\end{aligned}
\end{equation}
Thus, the implemented topological loss is the squared Euclidean distance
between the random feature representations of $\gamma$ and $0$. The resulting loss uses the same heat weighting of characters as the full heat kernel while reducing computations to finite-dimensional Euclidean linear algebra.

Recall the soft Dice loss \(\mathcal L_{\mathrm{Dice}}\)
defined in Equation~\eqref{eq:dice-loss}. We train the segmentation model using
the sum of the Dice loss and a weighted topological loss. The Dice loss
penalizes disagreement between the predicted and ground-truth masks,
while the topological loss penalizes discrepancies between their
persistence diagrams under the fixed cubical filtration from
Definition~\ref{def:fixed-cubical-filtration}. The parameter
\(w_{\mathrm{topo}}>0\) controls the relative contribution of the
topological loss during training:
\begin{equation}\label{eq:total-loss}
\mathcal L_{\mathrm{total}}(y,\hat y)
=
\mathcal L_{\mathrm{Dice}}(y,\hat y)
+
w_{\mathrm{topo}}\,
\mathcal L_{\mathrm{topo}}(\gamma).
\end{equation}

\subsection{Experimental setup}

We use a synthetic \(64\times64\) binary-segmentation dataset whose
images are formed by superposing rings, spirals, line segments, and blob
clusters. The ground-truth mask is the indicator function of the
foreground region in the noiseless image. We use \(200\) samples with a
fixed \(100/50/50\) train-validation-test split across all methods. At
each epoch, we independently resample Gaussian, Poisson, speckle, and
salt-and-pepper noise together with uniform jitter, and apply these
perturbations only to the input images.

All experiments use UNet, a standard encoder-decoder convolutional
segmentation network~\cite{Ronneberger2015UNetCN}. Consequently,
the differences between the methods reflect only the choice of training
loss. The baseline model minimizes the soft Dice loss
(Equation~\eqref{eq:dice-loss}). The remaining two models augment the Dice loss
with a topological penalty computed from the predicted and ground-truth
persistence diagrams. One model uses a \(2\)-Wasserstein distance
between persistence diagrams~\cite{9186664,Bubenik2022VirtualPD}, while
the other uses the RKHS topological loss
(Equation~\eqref{eq:Ltopo-rff}). Both topological losses use the fixed cubical
filtration from Definition~\ref{def:fixed-cubical-filtration} and the
same weighting parameter \(w_{\mathrm{topo}}=500\). For the RKHS loss,
we use the heat-kernel parameter \(t=10\) and the random-feature dimension
\(R=256\).

We train with Adam, learning rate \(10^{-3}\), batch size \(8\), and at
most twenty epochs, with early stopping based on the validation Dice
coefficient (Equation~\eqref{eq:dice-score}). At test time, we threshold
predictions at \(0.5\) as in Equation~\eqref{eq:thresholded-mask} and report
mean IoU (Equation~\eqref{eq:iou}) and Dice
coefficient (Equation~\eqref{eq:dice-score}) on the \(50\)-image test set.

\subsection{Results}

\begin{figure}[t]
    \centering
    \includegraphics[width=\linewidth]{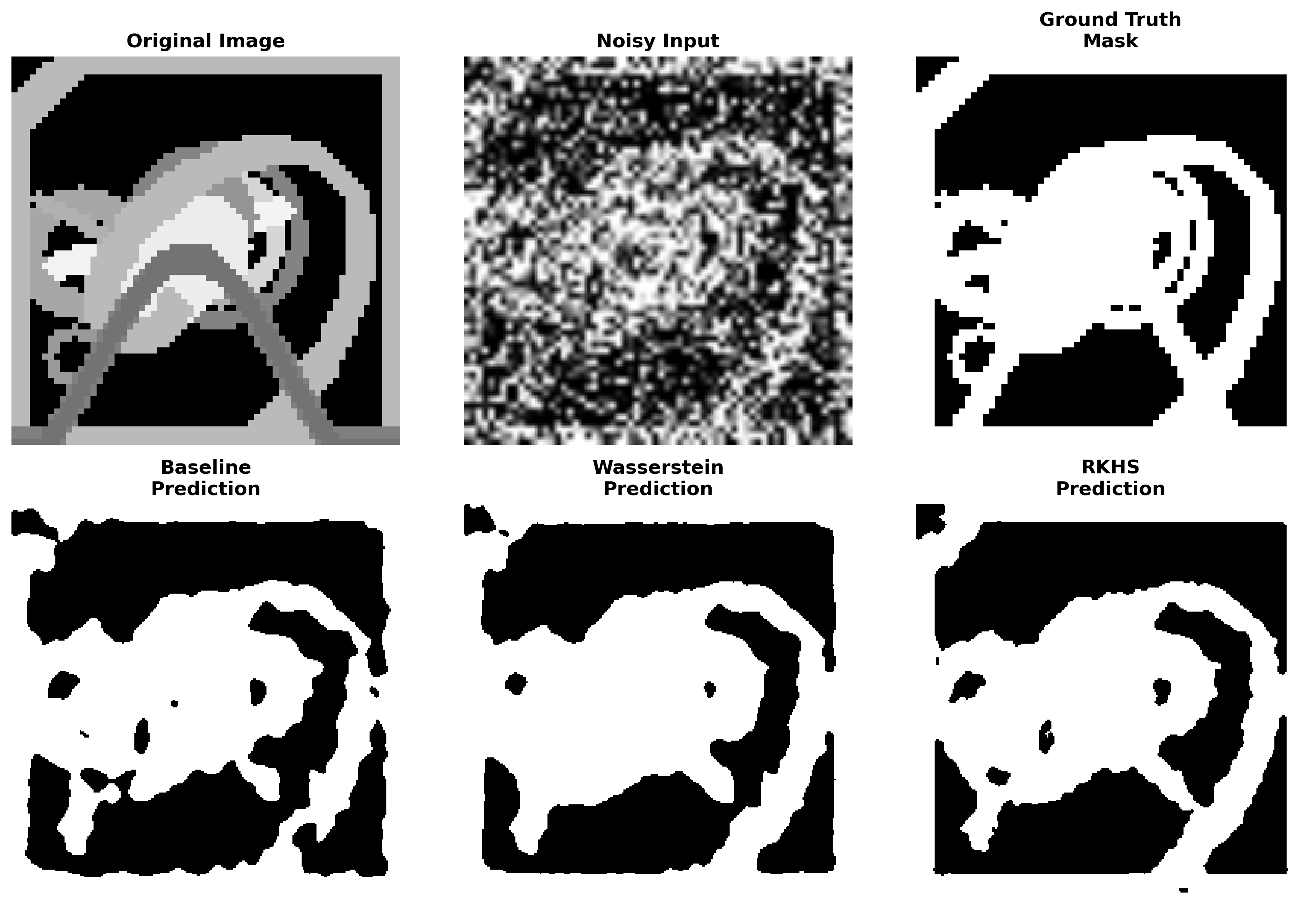}
    \caption{Synthetic testing data showing the original image, noisy network input, ground--truth mask, and the corresponding predictions from the Dice baseline, Wasserstein loss, and RKHS loss.}
    \label{fig:synthetic-qualitative}
\end{figure}

Figure~\ref{fig:synthetic-qualitative} shows an unseen test instance in which the Dice loss (Equation~\eqref{eq:dice-loss}) fails to recover the correct global topology, the Wasserstein loss partially corrects these structural errors but still distorts the underlying shape, and the RKHS loss best preserves the ground-truth topology, in line with the quantitative trends in Table~\ref{tab:three-way-results}.  

\begin{table}[t]
  \centering
  \caption{Three--way comparison on the high--noise synthetic test set. The IoU and Dice columns give mean test-set scores. The final two columns report relative improvements in mean IoU over the Dice baseline and Wasserstein models, respectively.}
  \label{tab:three-way-results}
  \begin{tabular}{lcccc}
    \toprule
    Model        & IoU    & Dice   & vs.\ Baseline & vs.\ Wasserstein \\
    \midrule
    Baseline     & 0.8485 & 0.9176 &         &       \\
    Wasserstein  & 0.8636 & 0.9264 & +1.79\% &       \\
    RKHS         & 0.8831 & 0.9377 & +4.08\% & +2.25\% \\
    \bottomrule
  \end{tabular}
\end{table}

Wasserstein improves mean IoU by $+0.0151$ relative to the baseline, while the RKHS loss improves mean IoU by $+0.0346$ and exceeds Wasserstein by an additional $+0.0195$.

\section{Conclusion}\label{sec:conclusion}

We have developed a heat-kernel RKHS on the virtual persistence diagram group $K(X,A)$ associated to a finite metric pair $(X,d,A)$, and shown that the resulting feature map is globally $W_1$-Lipschitz. The construction is organized spectrally by the Laplacian symbol $\lambda(\theta)$ on the dual torus, with heat multipliers $e^{-t\lambda(\theta)}$ providing an explicit scale for suppressing unstable phase variation through the parameter $t$.  This yields a family of Wasserstein-stable kernels together with finite-dimensional random Fourier feature approximations satisfying the same probabilistic Lipschitz bounds, thereby producing computable Euclidean embeddings of virtual persistence diagrams. In this sense, the paper provides an explicit spectral kernel framework for learning on virtual persistence diagrams.

In a noisy, topology-rich segmentation setting, the RKHS loss outperforms a Wasserstein topological loss: it more reliably recovers the global topology of the target and yields the largest gains precisely on images where the Dice baseline is most topologically distorted. This is consistent with the analytic picture: when the virtual persistence diagram difference $\gamma$ is far from $0$ in the transport metric, the heat-kernel embedding produces a globally Lipschitz map that corrects large topological discrepancies, whereas the Wasserstein loss does not incorporate heat-weighted phase stabilization. The experiments provide a concrete illustration of this mechanism.

The main limitations are that the theory is developed for finite-rank virtual persistence diagrams and that the Monte Carlo implementation of the heat kernel introduces stochastic variability. The stability-resolution tradeoff is governed by the heat parameter \(t\), and the empirical study is deliberately narrow, intended only to demonstrate the analytic mechanism rather than to serve as a broad benchmark.

Future directions include extending the harmonic analysis to infinite-rank virtual persistence diagrams and their Cauchy completions, identifying Sobolev-scale regularity properties of the heat embedding on the Pontryagin dual torus, and studying symmetric group actions on the virtual persistence diagram group.

\section*{Declarations}

\begin{itemize}

\item \textbf{Competing Interests} The authors declare that they have no competing interests.

\item \textbf{Funding} This research received no external funding.

\item \textbf{Data Availability} Not applicable.

\item \textbf{Code Availability} The implementation used in this work is available at \url{https://github.com/cfanning8/Virtual_Persistence_RKHS}.

\item \textbf{AI Usage} The authors used an AI language model to rephrase dense analytic descriptions and to proofread and correct the language and grammar during manuscript preparation. The authors assume responsibility for all content in this publication.

\item \textbf{Authors' Contributions}
C.F. developed the theoretical framework, conducted the experiments, and wrote the manuscript. M.E.A. supervised the project and provided critical feedback on the framework, experiments, and manuscript.

\end{itemize}


\bibliography{sn-bibliography}

\end{document}